\UseRawInputEncoding
\documentclass{article}
\usepackage[top=3cm,bottom=3cm,left=3cm,right=2.5cm]{geometry}
\usepackage{amsfonts}\usepackage{amssymb}
\usepackage{graphicx}\usepackage{amsmath}
\usepackage{subfigure}
\usepackage{abstract}
\usepackage{color}
\usepackage{cite}
\usepackage{graphics}
\usepackage{epstopdf}
\usepackage{amsthm}
\usepackage{float}
\usepackage{mathrsfs}
\usepackage{extarrows}
\usepackage{booktabs}
\usepackage{amssymb, amsmath, amscd, latexsym}

\newtheorem{theorem}{Theorem}[section]
\newtheorem{lemma}[theorem]{Lemma}

\newtheorem{remark}[theorem]{Remark}

\numberwithin{equation}{section}
\providecommand{\keywords}[1]{\textbf{Key words.} #1}

\title{A locking-free discontinuous Galerkin method for linear elastic Steklov eigenvalue problem}

\author{Yanjun Li$^{1,2}$,  Hai Bi$^{1,}$\footnote{Corresponding author. Email address:  bihaimath@gznu.edu.cn}
 \\\\
{\small $^1$ School of Mathematical Sciences,
Guizhou Normal University,Guiyang,  $550001$,  China}\\{\small $^2$ School of Big Data Applications and Economics, Guizhou University of Finance and Economics,}\\ {\small Guiyang, 550025, China}
}

\begin{document}
\date{}
\maketitle
\begin{abstract}
In this paper, a discontinuous Galerkin finite element method of Nitsche's version for the Steklov eigenvalue problem in linear elasticity is presented.  The a priori error estimates are analyzed under a low regularity condition, and the robustness with respect to nearly incompressible materials (locking-free) is proven. Furthermore, some numerical experiments are reported to show the effectiveness and robustness of the proposed  method.
\end{abstract}


\section{Introduction}
\indent In structural analysis and engineering design, linear elasticity problems are extensively
encountered when studying how solid objects deform and become internally stressed on prescribed loading conditions. Due to the wide range of applications, the approximate computation for elastic equations/eigenvalue problems has attracted the attention of academic circles, for instance, \cite{Brenner1992,Hernandez2009,Meddahi2013,oden,Ovtchinnikov,Russo,Walsh}, etc.
Recently, \cite{Dom¨ªnguez2021} introduced a Steklov-Lam\'{e} eigenproblem in linear elasticity. The study of the Dirichlet-to-Neumann map for linear elasticity is important in elastostatic problem. \cite{Dom¨ªnguez2021} studied the existence of eigenpairs of this problem and explored its conforming finite element approximations.
However, it is well known that the numerical approximations of elasticity problems by conforming finite element methods may significantly degrade as the Lam\'{e} dilatation parameter approaches a certain critical limit. This non-robustness of the finite element method is called ``locking''. To avoid the locking effects, several approaches have been investigated such as nonconforming finite element methods in \cite{Brenner1992,Kouhia}, mixed finite element methods in \cite{Brezzi1991,Chapelle}, higher-order methods in \cite{Vogelius}, discontinuous Galerkin finite element method in \cite{Hansbo2002,Wihler2004,Wihler2006} and so on.
In this paper, we discuss a discontinuous Galerkin finite element method (DGFEM) of Nitsche's version for the Steklov-Lam\'{e} eigenproblem in linear elasticity.

The DGFEM have been attractive due to their flexibility in handling general meshes, non-uniformity in degree of approximation and capturing the rough solutions more accurately. Hence, the DGFEM has been developed and applied to solve various problems, for example, elliptic problems in \cite{Arnold2002,Douglas1976,Rusten1996,Riviere2008,Wihler2002}, hyperbolic problems in
\cite{Antonietti2022,Antonietti2018,Johnson1986,Brezzi2004,Johnson1993,Krivodonova2004}, Navier-Stokes equations in \cite{Bassi1997,Pietro2009}, convection-diffusion systems in \cite{Cockburn1998b,Ern2005}, etc. In particular, since
the DGFEM was found to be able to overcome the locking phenomenon (see \cite{Hansbo2002,Wihler2004,Wihler2006,Pietro2013}), this method has been studied intensively for elasticity problems.

For the Steklov-Lam\'{e} eigenproblem in linear elasticity, \cite{Dom¨ªnguez2021} gave the a priori error estimates of conforming
finite element approximations under the condition that the eigenfunctions have high smoothness. For the case that the eigenfunctions have low smoothness or regularity, we have not yet seen the relevant numerical analysis.
In this paper, for the Steklov-Lam\'{e} eigenproblem, due to the features of DG finite element space, to derive the error estimate of eigenvalues we first need to prove that the discrete solution operator converges in $L^2(\partial\Omega)$ norm, and it needs us to derive the a prior error estimates of the corresponding source problem with right-hand side belongs to $\boldsymbol{L}^2(\partial\Omega)$. In this case, the solution $\boldsymbol{u}\in \boldsymbol{H}^{1+\frac{1}{2}}(\Omega)$, that is to say, we shall establish the a priori error estimates of DG discretization under the condition of low regularity.
In the analysis we first derive the a priori error estimates, Theorems 3.7 and 3.8, of the source problem associated with the Steklov-Lam\'{e} eigenproblem when the loading function has low regularity. We make clear the dependence of the Lam\'{e} parameters in each step of the proof  and get the estimates in which the constants are clearly independent of the Lam\'{e} parameters, which means our method is locking-free.
Then we prove that $T_h$, the discrete solution operator, converges to $T$, the generalized solution operator, in the sense of $ L^2(\partial \Omega)$ norm, thus, by using Babu\v{s}ka-Osborn spectral approximation theory \cite{Babuska1991} we obtain the error estimation of eigenvalues and eigenfunctions.

To support our theoretical analysis, we exhibit a large set of numerical experiments.
We choose four domains including 2$D$ and 3$D$ cases to test the convergence and robustness of the DGFEM for the Steklov eigenvalues when different values of Lam\'{e} parameter $\lambda$ are considered. Numerical results show that DGFEM is efficient for solving the Steklov-Lam\'{e} eigenproblem, and the approximations obtained by DGFEM converge uniformly with respect of $\lambda$, which is consistent with our expectations.

The remainder of the paper is organized as follows. The Steklov-Lam\'{e} eigenproblem and its DG discrete formulation are described in Section 2. The a priori error estimates of DGFEM for the source problem and the eigenproblem are presented in Section 3. Finally, in Section 4 some numerical experiments are provided.

Throughout this paper, we use the following notations. Vector fields will be denoted by bold symbols whereas tensor fields are denoted by bold Greek letters. The notation $\boldsymbol{a} \cdot \boldsymbol{b}$ is the standard dot product with induced norm $\|\cdot\|$. For tensors $\boldsymbol{\sigma}, \boldsymbol{\tau} \in \mathbb{R}^{d \times d}$, the notation $\boldsymbol{\sigma}: \boldsymbol{\tau}:=\operatorname{tr}\left(\boldsymbol{\tau}^{\mathrm{T}}\boldsymbol{\sigma}\right)$  where tr($\cdot )$ denotes the trace of a tensor (sum of the main diagonal). This inner product induces the Frobenius norm for tensors, also denoted as $\|\cdot\|$.

Let $H^{s}(\Omega)$ denote the usual Sobolev space of scalar fields with  $s\in \mathbb{R}$  on $\Omega$ and $H^{0}(\Omega)=L^2(\Omega)$, $\|\cdot\|_{s,\Omega}$ is the norm on $H^{s}(\Omega)$, and we just write $\|\cdot\|_{s}$ for simplicity. Let $H^{s}(\partial\Omega)$ denote the usual Sobolev space of scalar fields with order $s$ on $\partial\Omega$ with the norm $\|\cdot\|_{s,\partial\Omega}$. For Sobolev space of vector fields, we use the notation $\boldsymbol{H}^{s}(\Omega)$ and $\boldsymbol{H}^{s}(\partial\Omega)$ with the corresponding norm, also denoted by
$\|\cdot\|_{s}$ and $\|\cdot\|_{s,\partial\Omega}$, 
respectively.

Throughout this paper, we use the letter $C$ to denote a generic positive constant independent of the Lam\'{e} parameters $\mu$, $\lambda$ and the mesh diameter which may take different values in different contexts, and write $a \lesssim b$ when $a \leq C b$ and $a \gtrsim b$ when $a \geq C b$ for some positive constant $C$ for simplicity.

\section{The Steklov-Lam\'{e} eigenproblem and its DG approximation}\label{sec2}
In this paper, we assume that an isotropic and linearly elastic material occupies the domain $\Omega$ in $ \mathbb{R}^{d}~(d=2,3)$ and
$\Omega$ is a bounded convex polygonal or a smooth domain.
Consider the following Steklov-Lam\'{e} eigenproblem (see \cite{Dom¨ªnguez2021}): find non-zero displacement vector $\boldsymbol{u}$ and the frequencies $\omega\in \mathbb{R}$ satisfying
\begin{equation}\label{eq1}
\begin{cases}
-\mathrm{\mathrm{div}} \boldsymbol{\sigma}(\boldsymbol{u}) = \boldsymbol{0},~~ \rm{in} ~\Omega,\\
\boldsymbol{\sigma}(\boldsymbol{u})\boldsymbol{n}=\omega p\boldsymbol{u}, ~~\rm{on}~\partial\Omega,
\end{cases}
\end{equation}
where $\boldsymbol{n}$ is the outer unit normal vector on the boundary $\partial \Omega$, $\boldsymbol{\sigma(u)}$ is the Cauchy stress tensor given by the generalized Hooke law
\begin{align*}
\boldsymbol{\sigma}(\boldsymbol{u}) = 2\mu\boldsymbol{\varepsilon}(\boldsymbol{u})+\lambda tr(\boldsymbol{\varepsilon}(\boldsymbol{u}))\boldsymbol{I}
=2\mu\boldsymbol{\varepsilon}(\boldsymbol{u})+\lambda (\mathrm{div}\boldsymbol{u})\boldsymbol{I},
\end{align*}
where $\boldsymbol{I}\in \mathbb{R}^{d\times d}$ is the identity matrix, $\boldsymbol{\varepsilon}(\boldsymbol{u})$ is the strain tensor given by

\begin{align*}
\boldsymbol{\varepsilon}(\boldsymbol{u})=\frac{1}{2}(\nabla \boldsymbol{u}+(\nabla \boldsymbol{u})^{T})
\end{align*}
and $\nabla \boldsymbol{u}$ is the displacement gradient tensor.
The parameters $\lambda$ and $\mu$ are called the Lam\'{e} parameters satisfying $0<\mu_{1}<\mu<\mu_{2}$ and $0<\lambda<\infty$,
while the density of material $p\in L^{\infty}(\partial\Omega)$ is assumed to be strictly positive on $\partial \Omega$.

It is easy to know that this problem has zero eigenvalues, as suggested in \cite{Babuska1991}, and a shift needs to be added to the formulation to obtain a coercive bilinear form. We employ the following weak formulation for (\ref{eq1}): find $(\rho ,\boldsymbol{u})\in \mathbb{R}\times \boldsymbol{H}^{1}(\Omega)$ such that
\begin{align}\label{eq2}
a(\boldsymbol{u},\boldsymbol{v})= \rho b(\boldsymbol{u},\boldsymbol{v}),~ \forall \boldsymbol{v}\in  \boldsymbol{H}^{1}(\Omega),
\end{align}
where $\rho=\omega+1$, and the bilinear forms $a(\cdot,\cdot)$ and $b(\cdot,\cdot)$ are defined as follows, respectively:
\begin{align*}
&a(\boldsymbol{u},\boldsymbol{v}):= \int_{\Omega} \boldsymbol{\sigma}(\boldsymbol{u}):\boldsymbol{\varepsilon}(\boldsymbol{v})dx+\int_{\partial\Omega} p\boldsymbol{u}\cdot\boldsymbol{v}ds,\\
&\quad\quad\quad=2\mu\int_{\Omega}\boldsymbol{\varepsilon}(\boldsymbol{u}):\boldsymbol{\varepsilon}(\boldsymbol{v})dx+\lambda \int_{\Omega}(\mathrm{div}\boldsymbol{u})(\mathrm{div}\boldsymbol{v})dx
         +\int_{\partial\Omega} p\boldsymbol{u}\cdot\boldsymbol{v}ds,\quad\forall \boldsymbol{u}, \boldsymbol{v}\in  \boldsymbol{H}^{1}(\Omega),\\
&b(\boldsymbol{u},\boldsymbol{v}):=\int_{\partial\Omega} p\boldsymbol{u}\cdot\boldsymbol{v}ds,\quad\forall \boldsymbol{u}, \boldsymbol{v}\in  \boldsymbol{H}^{1}(\Omega).
\end{align*}
Without loss of generality, we assume that $p\equiv1$ in the rest of the paper.

We now specify some notations for the spatial discretization. Let ${\mathcal{T}_{h}}=\{K\}$ be a shape-regular mesh (see \cite{Ciarlet1991,Ern2004}) of $\Omega$ where $K$ is a triangle if $d=2$ or a tetrahedron if $d=3$. For each $ K\in\mathcal{T}_{h}$, we denote by $\boldsymbol{n}_{K}$ the unit outward normal vector on the boundary $\partial K$.
Let $\mathcal{E}_h$ denote the set of all edges/faces in the mesh. We decompose $\mathcal{E}_h$
into disjoint sets $\mathcal{E}_h^i$ and $\mathcal{E}_h^b$ which consists of interior edges/faces and those on the boundary, respectively.
For any face $e\in \mathcal{E}_{h} $, the diameter of $e$  is denoted by $h_{e}$, the diameter of a cell~$K\in\mathcal{T}_{h}$ is denoted by $h_{K}$ and $h=\max_{K\in \mathcal{T}_{h}}\{h_{K}\}$. We also require that the mesh satisfies the condition $h_e^{-1}\lesssim h_K^{-1}$. Denote the average and jump of $\boldsymbol{v}$ on $e$ by
\begin{equation*}
\{\boldsymbol{v}\}=\begin{cases}
\frac{1}{2}(\boldsymbol{v}^{+}+\boldsymbol{v}^{-}),~~ e\in ~\mathcal{E}_{h}^{i},\\
\boldsymbol{v}^{+},~~~~~~~~~~~~~~ e \in~\mathcal{E}_{h}^{b},
\end{cases}
[\![\boldsymbol{v}]\!]=\begin{cases}
\boldsymbol{v}^{+}-\boldsymbol{v}^{-},~~ e\in ~\mathcal{E}_{h}^{i},\\
\boldsymbol{v}^{+},~~~~~~~~~~ e\in~\mathcal{E}_{h}^{b},
\end{cases}
\end{equation*}
 where
 $\boldsymbol{v}^{+}=\boldsymbol{v}|_{K^{+}},\boldsymbol{v}^{-}=\boldsymbol{v}|_{K^{-}}$.

 Define the DG element space:
 \begin{align*}
  \boldsymbol{S}^{h}=\{\boldsymbol{v}\in L^{2}(\Omega)^{d}: \boldsymbol{v}|_{K}\in P_{k}(K)^{d},~\forall K\in\mathcal{T}_{h}\},
 \end{align*}
 where  $P_{k}(K)$ is the polynomial space of degree no more than $k~ (k\geq 1)$ on $K$. Introduce the piecewise 
 Sobolev functions space of degree $s$:
 \begin{align*}
 \boldsymbol{H}^{s}(\mathcal{T}_{h})=\{\boldsymbol{v}\in L^{2}(\Omega)^{d}: \boldsymbol{v}|_{K}\in H^{s}(K)^{d},~\forall K\in\mathcal{T}_{h}\}.
 \end{align*}

The DGFEM includes three common types: the symmetric interior penalty Galerkin method (SIPG) which comes from Nitsche method \cite{Nitche1971}, the non-symmetric interior penalty Galerkin method (NIPG) (cf. see \cite{Oden1998,Wihler2004}), and the incomplete interior penalty Galerkin method (IIPG) (see, e.g., \cite{Ortner2007}).
 In this paper, we discuss the SIPG method for (\ref{eq2}). Define the bilinear forms $a_{h}(\cdot,\cdot)$ and $b_{h}(\cdot,\cdot)$ as follows:
 \begin{align*}
\nonumber a_{h}(\boldsymbol{u_{h}},\boldsymbol{v_{h}})&=2\mu\left( \sum\limits_{K\in\mathcal{T}_{h}}\int_{K}\boldsymbol{\varepsilon}(\boldsymbol{u_{h}}):\boldsymbol{\varepsilon}(\boldsymbol{v_{h}})dx
-\sum_{e\in \mathcal{E}^{i}_{h}}\int_{e}\{\boldsymbol{\varepsilon}(\boldsymbol{u_{h}})\boldsymbol{n}\}\cdot[\![\boldsymbol{v_{h}}]\!]ds \right.\\
\nonumber&\left.
-\sum_{e\in \mathcal{E}^{i}_{h}}\int_{e}\{\boldsymbol{\varepsilon}(\boldsymbol{v_{h}})\boldsymbol{n}\}\cdot[\![\boldsymbol{u_{h}}]\!]ds
+\sum_{e\in \mathcal{E}^{i}_{h}}\frac{\gamma_{\mu} }{h_{e}} \int_{e}[\![\boldsymbol{u_{h}}]\!]\cdot[\![\boldsymbol{v_{h}}]\!]ds
\right)
 \end{align*}
 \begin{align*}
\nonumber&+\lambda\left(
\sum \limits_{K\in\mathcal{T}_{h}}\int_{K}(\mathrm{div}\boldsymbol{u_{h}})(\mathrm{div}\boldsymbol{v_{h}})dx
-\sum_{e\in \mathcal{E}^{i}_{h}}\int_{e}\{\mathrm{div}\boldsymbol{u_{h}}\}[\![\boldsymbol{v_{h}}\cdot\boldsymbol{n}]\!]ds\right.\\
\nonumber&\left.
-\sum_{e\in \mathcal{E}^{i}_{h}}\int_{e}\{\mathrm{div}\boldsymbol{v_{h}}\} [\![\boldsymbol{u_{h}}\cdot\boldsymbol{n}]\!]ds
+ \sum_{e\in \mathcal{E}^{i}_{h}}\frac{\gamma_{\lambda}{\color{red} }}{h_{e}} \int_{e}[\![\boldsymbol{u_{h}}\cdot\boldsymbol{n}]\!][\![\boldsymbol{v_{h}}\cdot\boldsymbol{n}]\!]ds\right)
+\sum_{e\in \mathcal{E}^{b}_{h}}\int_{e}\boldsymbol{u_{h}}\cdot\boldsymbol{v_{h}}ds,\\
b_{h}(\boldsymbol{u_{h}},\boldsymbol{v_{h}})&=\sum \limits_{e\in\mathcal{E}^{b}_{h}}\int_{e}\boldsymbol{u_{h}}\cdot\boldsymbol{v_{h}}ds.
\end{align*}

\noindent Then the DG finite element approximation for (\ref{eq2}) states as: find $(\rho_{h},\boldsymbol{u_{h}})\in \mathbb{R} \times\boldsymbol{S}^{h},\boldsymbol{u_{h}}\neq\boldsymbol{0},\rho_{h}=\omega_{h}+1$, such that
\begin{align}\label{eq3}
a_{h}(\boldsymbol{u_{h}},\boldsymbol{v_{h}})=\rho_{h}b_{h}(\boldsymbol{u_{h}},\boldsymbol{v_{h}}), \quad\forall \boldsymbol{v_{h}}\in \boldsymbol{S}^{h}.
\end{align}
Define the DG norm:
 \begin{align}\label{eq4}
\nonumber \|\boldsymbol{u_{h}}\|^{2}_{\mathrm{dG}}
&=2\mu\sum_{K\in {\mathcal{T}_{h}}}\|\boldsymbol{\varepsilon}(\boldsymbol{u_{h}})\|^{2}_{0,K}+ 2\mu\sum_{e\in \mathcal{E}^{i}_{h}}\gamma_{\mu} h^{-1}_{e}\|[\![\boldsymbol{u_{h}}]\!]\|^{2}_{0,e}+\lambda \sum_{K\in {\mathcal{T}_{h}}}\|\mathrm{div}\boldsymbol{u_{h}}\|^{2}_{0,K}\\
&\quad+\lambda\sum_{e\in \mathcal{E}^{i}_{h}}\gamma _{\lambda} h^{-1}_{e}\|[\![\boldsymbol{u_{h}}\cdot \boldsymbol{n}]\!]\|^{2}_{0,e}+\sum_{e\in \mathcal{E}^{b}_{h}}\|\boldsymbol{u_{h}}\|^{2}_{0,e},
\end{align}
and the energy-like norm:
\begin{align}\label{eq5}
&\|\boldsymbol{u_{h}}\|^{2}_{h}=\|\boldsymbol{u_{h}}\|^{2}_{\mathrm{dG}}+2\mu\sum_{e\in \mathcal{E}^{i}_{h}}h_{e}\|\{\boldsymbol{\varepsilon}(\boldsymbol{u_{h}})\boldsymbol{n}\}\|^{2}_{0,e}
+\lambda\sum_{e\in \mathcal{E}^{i}_{h}}h_{e}\|\{\mathrm{div}\boldsymbol{u_{h}}\}\|^{2}_{0,e}.
\end{align}

In order to show that the discretization (\ref{eq3}) is stable, first we will show that $a_{h}(\cdot, \cdot)$ is coercive on $\boldsymbol{S}^{h}\times \boldsymbol{S}^{h}$. To do so, we need the following inverse inequalities which can be referred to Lemma 4 in \cite{Hansbo2002}.
\begin{lemma}\label{lemma2.1}
For $\boldsymbol{v} \in \boldsymbol{S}^{h}$, there are constants $C_{\mu}$ and $C_{\lambda}$, independent of the diameter $h_{e}$, $\mu$ and $\lambda$, such that
\begin{align}\label{eq6}
\|h_{e}^{\frac{1}{2}} \boldsymbol{\varepsilon}(\boldsymbol{v}) \boldsymbol{n}\|_{0,\partial K}^{2} \leq C_{\mu}\|\boldsymbol{\varepsilon}(\boldsymbol{v})\|_{0,K}^{2},\\
\label{eq7}
\|h_{e}^{\frac{1}{2}} \mathrm{div} \boldsymbol{v}\|_{0,\partial K}^{2} \leq C_{\lambda}\|\mathrm{div} \boldsymbol{v}\|_{0,K}^{2} .
\end{align}
\end{lemma}
Referring to Proposition 6 in \cite{Hansbo2002}, we have the following coercive property.
\begin{lemma}
If $\gamma_{\mu;\lambda}\geq C_{\mu;\lambda}/(1-m)^{2}$ for $0<m<1$, then the coerciveness of $a_{h}(\cdot,\cdot)$ holds:
\begin{align}\label{eq8}
m\|\boldsymbol{v_{h}}\|_{\mathrm{dG}}^{2}\leq  a_{h}(\boldsymbol{v_{h}},\boldsymbol{v_{h}}),\quad for~all~\boldsymbol{v_{h}}\in \boldsymbol{S}^{h}
\end{align}
independent of $h$.
\end{lemma}
\begin{proof}
For any $\boldsymbol{v_{h}}\in \boldsymbol{S}^{h}$, we have
\begin{align}\label{eq9}
\nonumber &a_{h}(\boldsymbol{v_{h}},\boldsymbol{v_{h}})
= 2\mu \left ( \sum \limits_{K\in\mathcal{T}_{h}} \|\boldsymbol{\varepsilon}(\boldsymbol{v_{h}})\|^{2}_{0,K}- 2\sum_{e\in \mathcal{E}^{i}_{h}}\int_{e}\{\boldsymbol{\varepsilon}(\boldsymbol{v_{h}})\boldsymbol{n}\}\cdot[\![\boldsymbol{v_{h}}]\!]ds
+\sum_{e\in \mathcal{E}^{i}_{h}}\gamma_{\mu} \|h_{e}^{-\frac{1}{2}}[\![\boldsymbol{v_{h}}]\!]\|^{2}_{0,e}\right)\\
\nonumber&+\lambda\left(\sum \limits_{K\in\mathcal{T}_{h}} \|\mathrm{div} \boldsymbol{v_{h}} \|^{2}_{0,K}-2\sum_{e\in \mathcal{E}^{i}_{h}}\int_{e}\{\mathrm{div} \boldsymbol{v_{h}} \}[\![\boldsymbol{v_{h}} \cdot \boldsymbol{n}]\!]ds
 +\sum_{e\in \mathcal{E}^{i}_{h}}\int_{e}\gamma_{\lambda}\|h_{e}^{-\frac{1}{2}}[\![\boldsymbol{v_{h}} \cdot \boldsymbol{n}]\!]\|^{2}_{0,e}\right)\\
\nonumber&+\sum_{e\in \mathcal{E}^{b}_{h}}\|\boldsymbol{v_{h}}\|^{2}_{0,e}\\
&\equiv I+II+III.
\end{align}
To show the coercivity we need to bound the potentially negative terms by the positive terms. We first estimate the second term in $I$. Using the Cauchy-Schwarz inequality, the inverse inequality (\ref{eq6}) and Young's inequality, we derive that
\begin{align*}
 &\sum_{e\in \mathcal{E}^{i}_{h}}\int_{e}\{\boldsymbol{\varepsilon}(\boldsymbol{v_{h}})\boldsymbol{n}\}\cdot[\![\boldsymbol{v_{h}}]\!]ds
=\sum_{K\in \mathcal{T}_{h}} \int_{\partial K \backslash \partial \Omega}
 \{\boldsymbol{\varepsilon}(\boldsymbol{v_{h}})\boldsymbol{n}\}\cdot[\![\boldsymbol{v_{h}}]\!]ds\\
&  \leq  \sum_{K\in \mathcal{T}_{h}}\|h_{e}^{\frac{1}{2}}\{\boldsymbol{\varepsilon}(\boldsymbol{v_{h}})\boldsymbol{n}\}\|_{0,\partial K\backslash \partial\Omega}\|h_{e}^{-\frac{1}{2}}[\![\boldsymbol{v_{h}}]\!]\|_{0,\partial K\backslash \partial\Omega}
\lesssim
  \sum_{K\in \mathcal{T}_{h}}\|h_{e}^{\frac{1}{2}}\boldsymbol{\varepsilon}(\boldsymbol{v_{h}})\boldsymbol{n}\|_{0,\partial K\backslash \partial\Omega}\|h_{e}^{-\frac{1}{2}}[\![\boldsymbol{v_{h}}]\!]\|_{0,\partial K\backslash \partial\Omega}\\
&  \lesssim   \sum_{K\in \mathcal{T}_{h}} \sqrt{C_{\mu}}\|\boldsymbol{\varepsilon}(\boldsymbol{v_{h}})\|_{0,K}\|h_{e}^{-\frac{1}{2}}[\![\boldsymbol{v_{h}}]\!]\|_{0,\partial K\backslash \partial\Omega}
 \lesssim \sum\limits_{K\in\mathcal{T}_{h}}\left (\epsilon_{\mu}C_{\mu}\|\boldsymbol{\varepsilon}(\boldsymbol{v_{h}})\|^{2}_{0,K}
 +\frac{1}{4\epsilon_{\mu} }\|h_{e}^{-\frac{1}{2}}[\![\boldsymbol{v_{h}}]\!]\|^{2}_{0,\partial K\backslash \partial\Omega}\right),
\end{align*}
and taking $\epsilon_{\mu} =(1-m)/2C_{\mu}$ and $\gamma_{\mu}\geq m\gamma_{\mu}+1/(2\epsilon_{\mu}) =C_{\mu}/(1-m)^2$ we have
\begin{align*}
\nonumber I&\geq 2\mu\sum \limits_{K\in\mathcal{T}_{h}}\left( (1-2\epsilon_{\mu}C_{\mu})\|\boldsymbol{\varepsilon}(\boldsymbol{v_{h}})\|^{2}_{0,K}
+(\gamma_{\mu}-\frac{1}{2\epsilon_{\mu}}) \|h_{e}^{-\frac{1}{2}}[\![\boldsymbol{v_{h}}]\!]\|^{2}_{0,\partial K\backslash \partial\Omega}\right)\\
&\geq m2\mu\sum \limits_{K\in\mathcal{T}_{h}}\left( \|\boldsymbol{\varepsilon}(\boldsymbol{v_{h}})\|^{2}_{0,K}
 + \gamma_{\mu}\|h_{e}^{-\frac{1}{2}}[\![\boldsymbol{v_{h}}]\!]\|^{2}_{0,\partial K\backslash \partial\Omega}\right).
\end{align*}
With the same argument as above, using (\ref{eq7}) and taking $\epsilon_{\lambda} =(1-m)/2C_{\lambda}$ and $\gamma_{\lambda}\geq C_{\lambda}/(1-m)^{2}$, we obtain
\begin{align*}
 II\geq m\lambda \sum \limits_{K\in\mathcal{T}_{h}}  \left( \|\mathrm{div} \boldsymbol{v_{h}} \|^{2}_{0,K}
 + \gamma_{\lambda}\|h_{e}^{-\frac{1}{2}}[\![\boldsymbol{v_{h}} \cdot \boldsymbol{n}]\!]\|^{2}_{0,\partial K\backslash \partial\Omega}\right).
\end{align*}
Since $m\in (0,1)$, it is clear that
\begin{align*}
III\geq m\sum_{e\in \mathcal{E}^{b}_{h}}\|\boldsymbol{v_{h}}\|^{2}_{0,e}.
\end{align*}
Together with the above three inequalities and (\ref{eq9}), the coerciveness of $a_{h}(\cdot,\cdot)$ is valid.
\end{proof}

\indent By using the Cauchy-Schwarz inequality, it is easy to show the following continuity property.

\begin{lemma}
The bilinear form $a_{h}(\boldsymbol{u},\boldsymbol{v})$ is continuous:
\begin{align}\label{eq10}
\mid a_{h}(\boldsymbol{u},\boldsymbol{v}) \mid \lesssim \|\boldsymbol{u}\|_{h}\|\boldsymbol{v}\|_{h},\quad \forall \boldsymbol{u},\boldsymbol{v}\in \boldsymbol{H}^{1+s}(\mathcal{T}_{h}),~s>\frac{1}{2}.
\end{align}
\end{lemma}

\section{A priori error estimate}\label{sec3}
From Babu\v{s}ka-Osborn spectral approximation theory \cite{Babuska1991} we know that the convergence and the error estimates of the finite element method for an eigenvalue problem can be derived from the convergence and the error estimates of the finite element method for the corresponding source problem. Hence, in this section we first extend the work of \cite{Hansbo2002} to get the error estimates for the source problem associated with the Steklov-Lam\'{e} eigenproblem.

The source problem associated with (\ref{eq1}) is as follows: find $\boldsymbol{w}\in \boldsymbol{H}^{1}(\Omega)$, such that
\begin{align}\label{eq11}
a(\boldsymbol{w,v})= b(\boldsymbol{f},\boldsymbol{v}), \quad \forall \boldsymbol{v} \in \boldsymbol{H}^{1}(\Omega).
\end{align}

Thanks to \cite{Brenner1992,Hansbo2002}, we have the following regularity lemma.

\begin{lemma}\label{lemma3.1}
Let $\boldsymbol{w}$ be the solution of (\ref{eq11}), then the following regularity estimates hold:\\

\noindent1. If $\boldsymbol{f}\in \boldsymbol{H}^{r-\frac{1}{2}}(\partial\Omega)$, then $\boldsymbol{w}\in \boldsymbol{H}^{r+1}(\Omega)$ and
\begin{align}\label{eq12}
\|\boldsymbol{w}\|_{r+1}+\lambda \|\mathrm{div}\boldsymbol{w} \|_{r}\leq C_{1}\|\boldsymbol{f}\|_{r-\frac{1}{2},\partial\Omega},
\end{align}
where $r=1$ when $\Omega$ is a convex polygonal and $r$ can be large enough when $\partial\Omega$ is smooth enough; \\
2. If $\boldsymbol{f}\in \boldsymbol{H}^{-\frac{1}{2}}(\partial\Omega)$, then $\boldsymbol{w}\in \boldsymbol{H}^{1}(\Omega)$ and
\begin{align}\label{eq13}
\|\boldsymbol{w}\|_{1}+\lambda \|\mathrm{div}\boldsymbol{w} \|_{0}\leq C_{2}\|\boldsymbol{f}\|_{-\frac{1}{2},\partial\Omega};
\end{align}
3. If $\boldsymbol{f}\in \boldsymbol{L}^{2}(\partial\Omega)$, then $\boldsymbol{w}\in \boldsymbol{H}^{1+\frac{1}{2}}(\Omega)$ and
\begin{align}\label{eq14}
 \|\boldsymbol{w}\|_{1+\frac{1}{2}}+\lambda \|\mathrm{div}\boldsymbol{w} \|_{\frac{1}{2}}\leq C_{3}\|\boldsymbol{f}\|_{0,\partial\Omega}.
\end{align}
\nonumber Here the constants $C_{i} (i=1,2,3)$ are independent of $\mu$ and $\lambda$.
\end{lemma}
\begin{proof} Referring to (2.30) in \cite{Brenner1992} and (28) in \cite{Hansbo2002}, for integer $r\geq 1$ we can get (\ref{eq12}).

From (\ref{eq11}) and the definition of $a(\cdot,\cdot)$ we have
\begin{align}\label{eq15}
2\mu\int_{\Omega}\boldsymbol{\varepsilon}(\boldsymbol{w}):\boldsymbol{\varepsilon}(\boldsymbol{v})dx+\lambda \int_{\Omega}(\mathrm{div}\boldsymbol{w})(\mathrm{div}\boldsymbol{v})dx
         +\int_{\partial\Omega} \boldsymbol{w}\cdot\boldsymbol{v}ds=\int_{\partial\Omega} \boldsymbol{f}\cdot\boldsymbol{v}ds.
\end{align}
Let $\boldsymbol{v}=\boldsymbol{w}$, then
\begin{align*}
2\mu\int_{\Omega}\boldsymbol{\varepsilon}(\boldsymbol{w}):\boldsymbol{\varepsilon}(\boldsymbol{w})dx
+\int_{\partial\Omega} \boldsymbol{w}\cdot\boldsymbol{w}ds
\lesssim \|\boldsymbol{f}\|_{-\frac{1}{2},\partial\Omega}\|\boldsymbol{w}\|_{\frac{1}{2},\partial\Omega}.
\end{align*}
Thus, from Theorem 7 in \cite{Dom¨ªnguez2021}, we have
\begin{align}\label{eq16}
\|\boldsymbol{w}\|_{1}\lesssim \|\boldsymbol{f}\|_{-\frac{1}{2},\partial\Omega}.
\end{align}
By Lemma 2.1 in \cite{Brenner1992}, there exists $\boldsymbol{w}^{*}\in \widehat{\boldsymbol{H}}^{1}(\Omega)
=\{\boldsymbol{v}\in \boldsymbol{H}^{1}(\Omega): \int_{\Omega}\boldsymbol{v}dx=0,~\int_{\Omega}\mathrm{rot} \boldsymbol{v}dx=0\}$
such that
\begin{align}\label{eq17}
\mathrm{div} \boldsymbol{w}^{*}=\mathrm{div} \boldsymbol{w}
\end{align}
and
\begin{align}\label{eq18}
\|\boldsymbol{w}^{*}\|_{1} \lesssim \|\mathrm{div} \boldsymbol{w}\|_{0}.
\end{align}
Choosing $\boldsymbol{v}=\boldsymbol{w}^{*}$ in (\ref{eq15}), and from (\ref{eq17}), (\ref{eq16}) and (\ref{eq18}) we deduce
\begin{align}\label{eq19}
\nonumber&\lambda\|\mathrm{div} \boldsymbol{w}\|_{0}^{2}
\lesssim\|\boldsymbol{f}\|_{-\frac{1}{2},\partial\Omega}\|\boldsymbol{w}^{*}\|_{1}
  +2\mu\|\boldsymbol{\varepsilon}(\boldsymbol{w})\|_{0}\|\boldsymbol{\varepsilon}(\boldsymbol{w}^{*})\|_{0}
  +\|\boldsymbol{w}\|_{0,\partial\Omega}\|\boldsymbol{w}^{*}\|_{0,\partial\Omega}  \\
&\lesssim \|\boldsymbol{f}\|_{-\frac{1}{2},\partial\Omega}\|\mathrm{div} \boldsymbol{w}\|_{0}+2\mu\|\boldsymbol{f}\|_{-\frac{1}{2},\partial\Omega}\|\mathrm{div} \boldsymbol{w}\|_{0},
\end{align}
hence, we get
\begin{align}\label{eq20}
\lambda\|\mathrm{div} \boldsymbol{w}\|_{0}\lesssim \|\boldsymbol{f}\|_{-\frac{1}{2},\partial\Omega},
\end{align}
which together with (\ref{eq16}) and (\ref{eq20}) yields (\ref{eq13}).\\

\indent From the conclusions in the first two cases, we can define the solution operator
  $A:\boldsymbol{H}^{-\frac{1}{2}}(\partial\Omega)\to \boldsymbol{H}^{1}(\Omega)$ by
\begin{align}\label{eq21}
a(A\boldsymbol{f},\boldsymbol{v})=b(\boldsymbol{f},\boldsymbol{v}),\quad \forall \boldsymbol{v}\in \boldsymbol{H}^{1}(\Omega).
\end{align}
Define the norm
$$\|\boldsymbol{w}\|_{\widetilde{\boldsymbol{H}}^{1+r}(\Omega)}=\|\boldsymbol{w}\|_{1+r}+\lambda \|\mathrm{div}\boldsymbol{w} \|_{r},\quad\forall \boldsymbol{w}\in \boldsymbol{H}^{1+r}(\Omega)~(0\leq r\leq 1).$$
It is obvious that the norm $\|\cdot\|_{\widetilde{\boldsymbol{H}}^{1+r}(\Omega)}$ is equivalent to the norm $\|\cdot\|_{1+r}$. \\
Since $\boldsymbol{H}^{\frac{1}{2}}(\partial \Omega)\hookrightarrow \boldsymbol{H}^{-\frac{1}{2}}(\partial \Omega) $ and $\boldsymbol{H}^{2}(\Omega)\hookrightarrow \boldsymbol{H}^{1}(\Omega)$,  the operator $A$ can also be defined as
$A: \boldsymbol{H}^{\frac{1}{2}}(\partial \Omega)\rightarrow \boldsymbol{H}^{2}(\Omega) $.
Note that $A\boldsymbol{f}=\boldsymbol{w}$, then (\ref{eq12}) and (\ref{eq13}) can be rewritten as
\begin{align*}
 &\|A\boldsymbol{f}\|_{\widetilde{\boldsymbol{H}}^{2}(\Omega)}\leq C_{1}\|\boldsymbol{f}\|_{\frac{1}{2},\partial \Omega},\\
 &\|A\boldsymbol{f}\|_{\widetilde{\boldsymbol{H}}^{1}(\Omega)}\leq C_{2}\|\boldsymbol{f}\|_{-\frac{1}{2},\partial \Omega}.
\end{align*}
Then, from Proposition 14.1.5 in \cite{Brenner2007} we have
\begin{align*}
\|A\|_{\boldsymbol{L}^{2}(\partial \Omega)\rightarrow \boldsymbol{H}^{1+\frac{1}{2}}(\Omega) }
\leq \|A\|_{\boldsymbol{H}^{-\frac{1}{2}}(\partial \Omega)\rightarrow \boldsymbol{H}^{1}(\Omega)}^{\frac{1}{2}}
     \|A\|_{\boldsymbol{H}^{\frac{1}{2}}(\partial \Omega)\rightarrow \boldsymbol{H}^{2}(\Omega) }^{\frac{1}{2}}\leq \sqrt{C_{2}C_{1}},
\end{align*}
i.e., $A: \boldsymbol{L}^{2}(\partial \Omega)\rightarrow \boldsymbol{H}^{1+\frac{1}{2}}(\Omega) $ is bounded, hence,
\begin{align*}
\|A\boldsymbol{f}\|_{\widetilde{\boldsymbol{H}}^{1+\frac{1}{2}}(\Omega) }&\leq\|A\|_{\boldsymbol{L}^{2}(\partial \Omega)\rightarrow \boldsymbol{H}^{1+\frac{1}{2}}(\Omega)}\|\boldsymbol{f}\|_{0,\partial\Omega}\leq C_{3}\|\boldsymbol{f}\|_{0,\partial\Omega},
\end{align*}
where $C_{3}=\sqrt{C_{1}C_{2}}$ is independent of $\mu$ and $\lambda$.  From the definition of $\|\cdot\|_{\widetilde{\boldsymbol{H}}^{1+r}(\Omega)}$ we obtain (\ref{eq14}).
\end{proof}

\begin{remark}
\indent For any given $\boldsymbol{f}\in \boldsymbol{L}^{2}(\partial\Omega)$, from (\ref{eq14}) we easily have $\boldsymbol{w}\in \boldsymbol{H}^{1+r}(\Omega)$, $r<\frac{1}{2}$ and $r$ can be close to $\frac{1}{2}$ arbitrarily, and
\begin{align}\label{eq22}
 \|\boldsymbol{w}\|_{1+r}+\lambda \|\mathrm{div}\boldsymbol{w} \|_{r}\leq C_{R}\|\boldsymbol{f}\|_{0,\partial\Omega}.
\end{align}
\end{remark}

The DG approximation of (\ref{eq11}) is to find $\boldsymbol{w_{h}}\in \boldsymbol{S}^{h}$, such that
\begin{align}\label{eq23}
a_{h}(\boldsymbol{w_{h},v_{h}})= b_{h}(\boldsymbol{f},\boldsymbol{v_{h}}) ,\quad  \forall \boldsymbol{v_{h}} \in \boldsymbol{S}^{h}.
\end{align}
We have the following consistency property.
\begin{lemma}
Let $\boldsymbol{w}$ and $\boldsymbol{w}_{h}$ be the solution of (\ref{eq11}) and (\ref{eq23}), respectively,
 then the DG approximation (\ref{eq23}) is consistent:
\begin{align}\label{eq24}
a_{h}(\boldsymbol{w}-\boldsymbol{w_{h}},\boldsymbol{v_{h}})=0,\quad \forall \boldsymbol{v_{h}}\in \boldsymbol{S}^{h}.
\end{align}
\end{lemma}
\begin{proof} Applying Green's formula elementwise in $\mathcal{T}_{h}$, and using the fact that
$ [\![\phi\varphi]\!]=\{\phi\}[\![\varphi]\!]+[\![\phi]\!]\{\varphi\}$ and $\int_{e}[\![\boldsymbol{\sigma}(\boldsymbol{w})\boldsymbol{n}]\!]\cdot \boldsymbol{v}=0$ on inner face $e$, we deduce
\begin{align*}
&0=\sum_{K\in T_{h}}\int_{K}(-\mathrm{div} \boldsymbol{\sigma}(\boldsymbol{w}))\boldsymbol{v}dx
 =\sum_{K\in T_{h}}\int_{K}\boldsymbol{\sigma}(\boldsymbol{w}):\boldsymbol{\varepsilon}(\boldsymbol{v})dx
    -\sum_{e\in \mathcal{E}_{h} }\int_{e}(\boldsymbol{\sigma}(\boldsymbol{w})\boldsymbol{n}) \cdot\boldsymbol{v}ds\\
 &=\sum_{K\in T_{h}}\int_{K}\boldsymbol{\sigma}(\boldsymbol{w}):\boldsymbol{\varepsilon}(\boldsymbol{v})dx
  -\sum_{e\in \mathcal{E}^{i}_{h} }\int_{e}
  [\![(\boldsymbol{\sigma}(\boldsymbol{w})\boldsymbol{n})\cdot \boldsymbol{v}]\!]ds
 -\sum_{e\in \mathcal{E}^{b}_{h} }\int_{e}(\boldsymbol{\sigma}(\boldsymbol{w})\boldsymbol{n})\cdot \boldsymbol{v}ds\\
 &=\sum_{K\in T_{h}}\int_{K}\boldsymbol{\sigma}(\boldsymbol{w}):\boldsymbol{\varepsilon}(\boldsymbol{v})dx-\sum_{e\in \mathcal{E}^{i}_{h} }\int_{e}\left( \{\boldsymbol{\sigma}(\boldsymbol{w})\boldsymbol{n}\}\cdot[\![\boldsymbol{v}]\!]
     +[\![\boldsymbol{\sigma}(\boldsymbol{w})\boldsymbol{n}]\!]\cdot\{\boldsymbol{v}\}\right)ds
 -\sum_{e\in \mathcal{E}^{b}_{h} }\int_{e}( \boldsymbol{f}- \boldsymbol{w})\cdot\boldsymbol{v}ds\\
 &=\sum_{K\in T_{h}}\int_{K}\boldsymbol{\sigma}(\boldsymbol{w}):\boldsymbol{\varepsilon}(\boldsymbol{v})dx-\sum_{e\in \mathcal{E}^{i}_{h} }\int_{e}\{\boldsymbol{\sigma}(\boldsymbol{w})\boldsymbol{n}\}\cdot[\![\boldsymbol{v}]\!]ds
  -\sum_{e\in \mathcal{E}^{b}_{h} }\int_{e}( \boldsymbol{f}- \boldsymbol{w})\cdot\boldsymbol{v}ds
\end{align*}
\begin{align*}
  &=2\mu\left( \sum\limits_{K\in\mathcal{T}_{h}}\int_{K}\boldsymbol{\varepsilon}(\boldsymbol{w}):\boldsymbol{\varepsilon}(\boldsymbol{v})dx
-\sum_{e\in \mathcal{E}^{i}_{h}}\int_{e}\{\boldsymbol{\varepsilon}(\boldsymbol{w})\boldsymbol{n}\}\cdot[\![\boldsymbol{v}]\!]ds \right.\\
\nonumber&\quad\left.
-\sum_{e\in \mathcal{E}^{i}_{h}}\int_{e}\{\boldsymbol{\varepsilon}(\boldsymbol{v})\boldsymbol{n}\}\cdot[\![\boldsymbol{w}]\!]ds
+\sum_{e\in \mathcal{E}^{i}_{h}}\frac{\gamma_{\mu} }{h_{e}} \int_{e}[\![\boldsymbol{w}]\!]\cdot[\![\boldsymbol{v}]\!]ds
\right)\\
\nonumber&\quad+\lambda\left(
\sum \limits_{K\in\mathcal{T}_{h}}\int_{K}(\mathrm{div}\boldsymbol{w})(\mathrm{div}\boldsymbol{v})dx
-\sum_{e\in \mathcal{E}^{i}_{h}}\int_{e}\{\mathrm{div}\boldsymbol{w}\}[\![\boldsymbol{v}\cdot\boldsymbol{n}]\!]ds\right.\\
\nonumber&\quad\left.
-\sum_{e\in \mathcal{E}^{i}_{h}}\int_{e}\{\mathrm{div}\boldsymbol{v}\} [\![\boldsymbol{w}\cdot\boldsymbol{n}]\!]ds
+ \sum_{e\in \mathcal{E}^{i}_{h}}\frac{\gamma_{\lambda}{\color{red} }}{h_{e}} \int_{e}[\![\boldsymbol{w}\cdot\boldsymbol{n}]\!][\![\boldsymbol{v}\cdot\boldsymbol{n}]\!]ds\right)-\sum_{e\in \mathcal{E}^{b}_{h}}\int_{e}(\boldsymbol{f}-\boldsymbol{w})\cdot\boldsymbol{v}ds,
\end{align*}
which implies that
\begin{align}\label{eq25}
a_{h}(\boldsymbol{w},\boldsymbol{v})=b_{h}(\boldsymbol{f},\boldsymbol{v}),\quad  \forall \boldsymbol{v}\in \boldsymbol{H}^{1+r}(\mathcal{T}_{h}).
\end{align}
 It is obvious that $\boldsymbol{S}^{h}\subset \boldsymbol{H}^{1+r}(\mathcal{T}_{h})$, then subtracting (\ref{eq23}) from (\ref{eq25}) we  get (\ref{eq24}).
\end{proof}

\indent In order to analyze the error estimates of DG finite element solutions, we need the following trace inequality.
\begin{lemma}\label{lemma3.4}
Suppose that
$\boldsymbol{w}\in \boldsymbol{H}^{1+\xi}(K) (0<\xi<\frac{1}{2})$ and $\mathrm{div}\boldsymbol{\sigma(w)}\in \boldsymbol{L}^{2}(K)$,
then
\begin{align}\label{eq26}
 \|\boldsymbol{\sigma(w) n}\|_{\xi-\frac{1}{2},e}
  \lesssim
  h_{K}^{\xi-\frac{1}{2}}
       \left(h_{K}^{1-\xi}\|\mathrm{div}\boldsymbol{\sigma(w)}\|_{0,K}+\|\boldsymbol{\sigma(w)}\|_{\xi,K}\right),
\quad \forall K\in\mathcal{T}_{h},~ e\subset\partial K.
\end{align}
\end{lemma}
\begin{proof}
Inequality (\ref{eq26}) is contained in the proof of Corollary 3.3 on page 1384 of \cite{bernardi} or Lemma 2.1 in \cite{Cai2011}, or Lemma 2.3 in \cite{Cai2017}, or Lemma 3.3 in \cite{Bi2021}. The brief proof is presented as follows for easy reading.

For all $\boldsymbol{v} \in \boldsymbol{H}^{1-\xi}(K)$ with $0<\xi<\frac{1}{2}$, we shall prove that the following Green's formula
\begin{align}\label{eq27}
\int_{\partial K}\boldsymbol{\sigma(w) n \cdot v} ds=\int_{K} \mathrm{div}\boldsymbol{\sigma(w) v}dx +\int_{K}\boldsymbol{\sigma(w)}:\nabla \boldsymbol{ v}dx.
\end{align}
 The validity of (\ref{eq27}) follows from the standard density argument and the fact that (\ref{eq27}) holds for $C^{\infty}(\bar{K})$ functions.
This formula holds for all $\boldsymbol{w} \in \boldsymbol{H}^{1+\xi}(K)$ with $\mathrm{div}\boldsymbol{\sigma}(\boldsymbol{ w}) \in L^{2}(K)$ and for all $\boldsymbol{v} \in \boldsymbol{H}^{1-\xi}(K)$ with $0<\xi<\frac{1}{2}$. Let $\boldsymbol{H}^{-\xi}(K)$ be the dual of $\boldsymbol{H}_{0}^{\xi}(K)$ which is the closure of $C_{0}^{\infty}(K)$ in the $\boldsymbol{H}^{\xi}(K)$ norm. Since $\boldsymbol{H}^{\xi}(K)$ is the same space as $\boldsymbol{H}_{0}^{\xi}(K)$ for $\xi \in(0,\frac{1}{2})$ (see, e.g., Theorem 1.4.2.4 in \cite{Grisvard1985}), then $\nabla \boldsymbol{ v}$ is in $\boldsymbol{H}^{-\xi}(K)$. That is, $\boldsymbol{\sigma}(\boldsymbol{w})$ and $\nabla \boldsymbol{v}$ can be viewed as a duality pairing between $\boldsymbol{H}^{\xi}(K)$ and $\boldsymbol{H}^{-\xi}(K)$. Hence, (\ref{eq27}) is valid.

For any $\boldsymbol{g} \in \boldsymbol{H}^{\frac{1}{2}-\xi}(e)$,
from 
Lemma $3.2$ in \cite{Bi2021} we know that there exists a lifting $\boldsymbol{v}_{\boldsymbol{g}}$ of $\boldsymbol{g}$ such that $\boldsymbol{v}_{\boldsymbol{g}} \in \boldsymbol{H}^{1-\xi}(K),\boldsymbol{v}_{\boldsymbol{g}} \mid _{e}=\boldsymbol{g},\mathbf{v}_{\boldsymbol{g}} \mid _{\partial K\backslash e}=0$, and
\begin{align}\label{eq28}
\|\nabla\boldsymbol{v}_{\boldsymbol{g}}\|_{-\xi,K}+h_{K}^{\xi-1}\|\boldsymbol{v}_{\boldsymbol{g}}\|_{0,K}
\lesssim
  h_{K}^{\xi-\frac{1}{2}}
  \|\boldsymbol{g}\|_{\frac{1}{2}-\xi,e}.
\end{align}
From (\ref{eq27}), (\ref{eq28}) and the definition of the dual norm we deduce
\begin{align*}
\nonumber&\int_{e}\boldsymbol{\sigma(w) n }\cdot \boldsymbol{g} ds
=\int_{\partial K}\boldsymbol{\sigma(w) n }\cdot \boldsymbol{ v_{g}} ds
=\int_{K} \mathrm{div}\boldsymbol{\sigma(w)} \boldsymbol{v_{g}}dx +\int_{K}\boldsymbol{\sigma(w)}:\nabla \boldsymbol{ v_{g}}dx\\
\nonumber&\lesssim \|\mathrm{div}\boldsymbol{\sigma(w)} \|_{0,K}\|\boldsymbol{ v_{g}}\|_{0,K}
+\|\boldsymbol{\sigma(w)} \|_{\xi,K}\|\nabla \boldsymbol{ v_{g}}\|_{-\xi,K}\\
&\lesssim
  h_{K}^{\xi-\frac{1}{2}
  }\left(h_{K}^{1-\xi}\|\mathrm{div}\boldsymbol{\sigma(w)} \|_{0,K}
+\|\boldsymbol{\sigma(w)} \|_{\xi,K}\right)\|\boldsymbol{ g}\|_{\frac{1}{2}-\xi,e},
\end{align*}
by the definition of the dual norm we have
\begin{align*}
\|\boldsymbol{\sigma(w)}\boldsymbol{n}\|_{\xi-\frac{1}{2},e}=\sup_{\boldsymbol{g}\in \boldsymbol{H}^{\frac{1}{2}-\xi}(e)}
  \frac{\mid \int_{e}  \boldsymbol{\sigma(w) n }\cdot \boldsymbol{g} ds \mid}{\|\boldsymbol{g}\|_{\frac{1}{2}-\xi,e}}.
\end{align*}
Combining the above two relationships we obtain (\ref{eq26}).
\end{proof}

To obtain the a prior error estimate, we need the Brezzi-Douglas-Marini (BDM) interpolation 
which was introduced by Brezzi et al. (see \cite{Brezzi1985} and \cite{Brezzi1991}).

\begin{lemma}\label{lemma3.5}

If the mesh consists of triangle or tetrahedron, there is an interpolation operator $\boldsymbol{\pi}:\boldsymbol{H}^{1}(\Omega) \rightarrow \boldsymbol{S}^{h}$ such that it is valid the following properties for all
$\boldsymbol{u} \in \boldsymbol{H}^{1+r}(K)(r\geq0)$:

(1) $[\![\boldsymbol{n}\cdot \boldsymbol{\pi} \boldsymbol{u}]\!]=0$;

(2) $\|\boldsymbol{u}-\boldsymbol{\pi} \boldsymbol{u}\|_{m,K} \leq C h_{K}^{1+r-m}\|\boldsymbol{u}\|_{1+r,K}$ with $m=0,1,2$, and $m \leqslant r+1$;

(3) $\|\mathrm{div}(\boldsymbol{u}-\boldsymbol{\pi}\boldsymbol{u})\|_{m,K} \leq C h_{K}^{1+r-m}\|\mathrm{div}\boldsymbol{ u}\|_{1+r,K}$ with $m=0,1$, and $m \leqslant r$;
\end{lemma}
\begin{proof}
From Propositions III.3.6, Proposition III.3.7 and Proposition III.3.8 in \cite{Brezzi1991} and theirs proofs,
combining Scott-Zhang interpolation (see \cite{Ciarlet2013}),
we know that the conclusions hold true when $r$ is a non-negative real number.
\end{proof}

From the arguing method of Theorem 8 in \cite{Hansbo2002}, we can provide the following theorem.
\begin{theorem}\label{thm3.6}
Let $\boldsymbol{w}$ and $\boldsymbol{w_{h}}$ be the solution of (\ref{eq11}) and (\ref{eq23}), respectively,
$\boldsymbol{w}\in \boldsymbol{H}^{1+r}(\Omega)$, $r\in (\frac{1}{2}, k]$. Then there holds
\begin{align}\label{eq29}
\|\boldsymbol{w}-\boldsymbol{w_{h}}\|_{h}\lesssim h^{r}\left(\sqrt{2\mu}\|\boldsymbol{w}\|_{1+r}+\sqrt{\lambda}\|\mathrm{div} \boldsymbol{w}\|_{r}\right).
\end{align}
\end{theorem}
\begin{proof}
Taking $\boldsymbol{v}=\boldsymbol{\pi} \boldsymbol{w}$ and denoting $\boldsymbol{\eta}=\boldsymbol{w}-\boldsymbol{\pi}\boldsymbol{w}$, we have
\begin{align}\label{eq30}
\|\boldsymbol{w_{h}}-\boldsymbol{w}\|_{h}\leq \|\boldsymbol{w_{h}}-\boldsymbol{v}\|_{h}+\|\boldsymbol{\eta}\|_{h}.
\end{align}
Since $\|\cdot\|_{h}$ are $\|\cdot\|_{\mathrm{dG}}$ are equivalent in  $\boldsymbol{S}^{h}$, it follows from the coercivity and consistency that
\begin{align*}
\nonumber&\|\boldsymbol{w_{h}}-\boldsymbol{v}\|^{2}_{h} \lesssim m\|\boldsymbol{w_{h}}-\boldsymbol{v}\|^{2}_{\mathrm{dG}}
\leq a_{h}(\boldsymbol{w_{h}}-\boldsymbol{v},\boldsymbol{w_{h}}-\boldsymbol{v})\\
&\quad= a_{h}(\boldsymbol{\eta},\boldsymbol{w_{h}}-\boldsymbol{v})
\lesssim \|\boldsymbol{\eta}\|_{h}\|\boldsymbol{w_{h}}-\boldsymbol{v}\|_{h},
\end{align*}
which together with (\ref{eq30}) yields
\begin{align}\label{eq31}
\|\boldsymbol{w_{h}}-\boldsymbol{w}\|_{h}\lesssim \|\boldsymbol{\eta}\|_{h}.
\end{align}
To deal with the boundary terms we need the following trace inequality (see, e.g., \cite{Brenner2007}):
\begin{align}\label{eq32}
\|\boldsymbol{v}\|_{0,e}^{2} \lesssim \|\boldsymbol{v}\|_{0,K}\left(h_{K}^{-1}\|\boldsymbol{v}\|_{0,K}+\|\boldsymbol{v}\|_{1,K}\right), \quad \boldsymbol{v} \in \boldsymbol{H}^{1}(K).
\end{align}
From (\ref{eq32}) and Lemma \ref{lemma3.5}, we derive that
\begin{align*}
\nonumber&\sum_{e\in {\mathcal{E}^{b}_{h}}}\|\boldsymbol{\eta}\|^{2}_{0,e}
\lesssim \sum \limits_{K\in\mathcal{T}_{h}}\|\boldsymbol{\eta}\|_{0,K}\left(h_{K}^{-1}\|\boldsymbol{\eta}\|_{0,K} + \|\boldsymbol{\eta}\|_{1,K}\right)\\
&\lesssim \sum \limits_{K\in\mathcal{T}_{h}} \left( h^{-1}_{K}h^{2+2r}_{K}\| \boldsymbol{w}\|^{2}_{1+r,K}
+h^{1+2r}_{K}\| \boldsymbol{w}\|^{2}_{1+r,K}\right)
\lesssim\sum \limits_{K\in\mathcal{T}_{h}} h^{1+2r}_{K}\|\boldsymbol{w}\|^{2}_{1+r,K}.
\end{align*}
\noindent Since the mesh is shape-regular and satisfies $h_{e}^{-1} \lesssim h_{K}^{-1}$, from the definition of $\|\cdot\|_{h}$, the trace inequality (\ref{eq32}) and Lemma \ref{lemma3.5}, using the argument method similar to Theorem 8 in \cite{Hansbo2002}, we can deduce

\begin{align*}
\nonumber\|\boldsymbol{\eta}\|_{h}^{2}&\lesssim 2\mu\sum_{K\in {\mathcal{T}_{h}}}\|\boldsymbol{\varepsilon}(\boldsymbol{\eta})\|^{2}_{0,K}
+2\mu\sum_{e\in \mathcal{E}^{i}_{h}}h_{e}\|\{\boldsymbol{\varepsilon}(\boldsymbol{\eta})\boldsymbol{n}\}\|^{2}_{0,e}
+ 2\mu\sum_{e\in \mathcal{E}^{i}_{h}}\gamma_{\mu} h^{-1}_{e}\|[\![\boldsymbol{\eta}]\!]\|^{2}_{0,e}\\
\nonumber&\quad+\lambda \sum_{K\in {\mathcal{T}_{h}}}\|\mathrm{div}\boldsymbol{\eta}\|^{2}_{0,K}
   +\lambda\sum_{e\in \mathcal{E}^{i}_{h}}h_{e}\|\{\mathrm{div}\boldsymbol{\eta}\}\|^{2}_{0,e}
+\lambda\sum_{e\in \mathcal{E}^{i}_{h}}\gamma _{\lambda} h^{-1}_{e}\|[\![\boldsymbol{\eta} \cdot \boldsymbol{n}]\!]\|^{2}_{0,e}
   +\sum_{e\in \mathcal{E}^{b}_{h}}\|\boldsymbol{\eta}\|^{2}_{0,e}\\
\nonumber&\lesssim 2\mu\sum_{K\in {\mathcal{T}_{h}}}
     h_{K}^{2r}\|\boldsymbol{w}\|^{2}_{1+r,K}
+\lambda \sum_{K\in {\mathcal{T}_{h}}}h_{K}^{2r}\|\mathrm{div}\boldsymbol{w}\|^{2}_{r,K}
 \lesssim h^{2r}\left( 2\mu\|\boldsymbol{w}\|^{2}_{1+r}+\lambda \|\mathrm{div}\boldsymbol{w}\|^{2}_{r}\right),
\end{align*}
which implies that
\begin{align}\label{eq33}
\|\boldsymbol{\eta}\|_{h}\lesssim h^{r}(\sqrt{2\mu} \|\boldsymbol{w}\|_{1+r}+\sqrt{\lambda}\|\mathrm{div}\boldsymbol{w}\|_{r}).
\end{align}
Combining (\ref{eq33}) and (\ref{eq31}), we get (\ref{eq29}).
\end{proof}

\indent In order to prove the discrete solution operator converges to the exact solution operator in $L^2(\partial\Omega)$ norm, we first need to derive the error estimation of the source problem under the condition that the solution has low regularity.

\begin{theorem}\label{thm3.7}
For any given $\boldsymbol{f}\in \boldsymbol{L}^{2}(\partial\Omega)$, let $\boldsymbol{w}\in \boldsymbol{H}^{1+r}(\Omega)$~($0<r<\frac{1}{2}$) be the solution of (\ref{eq11}) and $\boldsymbol{w_{h}}$ be the solution of (\ref{eq23}). Suppose that the regularity estimate (\ref{eq22}) is valid, then there holds
\begin{align}\label{eq34}
\|\boldsymbol{w}-\boldsymbol{w_{h}}\|_{\mathrm{dG}} \lesssim  h^{r}\|\boldsymbol{f}\|_{0,\partial\Omega}.
\end{align}
\end{theorem}
\begin{proof}
Let $\boldsymbol{v}=\boldsymbol{\pi} \boldsymbol{w}$ and $\boldsymbol{\eta}=\boldsymbol{w}-\boldsymbol{\pi}\boldsymbol{w}$, we have
\begin{align}\label{eq35}
\|\boldsymbol{w_{h}}-\boldsymbol{w}\|_{\mathrm{dG}}\leq \|\boldsymbol{w_{h}}-\boldsymbol{v}\|_{\mathrm{dG}}+\|\boldsymbol{\eta}\|_{\mathrm{dG}}.
\end{align}

From (\ref{eq22}) we know that for any given $ \boldsymbol{f} \in \boldsymbol{L}^{2}(\partial \Omega), \boldsymbol{w}\in \boldsymbol{H}^{1+\xi}(\Omega),0<\xi<\frac{1}{2}$ and $\xi$ can be close to $\frac{1}{2}$ arbitrarily. Then
 using (\ref{eq8}), (\ref{eq24}), 
 the definition of $a_{h}(\cdot,\cdot)$  with the fact that $[\![\boldsymbol{\eta}\cdot\boldsymbol{n}]\!]=0$ on inner face $e$, the inverse estimation and the Cauchy-Schwarz inequality, we deduce that
\begin{align*}
&m\|\boldsymbol{w_{h}}-\boldsymbol{v}\|^{2}_{\mathrm{dG}}\leq a_{h}(\boldsymbol{w_{h}}-\boldsymbol{v},\boldsymbol{w_{h}}-\boldsymbol{v})=a_{h}(\boldsymbol{\eta},\boldsymbol{w_{h}}-\boldsymbol{v})\nonumber\\
& =2\mu\left( \sum\limits_{K\in\mathcal{T}_{h}}\int_{K}\boldsymbol{\varepsilon}(\boldsymbol{\eta}):\boldsymbol{\varepsilon}(\boldsymbol{w_{h}}-\boldsymbol{v})dx
-\sum_{e\in \mathcal{E}^{i}_{h}}\int_{e}\{\boldsymbol{\varepsilon}(\boldsymbol{\eta})\boldsymbol{n}\}\cdot[\![\boldsymbol{w_{h}}-\boldsymbol{v}]\!]ds \right.\\
\nonumber&\quad\left.
-\sum_{e\in \mathcal{E}^{i}_{h}}\int_{e}\{\boldsymbol{\varepsilon}(\boldsymbol{w_{h}}-\boldsymbol{v})\boldsymbol{n}\}\cdot[\![\boldsymbol{\eta}]\!]ds
+\sum_{e\in \mathcal{E}^{i}_{h}}\frac{\gamma_{\mu} }{h_{e}} \int_{e}[\![\boldsymbol{\eta}]\!]\cdot[\![\boldsymbol{w_{h}}-\boldsymbol{v}]\!]ds
\right)\\
\nonumber&\quad+\lambda\left(
\sum \limits_{K\in\mathcal{T}_{h}}\int_{K}(\mathrm{div}\boldsymbol{\eta})\mathrm{div}(\boldsymbol{w_{h}}-\boldsymbol{v})dx
-\sum_{e\in \mathcal{E}^{i}_{h}}\int_{e}\{\mathrm{div}\boldsymbol{\eta}\}[\![(\boldsymbol{w_{h}}-\boldsymbol{v})\cdot\boldsymbol{n}]\!]ds\right.\\
\nonumber&\quad\left.
-\sum_{e\in \mathcal{E}^{i}_{h}}\int_{e}\{\mathrm{div}(\boldsymbol{w_{h}}-\boldsymbol{v})\}
 [\![\boldsymbol{\eta}\cdot\boldsymbol{n}]\!]ds
+ \sum_{e\in \mathcal{E}^{i}_{h}}\frac{\gamma_{\lambda}{\color{red} }}{h_{e}} \int_{e}
[\![\boldsymbol{\eta}\cdot\boldsymbol{n}]\!]
[\![(\boldsymbol{w_{h}}-\boldsymbol{v})\cdot\boldsymbol{n}]\!]ds\right)\\
&\quad+\sum_{e\in \mathcal{E}^{b}_{h}}\int_{e}\boldsymbol{\eta}\cdot(\boldsymbol{w_{h}}-\boldsymbol{v})ds,\\
&\lesssim\|\boldsymbol{w_{h}}-\boldsymbol{v}\|_{\mathrm{dG}}\|\boldsymbol{\eta}\|_{\mathrm{dG}}
-2\mu
\sum_{e\in \mathcal{E}^{i}_{h}}\int_{e}\{\boldsymbol{\varepsilon}(\boldsymbol{\eta})\boldsymbol{n}\}\cdot[\![\boldsymbol{w_{h}-v}]\!]ds
-\lambda
\sum_{e\in \mathcal{E}^{i}_{h}}\int_{e}\{\mathrm{div}\boldsymbol{\eta}\}[\![(\boldsymbol{w_{h}-v})\cdot \boldsymbol{n}]\!]ds
\nonumber\\
&\lesssim\|\boldsymbol{w_{h}}-\boldsymbol{v}\|_{\mathrm{dG}}\|\boldsymbol{\eta}\|_{\mathrm{dG}}+
   \sum_{e\in \mathcal{E}^{i}_{h}}\| (2\mu\boldsymbol{\varepsilon}(\boldsymbol{\eta})+\lambda \mathrm{div}\boldsymbol{\eta} \mathbf{I})\boldsymbol{n}\|_{\xi-\frac{1}{2}, e}\|[\![\boldsymbol{w_{h}-v}]\!]\|_{\frac{1}{2}-\xi, e}\nonumber
\end{align*}
\begin{align*}
&\lesssim\|\boldsymbol{w_{h}}-\boldsymbol{v}\|_{\mathrm{dG}}\|\boldsymbol{\eta}\|_{\mathrm{dG}}
   +\sum_{e\in \mathcal{E}^{i}_{h}}\| (2\mu\boldsymbol{\varepsilon}(\boldsymbol{\eta})+\lambda \mathrm{div}\boldsymbol{\eta} \mathbf{I})\boldsymbol{n}\|_{\xi-\frac{1}{2}, e} h_{e}^{\xi-\frac{1}{2}}\|[\![\boldsymbol{w_{h}-v}]\!]\|_{0, e}\nonumber\\
&\lesssim \|\boldsymbol{w_{h}}-\boldsymbol{v}\|_{\mathrm{dG}}\|\boldsymbol{\eta}\|_{\mathrm{dG}}
+\left(
\sum_{e\in \mathcal{E}^{i}_{h}}h^{2\xi}_{e}\|(2\mu\boldsymbol{\varepsilon}(\boldsymbol{\eta})+\lambda \mathrm{div}\boldsymbol{\eta} \mathbf{I})\boldsymbol{n}\|^{2}_{\xi-\frac{1}{2}, e}
  \right)^{\frac{1}{2}}
  \left(
  \sum_{e\in \mathcal{E}^{i}_{h}}h^{-1}_{e}\|[\![\boldsymbol{w_{h}-v}]\!]\|^{2}_{0, e}
  \right)^{\frac{1}{2}}\nonumber\\
  &\lesssim \|\boldsymbol{w_{h}}-\boldsymbol{v}\|_{\mathrm{dG}} \left(\|\boldsymbol{\eta}\|_{\mathrm{dG}}
   +(\sum_{e\in \mathcal{E}^{i}_{h}}h^{2\xi}_{e}\|(2\mu\boldsymbol{\varepsilon}(\boldsymbol{\eta})+\lambda \mathrm{div}\boldsymbol{\eta} \mathbf{I})\boldsymbol{n}\|^{2}_{\xi-\frac{1}{2}, e})^{\frac{1}{2}}\right).
\end{align*}

\noindent Through simplification, we get
\begin{align}\label{eq36}
 &m\|\boldsymbol{w_{h}}-\boldsymbol{v}\|_{\mathrm{dG}}
  \lesssim \|\boldsymbol{\eta}\|_{\mathrm{dG}}
  +\left(\sum_{e\in \mathcal{E}^{i}_{h}}h_{e}^{2\xi}\| (2\mu\boldsymbol{\varepsilon}(\boldsymbol{\eta})+\lambda \mathrm{div}\boldsymbol{\eta} \mathbf{I})\boldsymbol{n}\|^{2}_{\xi-\frac{1}{2}, e}\right)^{\frac{1}{2}}.
\end{align}
Combining (\ref{eq35}) and (\ref{eq36}), we have
\begin{align}\label{eq37}
\|\boldsymbol{w}-\boldsymbol{w_{h}}\|_{\mathrm{dG}}\lesssim \|\boldsymbol{\eta}\|_{\mathrm{dG}}
+\left(
\sum_{e\in \mathcal{E}^{i}_{h}}h_{e}^{2\xi}\| (2\mu\boldsymbol{\varepsilon}(\boldsymbol{\eta})+\lambda \mathrm{div}\boldsymbol{\eta} \mathbf{I})\boldsymbol{n}\|^{2}_{\xi-\frac{1}{2}, e}
\right)^{\frac{1}{2}}.
\end{align}

\indent We first estimate the first term $\|\boldsymbol{\eta}\|_{\mathrm{dG}}$ on the right-hand side of (\ref{eq37}). Note that the mesh satisfies $h_e^{-1} \lesssim h_K^{-1}$. Starting from the definition of $\|\cdot\|_{\mathrm{dG}}$ and using the crucial fact that $[\![\boldsymbol{\eta}\cdot\boldsymbol{n}]\!]=0$, together with trace inequality (\ref{eq32}), Lemma \ref{lemma3.5}, and the triangle inequality to split the remaining jump contributions, we deduce
\begin{align}\label{eq38}
\nonumber&\|\boldsymbol{\eta}\|^{2}_{\mathrm{dG}}
=2\mu  \sum \limits_{K\in\mathcal{T}_{h}}\|\boldsymbol{\varepsilon}(\boldsymbol{\eta}) \|^{2}_{0,K}+2\mu\sum_{e\in \mathcal{E}^{i}_{h}}\frac{\gamma_{\mu}}{h_{e}} \|[\![\boldsymbol{\eta}]\!]\|^{2}_{0,e}
    +\lambda \sum \limits_{K\in\mathcal{T}_{h}}\|\mathrm{div} \boldsymbol{\eta}\|^{2}_{0,K}
         +\sum_{e\in \mathcal{E}^{b}_{h}}\|\boldsymbol{\eta}\|^{2}_{0,e}\\
\nonumber&\lesssim  2\mu \sum \limits_{K\in\mathcal{T}_{h}}\|\nabla \boldsymbol{\eta}+ (\nabla \boldsymbol{\eta})^{T} \|^{2}_{0,K}
   +2\mu\sum \limits_{K\in\mathcal{T}_{h}}\left(h^{-2}_{K}\| \boldsymbol{\eta}\|^{2}_{0,K}
  +h_{K}^{-1}\| \boldsymbol{\eta}\|_{0,K} \| \boldsymbol{\eta}\|_{1,K}\right)\\
\nonumber&\quad+\sum \limits_{K\in\mathcal{T}_{h}} \lambda h^{2r}_{K}\|\mathrm{div}\boldsymbol{w}\|^{2}_{r,K}
     +\sum \limits_{K\in\mathcal{T}_{h}}\|\boldsymbol{\eta}\|_{0,K}\left( h_{K}^{-1}\|\boldsymbol{\eta}\|_{0,K}+\|\boldsymbol{\eta}\|_{1,K}\right)\\
\nonumber&\lesssim 2\mu\sum \limits_{K\in\mathcal{T}_{h}} \|\boldsymbol{\eta}\|^{2}_{1,K}
   +2\mu \sum \limits_{K\in\mathcal{T}_{h}}
   \left(
      h^{2r}_{K}\| \boldsymbol{w}\|^{2}_{1+r,K}+h^{2r}_{K}\| \boldsymbol{w}\|_{1+r,K}^{2}
   \right)\\
\nonumber&\quad+ \sum \limits_{K\in\mathcal{T}_{h}} \lambda h^{2r}_{K}\|\mathrm{div}\boldsymbol{w}\|^{2}_{r,K}
    + \sum \limits_{K\in\mathcal{T}_{h}} \left( h^{1+2r}_{K}\| \boldsymbol{w}\|^{2}_{1+r,K}
 +h^{1+2r}_{K}\| \boldsymbol{w}\|^{2}_{1+r,K} \right)\\
\nonumber&\lesssim 2\mu \sum \limits_{K\in\mathcal{T}_{h}} h^{2r}_{K}\|\boldsymbol{w}\|^{2}_{1+r,K}
    +2\mu \sum \limits_{K\in\mathcal{T}_{h}}h^{2r}_{K} \| \boldsymbol{w}\|^{2}_{1+r,K}
    +\sum \limits_{K\in\mathcal{T}_{h}} \lambda h^{2r}_{K}\|\mathrm{div}\boldsymbol{w}\|^{2}_{r,K}
    +\sum \limits_{K\in\mathcal{T}_{h}} h^{2r+1}_{K}\|\boldsymbol{w}\|^{2}_{1+r,K}\\
&\lesssim 2\mu h^{2r}\|\boldsymbol{w}\|^{2}_{1+r}+\lambda h^{2r}\|\mathrm{div}\boldsymbol{w}\|^{2}_{r},
\end{align}
which together with (\ref{eq22}) yields
\begin{align}\label{eq39}
\|\boldsymbol{\eta}\|_{\mathrm{dG}}\lesssim h^{r}\|\boldsymbol{f}\|_{0,\partial\Omega}.
\end{align}
Next, we will estimate the second term on the right-hand side of (\ref{eq37}). Invoking (\ref{eq26})
 with $\xi=r+\frac{1}{2}(\frac{1}{2}-r)\doteq r+\delta $,
the inverse estimation, formulas $\textit {(2)}$ and $\textit {(3)}$ in  Lemma \ref{lemma3.5} and (\ref{eq22}), we deduce

\begin{align}\label{eq40}
\nonumber&\left( \sum_{e\in \mathcal{E}^{i}_{h}}h_{e}^{2\xi}\| (2\mu\boldsymbol{\varepsilon}(\boldsymbol{\eta})
  +\lambda \mathrm{div} \boldsymbol{\eta} \mathbf{I}) \boldsymbol{n}\|^{2}_{\xi-\frac{1}{2}, e}   \right)^{\frac{1}{2}}\\
\nonumber&\lesssim\left( \sum_{K\in \mathcal{T}_{h}}h_{K}^{2\xi}h^{-2\delta}_{K}\left(h^{1-\xi}_{K}
      \|\mathrm{div} \boldsymbol{\sigma}(\boldsymbol{\eta})\|_{0,K}
      +\|\boldsymbol{\sigma}(\boldsymbol{\eta})\|_{\xi,K}\right)^{2}
  \right)^{\frac{1}{2}}\\
\nonumber&\lesssim\left( \sum_{K\in \mathcal{T}_{h}}h_{K}^{2\xi-2\delta}\left(h^{1-\xi}_{K}
  \|\mathrm{div} \boldsymbol{\sigma}(\boldsymbol{\pi w})\|_{0,K}
  +\|\boldsymbol{\sigma}(\boldsymbol{\eta})\|_{\xi,K}\right)^{2}
  \right)^{\frac{1}{2}}\\
\nonumber&\lesssim\left( \sum_{K\in \mathcal{T}_{h}}h_{K}^{2\xi-2\delta}\left(h^{1-\xi}_{K}
  \|\boldsymbol{\sigma}(\boldsymbol{\pi w})\|_{1,K}
  +\|\boldsymbol{\sigma}(\boldsymbol{\eta})\|_{\xi,K}\right)^{2}
  \right)^{\frac{1}{2}}\\
\nonumber&\lesssim\left( \sum_{K\in \mathcal{T}_{h}}h_{K}^{2\xi-2\delta}\left(h^{1-\xi}_{K}
  (2\mu h^{\xi-1}_{K}\|\boldsymbol{\varepsilon}(\boldsymbol{\pi w}) \|_{\xi,K}+\lambda h^{\xi-1}_{K}\|\mathrm{div} (\boldsymbol{\pi w})\boldsymbol{I}\|_{\xi,K})
  +\|2\mu\boldsymbol{\varepsilon}(\boldsymbol{\eta}) \|_{\xi,K}+\|\lambda \mathrm{div} \boldsymbol{\eta}\boldsymbol{I} \|_{\xi,K}\right)^{2}
  \right)^{\frac{1}{2}}\\
\nonumber&\lesssim\left( \sum_{K\in \mathcal{T}_{h}}h_{K}^{2\xi-2\delta}\left(
  2\mu \|(\boldsymbol{\pi w-w+w}) \|_{1+\xi,K}+\lambda \|\mathrm{div} \boldsymbol{(\pi w-w+w)}\|_{\xi,K}
   \right.\right.\\
\nonumber&\quad \left.\left.
  +\|2\mu\boldsymbol{\eta}  \|_{1+\xi,K}+\|\lambda \mathrm{div} \boldsymbol{\eta} \|_{\xi,K}\right)^{2}
  \right)^{\frac{1}{2}}\\
\nonumber&\lesssim\left( \sum_{K\in \mathcal{T}_{h}}h_{K}^{2\xi-2\delta}\left(
  2\mu \|\boldsymbol{w} \|_{1+\xi,K}+\lambda \|\mathrm{div} \boldsymbol{w}\|_{\xi,K}
  +2\mu\|\boldsymbol{w}  \|_{1+\xi,K}+\lambda\| \mathrm{div} \boldsymbol{w} \|_{\xi,K}\right)^{2}
  \right)^{\frac{1}{2}}\\
&\lesssim h^{\xi-\delta}\left(
  2\mu \|\boldsymbol{w} \|_{1+\xi}+\lambda \|\mathrm{div} \boldsymbol{w}\|_{\xi}
  \right)
\lesssim h^{r}\|\boldsymbol{f} \|_{0,\partial\Omega}.
\end{align}
 Substituting (\ref{eq39}) and (\ref{eq40}) into (\ref{eq37}), we obtain the desired result (\ref{eq34}) which shows that the error is independent of $\lambda$.
\end{proof}

Define the operator $T:\boldsymbol{L}^{2}(\partial\Omega)\rightarrow \boldsymbol{L}^{2}(\partial\Omega)$ by
\begin{align*}
T\boldsymbol{f}=(A\boldsymbol{f})',
\end{align*}
where $'$ denotes the restriction to $\partial\Omega$. Then (\ref{eq2}) has the equivalent operator form:
\begin{align*}
    T\boldsymbol{u}=\frac{1}{\rho}\boldsymbol{u}.
\end{align*}
Define the discrete solution operator
 $A_{h}:\boldsymbol{L}^{2}(\partial\Omega)\to \boldsymbol{S}^{h}$ satisfying
 \begin{align*}
a_{h}(A_{h}\boldsymbol{f},\boldsymbol{v})=b_{h}(\boldsymbol{f}, \boldsymbol{v}), \quad\forall \boldsymbol{v}\in \boldsymbol{S}^{h}.
 \end{align*}
 Let $\delta \boldsymbol{S}^{h}$ be the function space defined on $\partial\Omega$ which contains restriction of functions in $\boldsymbol{S}^{h}$ to $\partial\Omega$.
Define the operator $T_{h}:\boldsymbol{L}^{2}(\partial\Omega)\rightarrow \delta \boldsymbol{S}^{h}\subset \boldsymbol{L}^{2}(\partial\Omega)$ satisfying $T_{h}\boldsymbol{f}=(A_{h}\boldsymbol{f})'$. Then (\ref{eq3}) has the equivalent operator form:
\begin{align*}
    T_{h}\boldsymbol{u_{h}}=\frac{1}{\rho_{h}}\boldsymbol{u_{h}}.
\end{align*}
Thus, from (\ref{eq8}) together with the fact that $\|\cdot\|_{h}$ and $\|\cdot\|_{\mathrm{dG}}$ are equivalent in $\boldsymbol{S}^{h}$ and the definition of $A_h$, we deduce that
\begin{align*}
\|A_{h}\boldsymbol{f}\|^{2}_{h}\lesssim a_{h}(A_{h}\boldsymbol{f},A_{h}\boldsymbol{f})=C b_{h}(\boldsymbol{f},A_{h}\boldsymbol{f})
\lesssim \|\boldsymbol{f}\|_{0,\partial\Omega}\|A_{h}\boldsymbol{f}\|_{0,\partial\Omega}
\lesssim \|\boldsymbol{f}\|_{0,\partial\Omega}\|A_{h}\boldsymbol{f}\|_{h},
\end{align*}
which yields
\begin{align}\label{eq41}
\|A_{h}\boldsymbol{f}\|_{h}\lesssim \|\boldsymbol{f}\|_{0,\partial\Omega}.
\end{align}

\begin{theorem}\label{thm3.8}
Suppose that the regularity estimate (\ref{eq22}) is valid, and let $\boldsymbol{w}$ and $\boldsymbol{w_{h}}$  be the solutions of (\ref{eq11}) and (\ref{eq23}), respectively. When $\boldsymbol{f}\in \boldsymbol{L}^{2}(\partial\Omega)$ and $\boldsymbol{w}\in \boldsymbol{H}^{1+r}(\Omega)$~($0<r<\frac{1}{2}$), there holds
\begin{align}\label{eq42}
\|\boldsymbol{w}-\boldsymbol{w_{h}}\|_{0,\partial\Omega}\lesssim h^{2r}\|\boldsymbol{f}\|_{0,\partial\Omega};
\end{align}
When $\boldsymbol{w}\in \boldsymbol{H}^{1+t}(\Omega)$ $(\frac{1}{2}<t\leq k)$, there holds
\begin{align}\label{eq43}
\|\boldsymbol{w}-\boldsymbol{w_{h}}\|_{0,\partial\Omega}\lesssim h^{r+t}\left(\sqrt{2\mu}\|\boldsymbol{w}\|_{1+t}+\sqrt{\lambda}\|\mathrm{div} \boldsymbol{w}\|_{t}\right).
\end{align}
\end{theorem}
\begin{proof}
Introduce an auxiliary problem: for any given $\boldsymbol{g}\in \boldsymbol{L}^{2}(\partial\Omega)$, find $\boldsymbol{\psi}\in \boldsymbol{H}^{1}(\Omega)$ such that
\begin{align*}
a(\boldsymbol{v},\boldsymbol{\psi})=b(\boldsymbol{v,g}),\quad\forall \boldsymbol{v}\in \boldsymbol{H}^{1}(\Omega).
\end{align*}
Since $a(\cdot,\cdot)$ is symmetric and (\ref{eq22}) is valid, we know that $\boldsymbol{\psi}\in \boldsymbol{H}^{1+\xi}(\Omega),0<\xi<\frac{1}{2}$ and $\xi$ can be close to $\frac{1}{2}$ arbitrarily.
By using (\ref{eq25}), (\ref{eq24}) and the definitions of $a_{h}(\cdot,\cdot)$ and $\|\cdot\|_{\mathrm{dG}}$, we deduce
 \begin{align*}
\nonumber &b_{h}(\boldsymbol{w}-\boldsymbol{w_{h}},\boldsymbol{g})=a_{h}(\boldsymbol{w}-\boldsymbol{w_{h}},\boldsymbol{\psi})
      =a_{h}(\boldsymbol{w}-\boldsymbol{w_{h}},\boldsymbol{\psi}-\boldsymbol{\pi}\boldsymbol{\psi})\\
&\lesssim \|\boldsymbol{w}-\boldsymbol{w_{h}}\|_{\mathrm{dG}}\|\boldsymbol{\psi}-\boldsymbol{\pi}\boldsymbol{\psi}\|_{\mathrm{dG}}\nonumber-2\mu\sum_{e\in \mathcal{E}^{i}_{h}}\int_{e}\{\boldsymbol{\varepsilon}(\boldsymbol{\psi}-\boldsymbol{\pi}\boldsymbol{\psi})\boldsymbol{n}\}\cdot[\![\boldsymbol{w}-\boldsymbol{w_{h}}]\!]ds\\
&\quad-\lambda\sum_{e\in \mathcal{E}^{i}_{h}}\int_{e}\{\mathrm{div}(\boldsymbol{\psi}-\boldsymbol{\pi}\boldsymbol{\psi})\}[\![(\boldsymbol{w-w_{h}})\cdot \boldsymbol{n}]\!]ds.
\end{align*}

\indent When $\boldsymbol{w}\in \boldsymbol{H}^{1+r}(\Omega)$~$(0<r<\frac{1}{2})$,
 noticing that $[\![\boldsymbol{w}]\!]=0$ on inner face,
then using 
the inverse inequality, Cauchy-Schwarz inequality and (\ref{eq34}),
we deduce that

\begin{align}\label{eq44}
\nonumber &b_{h}(\boldsymbol{w}-\boldsymbol{w_{h}},\boldsymbol{g})
\lesssim  \|\boldsymbol{w}-\boldsymbol{w_{h}}\|_{\mathrm{dG}}\|\boldsymbol{\psi}-\boldsymbol{\pi}\boldsymbol{\psi}\|_{\mathrm{dG}}\\
\nonumber&\quad
  -\sum_{e\in \mathcal{E}^{i}_{h}}\int_{e}\{
  (2\mu \boldsymbol{\varepsilon}(\boldsymbol{\psi}-\boldsymbol{\pi}\boldsymbol{\psi})
  +\lambda \mathrm{div}(\boldsymbol{\psi}-\boldsymbol{\pi}\boldsymbol{\psi})\boldsymbol{I})\boldsymbol{n}\}\cdot[\![\boldsymbol{w-w_{h}}]\!]ds
\\
\nonumber
&\lesssim \|\boldsymbol{w}-\boldsymbol{w_{h}}\|_{\mathrm{dG}}\|\boldsymbol{\psi}-\boldsymbol{\pi}\boldsymbol{\psi}\|_{\mathrm{dG}}
+\sum_{e\in \mathcal{E}^{i}_{h}}
\|(2\mu \boldsymbol{\varepsilon}(\boldsymbol{\psi}-\boldsymbol{\pi}\boldsymbol{\psi})
 +\lambda \mathrm{div}(\boldsymbol{\psi}-\boldsymbol{\pi}\boldsymbol{\psi})\boldsymbol{I})\boldsymbol{n}\|_{\xi-\frac{1}{2},e}
\|[\![\boldsymbol{w-w_{h}}]\!]\|_{\frac{1}{2}-\xi,e}
\nonumber\\
&\lesssim\|\boldsymbol{w}-\boldsymbol{w_{h}}\|_{\mathrm{dG}}\|\boldsymbol{\psi}-\boldsymbol{\pi}\boldsymbol{\psi}\|_{\mathrm{dG}}\nonumber\\
&\quad+\sum_{e\in \mathcal{E}^{i}_{h}}\| (2\mu\boldsymbol{\varepsilon}(\boldsymbol{\psi}-\boldsymbol{\pi}\boldsymbol{\psi})
+\lambda \mathrm{div}(\boldsymbol{\psi}-\boldsymbol{\pi}\boldsymbol{\psi} )\boldsymbol{I})\boldsymbol{n}\|_{\xi-\frac{1}{2}, e} h_{e}^{\xi-\frac{1}{2}}\|[\![\boldsymbol{w}-\boldsymbol{w_{h}}]\!]\|_{0, e}\nonumber\\
&\lesssim \|\boldsymbol{w}-\boldsymbol{w_{h}}\|_{\mathrm{dG}}\|\boldsymbol{\psi}-\boldsymbol{\pi}\boldsymbol{\psi}\|_{\mathrm{dG}}\nonumber\\
&\quad+\left(
\sum_{e\in \mathcal{E}^{i}_{h}}h^{2\xi}_{e}\|(2\mu\boldsymbol{\varepsilon}(\boldsymbol{\psi}-\boldsymbol{\pi}\boldsymbol{\psi})+\lambda \mathrm{div}(\boldsymbol{\psi}-\boldsymbol{\pi}\boldsymbol{\psi}) \boldsymbol{I})\boldsymbol{n}\|^{2}_{\xi-\frac{1}{2}, e}
  \right)^{\frac{1}{2}}
  \left(
  \sum_{e\in \mathcal{E}^{i}_{h}}h^{-1}_{e}\|[\![\boldsymbol{w}-\boldsymbol{w_{h}}]\!]\|^{2}_{0, e}
  \right)^{\frac{1}{2}}\nonumber\\
&\lesssim \|\boldsymbol{w}-\boldsymbol{w_{h}}\|_{\mathrm{dG}} \left(\|\boldsymbol{\psi}-\boldsymbol{\pi}\boldsymbol{\psi}\|_{\mathrm{dG}}
  +(\sum_{e\in \mathcal{E}^{i}_{h}}h^{2\xi}_{e}\|(2\mu\boldsymbol{\varepsilon}(\boldsymbol{\psi}-\boldsymbol{\pi}\boldsymbol{\psi})
  +\lambda \mathrm{div}(\boldsymbol{\psi}-\boldsymbol{\pi}\boldsymbol{\psi}) \boldsymbol{I})\boldsymbol{n}\|^{2}_{\xi-\frac{1}{2}, e})^{\frac{1}{2}}\right)\nonumber\\
   &\lesssim h^{r}\|\boldsymbol{f}\|_{0,\partial\Omega}\left(\|\boldsymbol{\psi}-\boldsymbol{\pi}\boldsymbol{\psi}\|_{\mathrm{dG}}
   +(\sum_{e\in \mathcal{E}^{i}_{h}}h^{2\xi}_{e}\|(2\mu\boldsymbol{\varepsilon}(\boldsymbol{\psi}-\boldsymbol{\pi}\boldsymbol{\psi})+\lambda \mathrm{div}(\boldsymbol{\psi}-\boldsymbol{\pi}\boldsymbol{\psi}) \boldsymbol{I})\boldsymbol{n}\|^{2}_{\xi-\frac{1}{2}, e})^{\frac{1}{2}}\right).
 \end{align}
 Using (\ref{eq26}) with $\xi= r+\frac{1}{2}(\frac{1}{2}-r)\doteq r+\delta$, then by the same proof method of (\ref{eq40}) we can derive
\begin{align}\label{eq45}
\nonumber&\left( \sum_{e\in \mathcal{E}^{i}_{h}}h_{e}^{2\xi}\| (2\mu\boldsymbol{\varepsilon}(\boldsymbol{\psi}-\boldsymbol{\pi}\boldsymbol{\psi})
  +\lambda \mathrm{div}(\boldsymbol{\psi}-\boldsymbol{\pi}\boldsymbol{\psi})\boldsymbol{I}) \boldsymbol{n}\|^{2}_{\xi-\frac{1}{2}, e}   \right)^{\frac{1}{2}}\\
&\lesssim h^{\xi-\delta}\left(
  2\mu \|\boldsymbol{\psi} \|_{1+\xi}+\lambda \|\mathrm{div} \boldsymbol{\psi}\|_{\xi}
 \right)
\lesssim h^{r}\|\boldsymbol{g} \|_{0,\partial\Omega}.
\end{align}
For the solution $\boldsymbol{\psi}$ of the auxiliary problem, using a similar proof technique of (\ref{eq38}) we can prove
\begin{align}\label{eq46}
 \|\boldsymbol{\psi}-\boldsymbol{\pi}\boldsymbol{\psi}\|_{\mathrm{dG}}
 \lesssim
h^{\xi}\|\boldsymbol{g} \|_{0,\partial\Omega}.
\end{align}

\noindent Substituting (\ref{eq45}) and (\ref{eq46}) into (\ref{eq44}), and using the Riesz representation theorem, we get (\ref{eq42}).\\

\indent When $\boldsymbol{w}\in \boldsymbol{H}^{1+t}(\Omega)$ $(\frac{1}{2}<t\leq k)$, similar to the above analysis,  using the fact that $[\![(\boldsymbol{\psi}-\boldsymbol{\pi}\boldsymbol{\psi})\cdot \mathbf{n}]\!]=0$ on inner face $e$, Cauchy-Schwarz inequality, the definitions of $\|\cdot\|_{\mathrm{dG}}$ and $\|\cdot\|_{h}$, (\ref{eq26}) with $\xi=r+\frac{1}{2}(\frac{1}{2}-r)\doteq r+\delta$, the inverse inequality, (\ref{eq29}), (\ref{eq39}) and (\ref{eq40}), we have
\begin{align*}
\nonumber &b_{h}(\boldsymbol{w}-\boldsymbol{w_{h}},\boldsymbol{g}) =a_{h}(\boldsymbol{w}-\boldsymbol{w_{h}},\boldsymbol{\psi}-\boldsymbol{\pi}\boldsymbol{\psi})\\
&\lesssim  \|\boldsymbol{w}-\boldsymbol{w_{h}}\|_{\mathrm{dG}}\|\boldsymbol{\psi}-\boldsymbol{\pi}\boldsymbol{\psi}\|_{\mathrm{dG}}
-\sum_{e\in \mathcal{E}^{i}_{h}}\int_{e}
    \{( 2\mu\boldsymbol{\varepsilon}(\boldsymbol{\psi}-\boldsymbol{\pi}\boldsymbol{\psi})
  +\lambda\mathrm{div}(\boldsymbol{\psi}-\boldsymbol{\pi}\boldsymbol{\psi})\boldsymbol{I})\boldsymbol{n}\}\cdot[\![\boldsymbol{w-w_{h}}]\!]ds\\
&\lesssim \|\boldsymbol{w}-\boldsymbol{w_{h}}\|_{\mathrm{dG}}\|\boldsymbol{\psi}-\boldsymbol{\pi}\boldsymbol{\psi}\|_{\mathrm{dG}}\nonumber\\
&\quad +\sum_{e\in \mathcal{E}^{i}_{h}}
      \|(2\mu \boldsymbol{\varepsilon}(\boldsymbol{\psi}-\boldsymbol{\pi}\boldsymbol{\psi})+\lambda \mathrm{div}(\boldsymbol{\psi}-\boldsymbol{\pi}\boldsymbol{\psi})\boldsymbol{I})\boldsymbol{n}\|_{\xi-\frac{1}{2},e}
      \|[\![\boldsymbol{w-w_{h}}]\!]\|_{\frac{1}{2}-\xi,e}
      \\
&\lesssim\|\boldsymbol{w}-\boldsymbol{w_{h}}\|_{\mathrm{dG}}\|\boldsymbol{\psi}-\boldsymbol{\pi}\boldsymbol{\psi}\|_{\mathrm{dG}}\nonumber\\
&\quad+\sum_{e\in \mathcal{E}^{i}_{h}}\| (2\mu\boldsymbol{\varepsilon}(\boldsymbol{\psi}-\boldsymbol{\pi}\boldsymbol{\psi})
+\lambda \mathrm{div}(\boldsymbol{\psi}-\boldsymbol{\pi}\boldsymbol{\psi} )\boldsymbol{I})\boldsymbol{n}\|_{\xi-\frac{1}{2}, e} h_{e}^{\xi-\frac{1}{2}}\|[\![\boldsymbol{w}-\boldsymbol{w_{h}}]\!]\|_{0, e}\nonumber\\
&\lesssim \|\boldsymbol{w}-\boldsymbol{w_{h}}\|_{\mathrm{dG}}\|\boldsymbol{\psi}-\boldsymbol{\pi}\boldsymbol{\psi}\|_{\mathrm{dG}}\nonumber\\
&\quad+\left(
\sum_{e\in \mathcal{E}^{i}_{h}}h^{2\xi}_{e}\|(2\mu\boldsymbol{\varepsilon}(\boldsymbol{\psi}-\boldsymbol{\pi}\boldsymbol{\psi})+\lambda \mathrm{div}(\boldsymbol{\psi}-\boldsymbol{\pi}\boldsymbol{\psi}) \boldsymbol{I})\boldsymbol{n}\|^{2}_{\xi-\frac{1}{2}, e}
  \right)^{\frac{1}{2}}
  \left(
  \sum_{e\in \mathcal{E}^{i}_{h}}h^{-1}_{e}\|[\![\boldsymbol{w}-\boldsymbol{w_{h}}]\!]\|^{2}_{0, e}
  \right)^{\frac{1}{2}}\nonumber\\
&\lesssim \|\boldsymbol{w}-\boldsymbol{w_{h}}\|_{h} \left(\|\boldsymbol{\psi}-\boldsymbol{\pi}\boldsymbol{\psi}\|_{\mathrm{dG}}
  +(\sum_{e\in \mathcal{E}^{i}_{h}}h^{2\xi}_{e}\|(2\mu\boldsymbol{\varepsilon}(\boldsymbol{\psi}-\boldsymbol{\pi}\boldsymbol{\psi})
  +\lambda \mathrm{div}(\boldsymbol{\psi}-\boldsymbol{\pi}\boldsymbol{\psi}) \boldsymbol{I})\boldsymbol{n}\|^{2}_{\xi-\frac{1}{2}, e})^{\frac{1}{2}}\right)\nonumber\\
   &\lesssim h^{t}\left(\sqrt{2\mu}\|\boldsymbol{w}\|_{1+t}+\sqrt{\lambda}\|\mathrm{div} \boldsymbol{w}\|_{t}\right)
   \left(h^{\xi}\|\boldsymbol{g}\|_{0,\partial\Omega}
   +h^{r}\|\boldsymbol{g}\|_{0,\partial\Omega}\right)\\
&\lesssim h^{r+t}\left(\sqrt{2\mu}\|\boldsymbol{w}\|_{1+t}+\sqrt{\lambda}\|\mathrm{div} \boldsymbol{w}\|_{t}\right)
\|\boldsymbol{g}\|_{0,\partial\Omega},
\end{align*}
which together with the Riesz representation theorem yields (\ref{eq43}).
\end{proof}

Define the norm of operator $T$ as follows:
\begin{align*}
&\|T\|_{0,\partial\Omega}
 =\mathop{sup}\limits_{0\neq \boldsymbol{f}\in \boldsymbol{L}^{2}(\partial\Omega)}\frac{\|T\boldsymbol{f}\|_{0,\partial\Omega}}{\|\boldsymbol{f}\|_{0,\partial\Omega}}.
\end{align*}

\begin{lemma}
The operator $T$ converges to $T_{h}$ in terms of
 \begin{equation}\label{eq47}
  \|T_{h}-T\|_{0,\partial\Omega}\rightarrow 0~as~h\rightarrow 0.
 \end{equation}
\end{lemma}

\begin{proof}
 Together with the operators standard arguments and (\ref{eq42}), we deduce that
\begin{align*}
 &\|T_{h}-T\|_{0,\partial\Omega}
 =\mathop{sup}\limits_{0\neq \boldsymbol{f}\in \boldsymbol{L}^{2}(\partial\Omega)}\frac{\|T_{h}\boldsymbol{f}-T\boldsymbol{f}\|_{0,\partial\Omega}}{\|\boldsymbol{f}\|_{0,\partial\Omega}}
 =\mathop{sup}\limits_{0\neq \boldsymbol{f}\in \boldsymbol{L}^{2}(\partial\Omega)}\frac{\|A_{h}\boldsymbol{f}-A\boldsymbol{f}\|_{0,\partial\Omega }}{\|\boldsymbol{f}\|_{0,\partial\Omega}}\\
&\leq\mathop{sup}\limits_{0\neq \boldsymbol{f}\in \boldsymbol{L}^{2}(\partial\Omega)}\frac{Ch^{2r}\|\boldsymbol{f}\|_{0,\partial\Omega }}{\|\boldsymbol{f}\|_{0,\partial\Omega}}\leq Ch^{2r}\rightarrow0~(h\rightarrow0),
\end{align*}
i.e., (\ref{eq47}) holds. Since $T_{h}$ is a finite rank operator, $T$ is a compact operator.
 The proof is completed.
\end{proof}

\indent  Let $M(\rho)$ be the space of  eigenvectors of (\ref{eq2}) associated with eigenvalue $\rho$, and $M_{h}(\rho)$ be the direct sum of the generalized eigenspace of (\ref{eq3}) associated with $\rho_{h}$ that converge to $\rho$. Denote
\begin{align*}
\|(T-T_{h}) \mid _{M(\rho)}\|_{0,\partial\Omega}
 =\mathop{sup}_{\boldsymbol{f}\in M(\rho),\|\boldsymbol{f}\|_{0,\partial\Omega}=1}\|T\boldsymbol{f}-T_{h}\boldsymbol{f}\|_{0,\partial\Omega}.
\end{align*}

Thanks to \cite{Babuska1991}, we can prove the following theorem.

\begin{theorem}\label{thm3.10}
Suppose that the regularity estimate (\ref{eq22}) is valid and assume $M(\rho)\subset \boldsymbol{H}^{1+s}(\Omega)$
($ s\geq1 $),
 $t=min\{k,s\}$, then
\begin{align}\label{eq48}
\mid \rho_{h} -\rho  \mid\lesssim h^{2t}.
\end{align}
Let $\boldsymbol{u_{h}}\in M_{h}(\rho)$ be the eigenfunction of (\ref{eq3}), then there exists $\boldsymbol{u}\in M(\rho)$
such that
\begin{align}
\label{eq49}& \|\boldsymbol{u}-\boldsymbol{u_{h}}\|_{0,\partial\Omega}\lesssim h^{r+t},\\
\label{eq50}&  \|\boldsymbol{u}-\boldsymbol{u_{h}}\|_{h}\lesssim h^{t},
\end{align}
where $r<\frac{1}{2}$ and $r$ can be close to $\frac{1}{2}$ arbitrarily.
\end{theorem}

\begin{proof}
Let $\rho$  and  $\rho_{h}$ be the $mth$ eigenvalues of (\ref{eq2}) and (\ref{eq3}), respectively, and dim$M(\rho)=q$. From Theorem 7.3 in \cite{Babuska1991} we get
 \begin{align}\label{eq51}
\mid \rho-\rho_{h} \mid \lesssim \sum^{m+q-1}_{i,j=m} \mid b_{h}((T-T_{h})\boldsymbol{\varphi}_{i},\boldsymbol{\varphi}_{j})\mid+\|(T-T_{h}) \mid _{M(\rho)}\|^{2}_{0,\partial\Omega},
 \end{align}
 where $\boldsymbol{\varphi}_{m},\cdots,\boldsymbol{\varphi}_{m+q-1}$ are the basis functions for $M(\rho)$. Then, from (\ref{eq25}), (\ref{eq24}), (\ref{eq10}) and (\ref{eq29}), we derive
 \begin{align}\label{eq52}
 \nonumber & b_{h}((T-T_{h})\boldsymbol{\varphi}_{i},\boldsymbol{\varphi}_{j})= b_{h}((A-A_{h})\boldsymbol{\varphi}_{i},\boldsymbol{\varphi}_{j})=a_{h}(A\boldsymbol{\varphi}_{i}-A_{h}\boldsymbol{\varphi}_{i},A\boldsymbol{\varphi}_{j})\\
  &=a_{h}(A\boldsymbol{\varphi}_{i}-A_{h}\boldsymbol{\varphi}_{i},A\boldsymbol{\varphi}_{j}-A_{h}\boldsymbol{\varphi}_{j})
\lesssim \|A\boldsymbol{\varphi}_{i}-A_{h}\boldsymbol{\varphi}_{i}\|_{h}\|A\boldsymbol{\varphi}_{j}-A_{h}\boldsymbol{\varphi}_{j}\|_{h}
 \lesssim h^{2t}.
 \end{align}

Noting that $A\boldsymbol{f}=\boldsymbol{w}$, $A_{h}\boldsymbol{f}=\boldsymbol{w_{h}}$ and  using (\ref{eq43}), we have
\begin{align}\label{eq53}
\nonumber &\|(T-T_{h}) \mid _{M(\rho)}\|_{0,\partial\Omega}
 =\mathop{sup}_{\boldsymbol{f}\in M(\rho),\|\boldsymbol{f}\|_{0,\partial\Omega}=1}\|T\boldsymbol{f}-T_{h}\boldsymbol{f}\|_{0,\partial\Omega}\\
&\lesssim\mathop{sup}_{\boldsymbol{f}\in M(\rho),\|\boldsymbol{f}\|_{0,\partial\Omega}=1}h^{r+t}\|\boldsymbol{f}\|_{0,\partial\Omega}
\lesssim h^{r+t}
\end{align}
Substituting (\ref{eq52}) and (\ref{eq53}) into (\ref{eq51}), we get (\ref{eq48}).

Since $ \|T_{h}-T\|_{0,\partial\Omega}\rightarrow 0$, from the spectral approximation theory (see Theorems 7.3 and 7.4 in \cite{Babuska1991}) we know that there exists $\boldsymbol{u}\in M(\rho)$ such that
\begin{align}\label{eq54}
\|\boldsymbol{u}-\boldsymbol{u_{h}}\|_{0,\partial\Omega}\leq C\|(T-T_{h}) \mid _{M(\rho)}\|_{0,\partial\Omega}.
\end{align}
Then (\ref{eq49}) follows directly from (\ref{eq54}) and (\ref{eq53}).

Since $\boldsymbol{u}=\rho A\boldsymbol{u}$ and $\boldsymbol{u_{h}}=\rho_{h} A_{h}\boldsymbol{u_{h}}$, using the triangular inequality,
(\ref{eq41}), (\ref{eq48}), (\ref{eq49}) and (\ref{eq29}), we deduce
 \begin{align*}
&\|\boldsymbol{u_{h}}-\boldsymbol{u}\|_{h}=\|\rho_{h} A_{h}\boldsymbol{u_{h}}-\rho A\boldsymbol{u}\|_{h}
             \lesssim\|\rho_{h} A_{h}\boldsymbol{u_{h}}-\rho A_{h}\boldsymbol{u}\|_{h}+\|\rho A_{h}\boldsymbol{u}-\rho A\boldsymbol{u}\|_{h}\\
&  \lesssim \|\rho_{h} \boldsymbol{u_{h}}-\rho \boldsymbol{u}\|_{0,\partial\Omega}+\|A_{h}\boldsymbol{u}-A\boldsymbol{u}\|_{h}\\
&  \lesssim  \mid \rho_{h} -\rho \mid + \|\boldsymbol{u_{h}}-\boldsymbol{u}\|_{0,\partial\Omega}+\|A_{h}\boldsymbol{u}-A\boldsymbol{u}\|_{h}
\lesssim h^{2t} + h^{r+t}+h^{t}
\lesssim h^{t},
 \end{align*}
i.e., (\ref{eq50}) is valid.
\end{proof}

\section{Numerical experiments}

In this section, we will report some numerical experiments about the approximations of the first 7 non-zero eigenvalues of~(\ref{eq2}) on four domains: the unit square $\Omega_S:=(0,1)^{2}$, the unit disk $\Omega_D$  with center at 0 and radius 1, the L-shaped domain $\Omega_L:=(-1,1)^{2}\backslash[0,1)^{2}$, and the unit cube  $\Omega_C:=(0,1)^{3}$. We take $\mu=1$ and take different Lam\'{e} parameters $\lambda\in\{1,10,10^{2},10^{3},10^{4},10^{5},10^{6}\}$ on each domain to observe the effect of $\lambda$, and set the penalty parameter $\gamma_{\mu}=\gamma_{\lambda}=10, 40, 90$. We solve the problem in Python on a computer  with 2.10GHZ CPU and 64GB RAM. The numerical results are  created using scikit-fem \cite{Gustafsson2020} which relies heavily on NumPy \cite{Harris2020} and SciPy \cite{Virtanen2020}.

\indent In the experiments we use the linear DG element (P1), the quadratic DG element (P2) and the cubic  DG element (P3) to compute the approximate eigenvalues on a sequence of shape-regular meshes, and we take the results calculated by P3 elements on the smallest meshes as possible as we can implement as the reference values.

\indent Figs. 1-2 show the convergence orders of the eigenvalues obtained by P2 elements with $\lambda=1$ and $\lambda=10000$ on four domains, respectively. It is well-known that the convergence order of approximate eigenvalues depends on the smoothness of eigenfunctions and the degree of polynomials that we used (see (\ref{eq46})). When the eigenfunction is smooth enough or belongs to $\boldsymbol{H}^{1+k}(\Omega)$, the corresponding approximate eigenvalues calculated on uniform meshes can achieve the optimal convergence order $O(dof^{-\frac{2k}{d}})$ (see \cite{Ciarlet1978}), at which point the error curve is parallel to the line with slope $-\frac{2k}{d}$. From the top-left corner of Fig. 1 we can see that the error curves of the 4th and the 7th approximate eigenvalue on $\Omega_{S}$ are parallel to the line with slope -2 when P2 elements are used, which indicates that the 4th and the 7th eigenvalue on $\Omega_{S}$ reach the optimal convergence order, while the convergence orders of the other five eigenvalues are not optimal. We think this is because different eigenfunctions have different smoothness in general (see, e.g., pages 736-737 in \cite{Babuska1991}), thus different eigenvalues have different convergence rates.

It can also be observed from the bottom-left corner of Fig.1 that when using P2 element, most of the error curves are parallel to the line whose slope is -1 but not -2, which implies that on $\Omega_{L}$  the eigenfunctions are less smooth and the eigenvalues cannot reach the optimal convergence order.

From the top-right corner of Fig. 1 we find that the convergence order of the first 7 eigenvalues on $\Omega_D$ only reaches -1 but not -2 when using P2 element. The reason is that in computation the substitution of straight-sided triangle for the curved edge of the disk will cause errors and  result in the loss of convergence order.

From  Fig. 2, when $\lambda=10^{4}$, we can see that the first 7 approximate eigenvalues all converge and their error curves are almost parallel to the line with slope -1, which means they all reach the convergence rate of order 1. In the case of $\lambda=10^{6}$, the
first 7 non-zero eigenvalues still converge, and due to the space limitation we do not present the error curves.

Figs. 3-4 show the influence of Lam\'{e} parameter $\lambda$ for the first eigenvalue when using P1, P2 and P3 elements in different domains. It can be seen from Figs. 3-4 that when $\lambda$ increases the error curves coincide almost exactly, which means the change of $\lambda$ has no effect on the eigenvalues, in other word, the SIPG method is robust with respect to nearly incompressible materials. Both theoretical analysis and numerical experiments show that the SIPG method is an efficient and robust approach to solve the Steklov-Lam\'{e} eigenproblem in linear elasticity.

In addition, we also solve this problem by using conforming finite element methods (conforming FEM) and the results are shown in Figs. 5-6.
From Fig. 5 we can see that when computing the approximations on the convex domain $\Omega_S$ and the non-convex domain  $\Omega_L$, the error curves by conforming FEM and DGFEM are almost parallel. In the left column of Fig. 5, the error curves are all parallel to the line with slope -1 while those in the right column are parallel to the line with slope greater than -1. It indicates that for Steklov-Lam\'{e} eigenproblem, the eigenfunctions' smoothness is inherited by the problem configuration and independent of computing method.
We also notice that the results by using conforming P2 elements are locking-free. The results obtained by conforming P1 elements on different meshes have different locking properties. Usually volumetric locking causes artificial stiffening which we think will also increase the frequency. It often happens for low order polynomial and high values of parameter $\lambda$. In our experiments, for large degrees of freedom (around 250,000 or more DOFs ), the conforming method is always locking-free although theoretically we cannot prove this yet. However, an interesting thing in numerical experiments is that different meshes lead to different locking properties when using P1 conforming element. When the degrees of freedom is less, the bisect mesh can be used to obtain the results of locking-free, but on uniform refined mesh the results are locked. Whether the analysis in this paper can be extended to general polygonal and polyhedral meshes and locking-free performance on general grids are the goal of our future research. \\

\begin{figure}
\begin{tabular}{llll}

\includegraphics[width=0.5\textwidth]{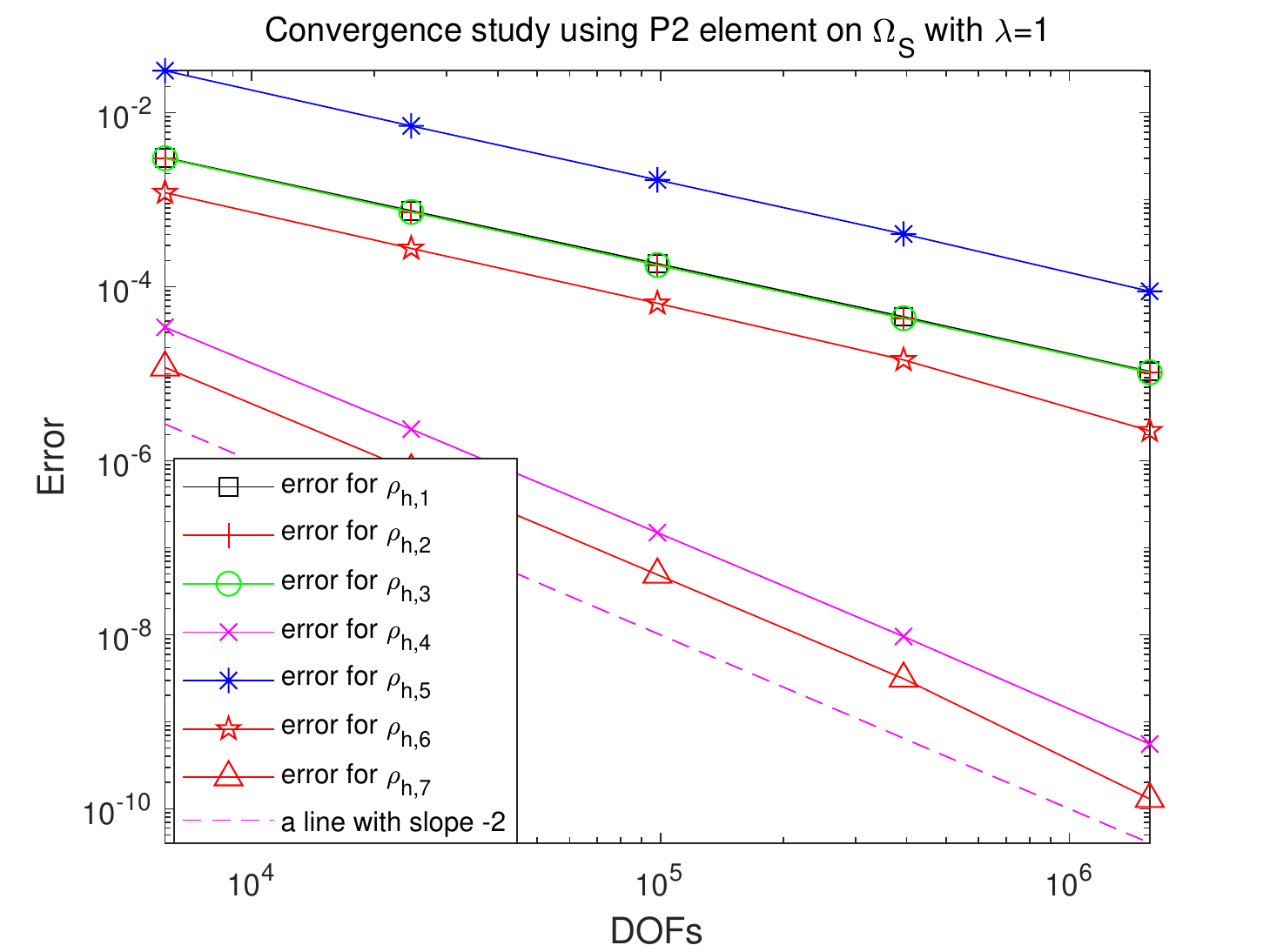}&\includegraphics[width=0.5\textwidth]{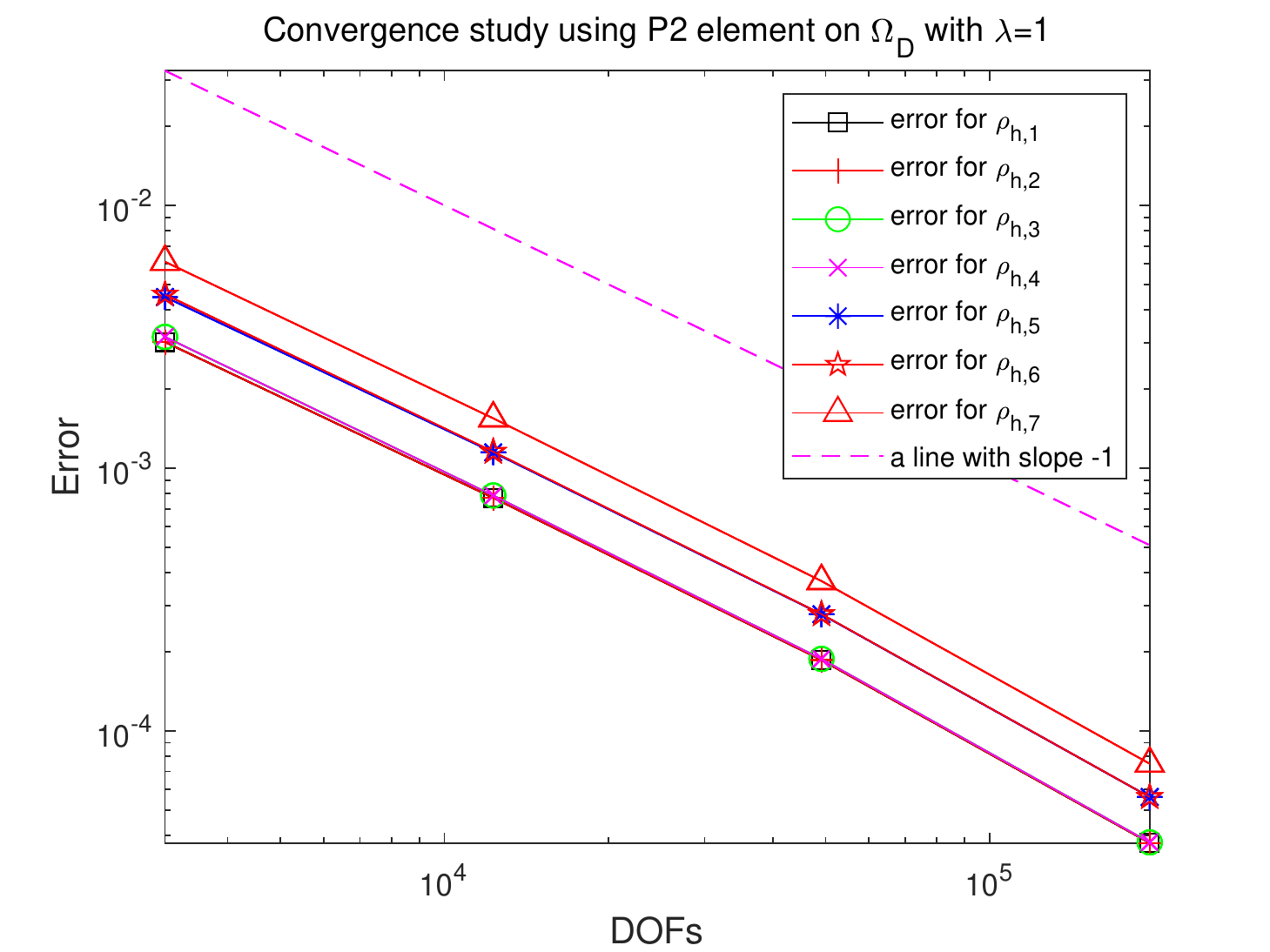}\\
\includegraphics[width=0.5\textwidth]{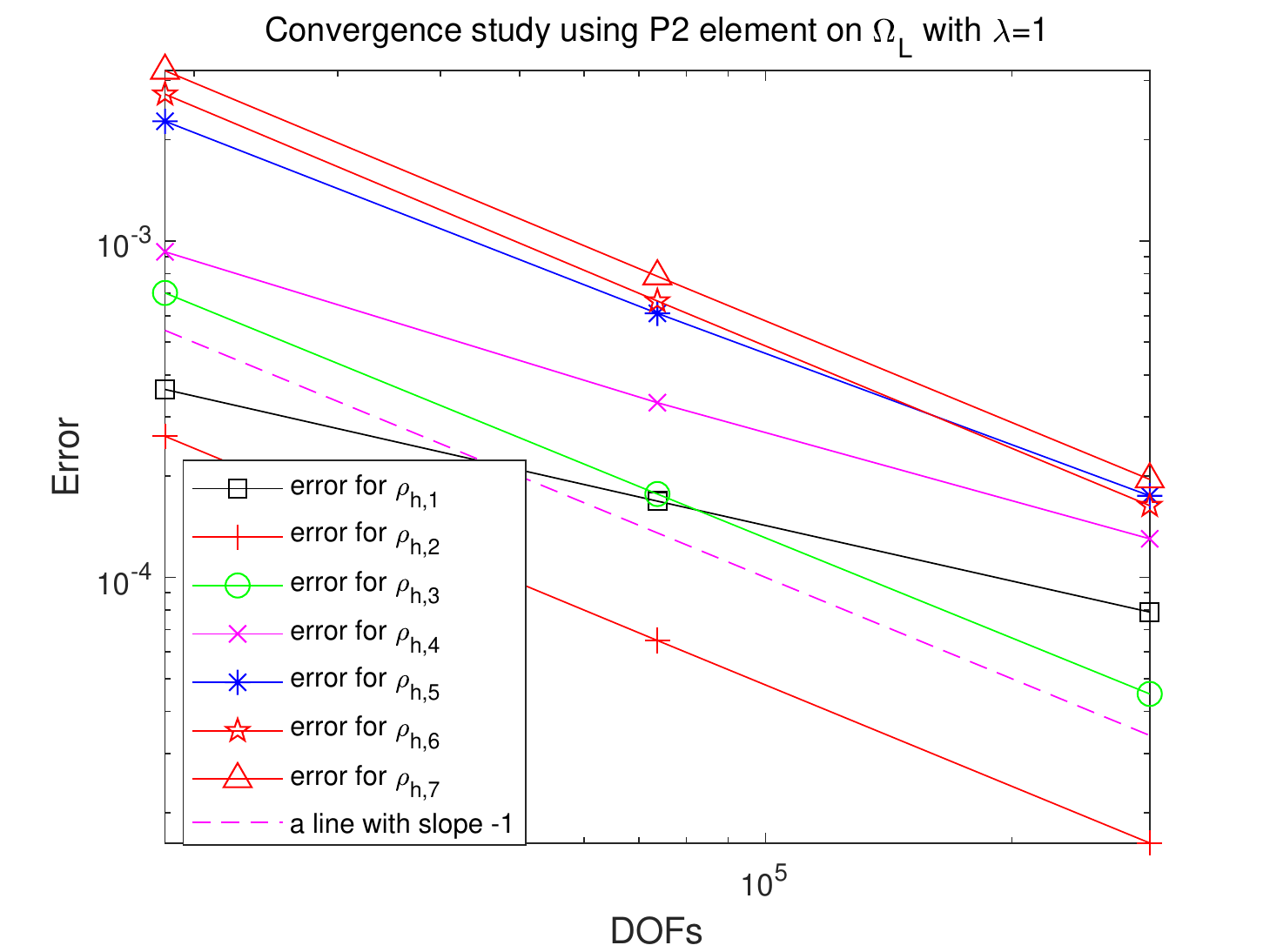}&\includegraphics[width=0.5\textwidth]{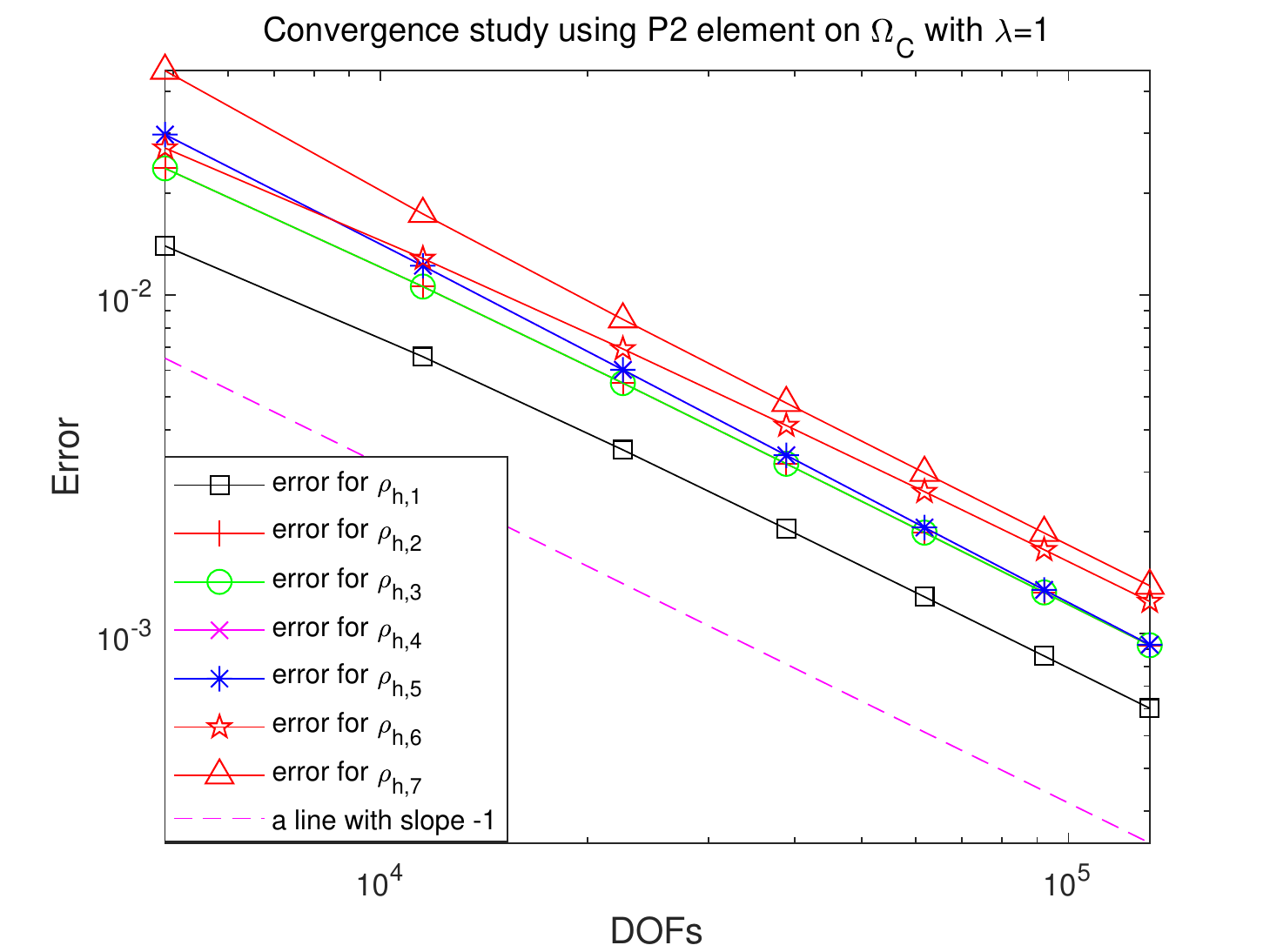}
\end{tabular}
\\{\textrm{\small\bf Figure 1. The convergence study for the first 7 non-zero eigenvalues on $\Omega_{S}$, $\Omega_{D}$, $\Omega_{L}$ and $\Omega_{C}$ using the P2 element when $\lambda=1$.}}
\end{figure}

\begin{figure}
\begin{tabular}{llll}
\includegraphics[width=0.5\textwidth]{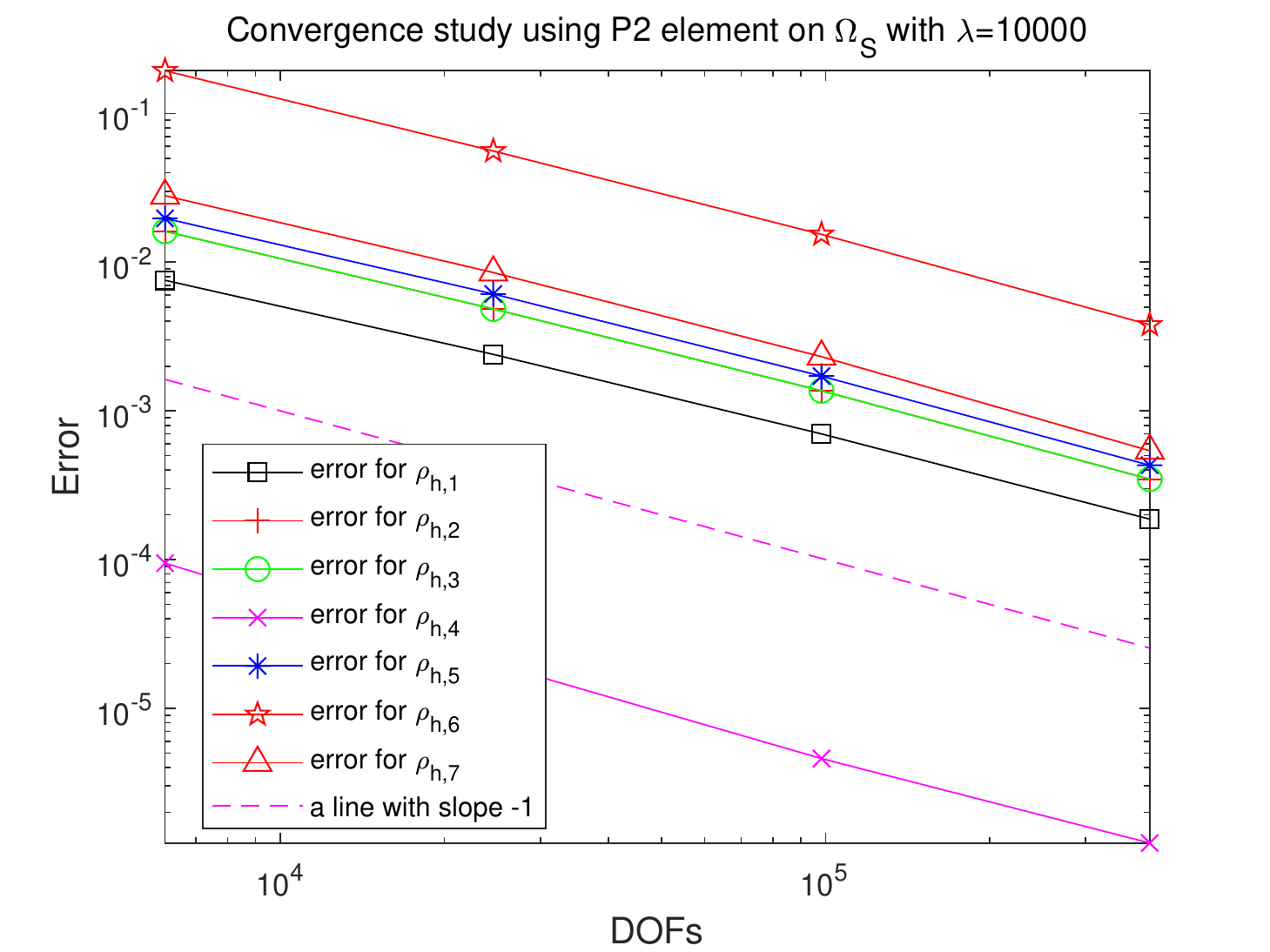}&\includegraphics[width=0.5\textwidth]{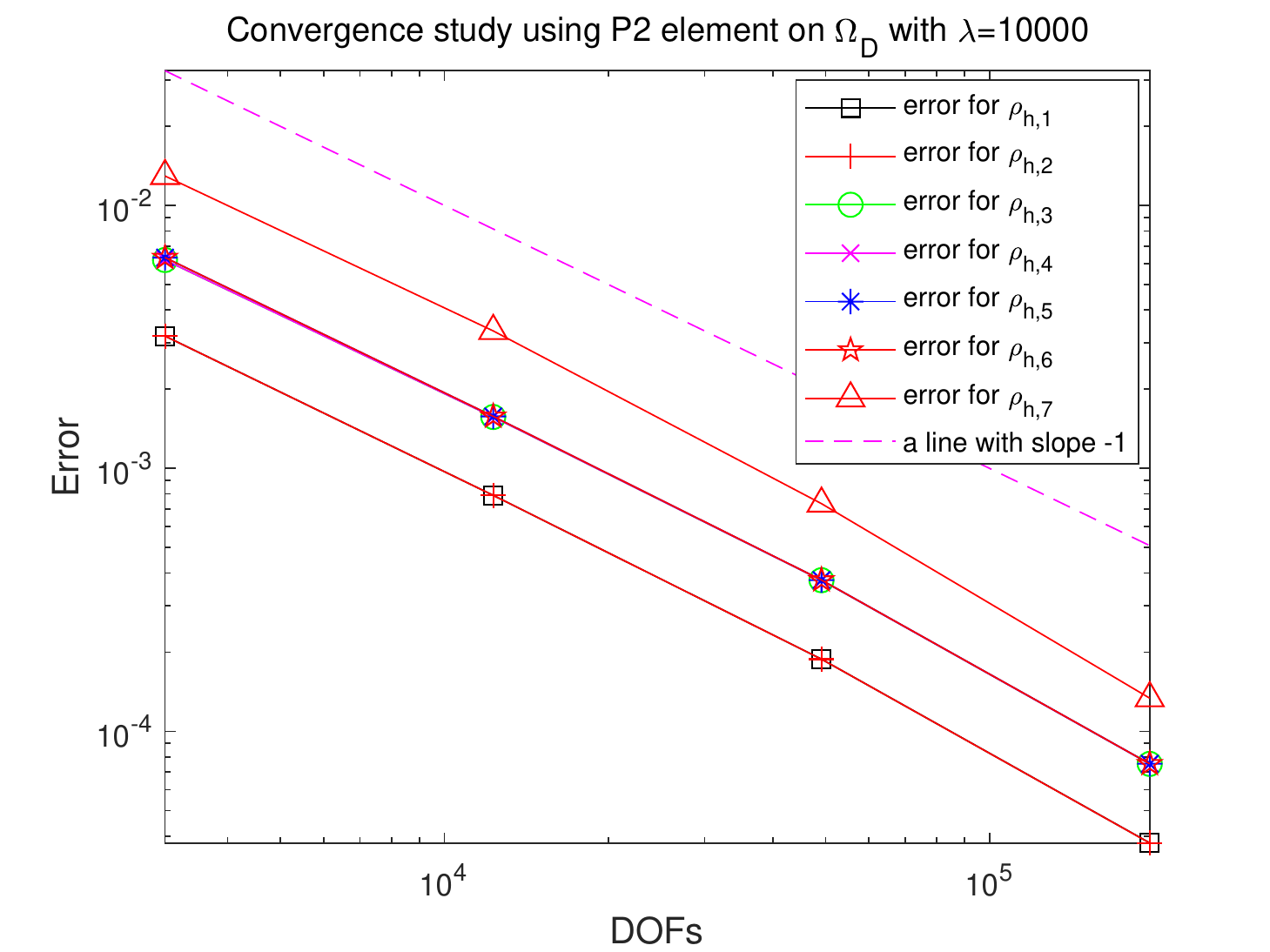}\\
\includegraphics[width=0.5\textwidth]{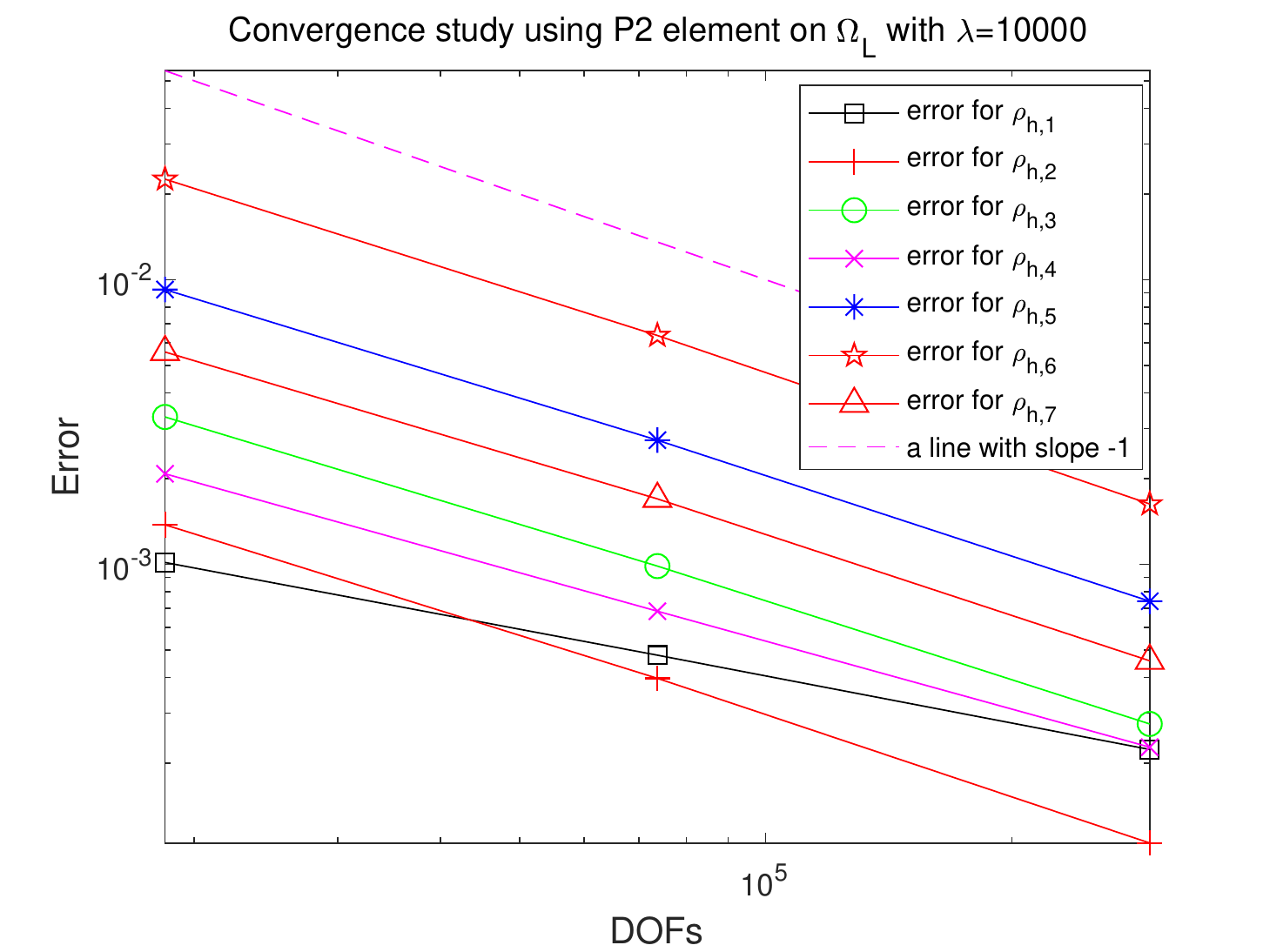}&\includegraphics[width=0.5\textwidth]{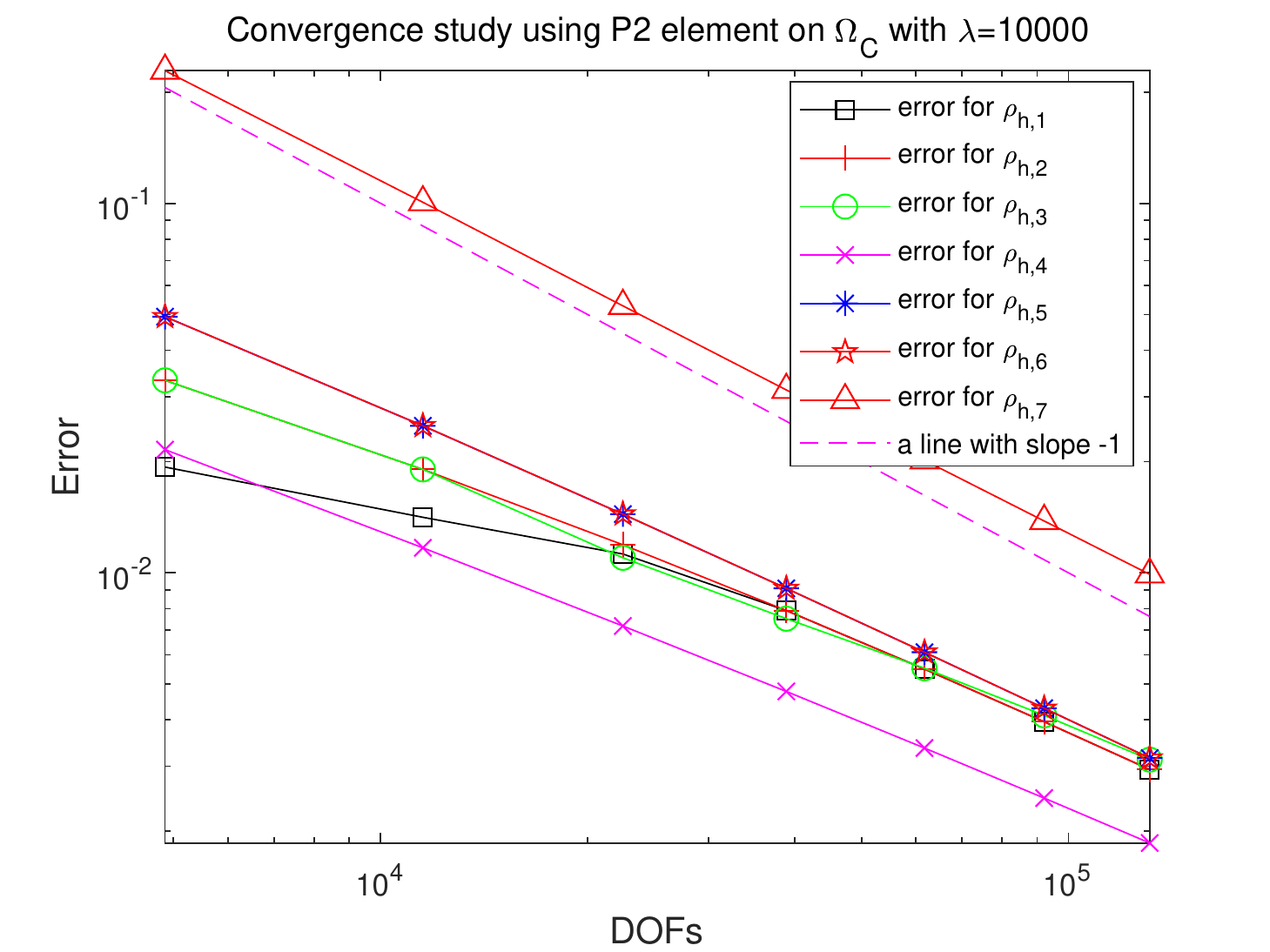}
\end{tabular}
\\{\textrm{\small\bf Figure 2. The convergence study for the first 7 non-zero eigenvalues on $\Omega_{S}$, $\Omega_{D}$, $\Omega_{L}$ and $\Omega_{C}$ using the P2 element when $\lambda=10000$.}}
\end{figure}

\begin{figure}
\begin{tabular}{llll}
\includegraphics[width=0.5\textwidth]{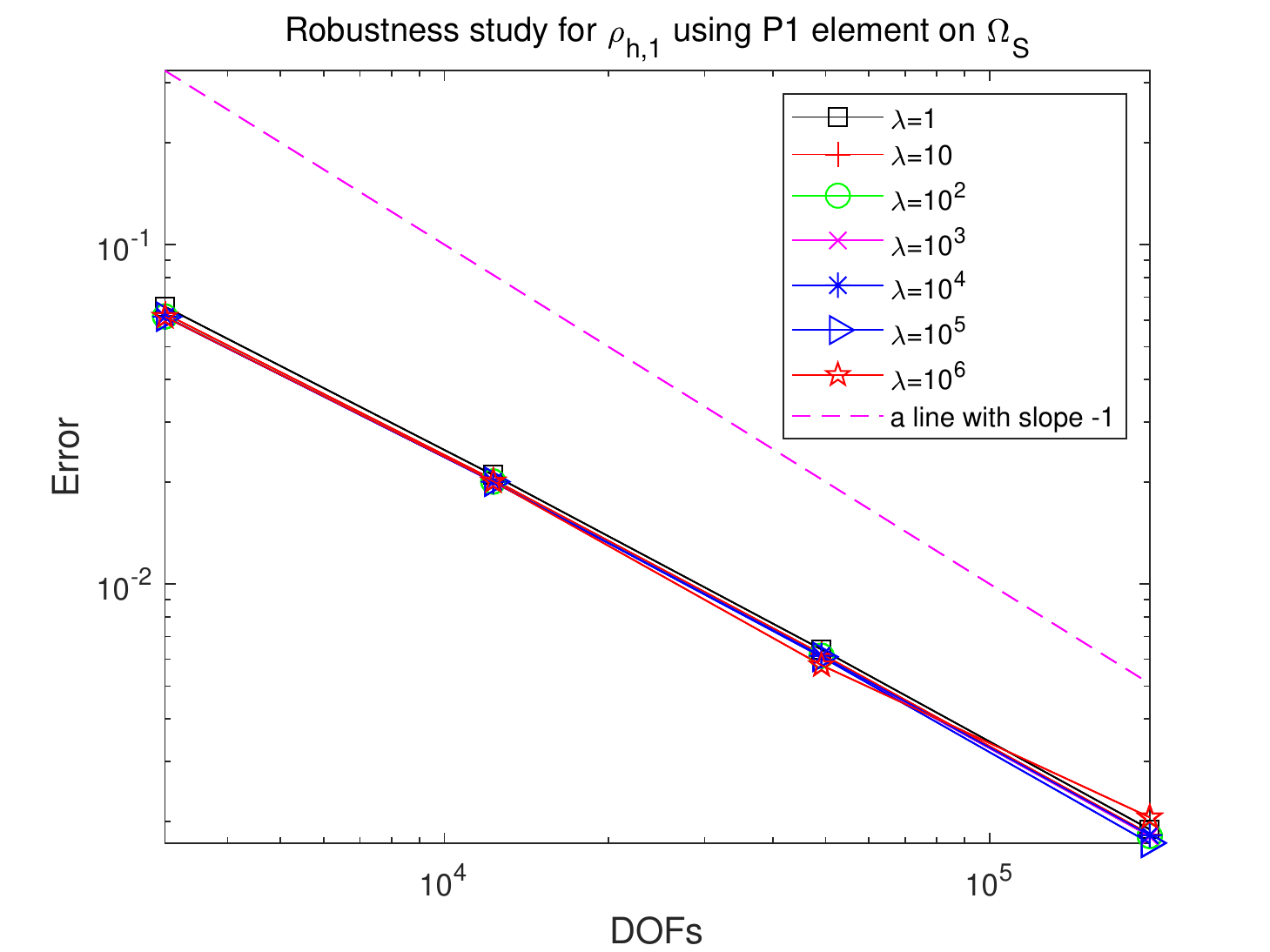}&\includegraphics[width=0.5\textwidth]{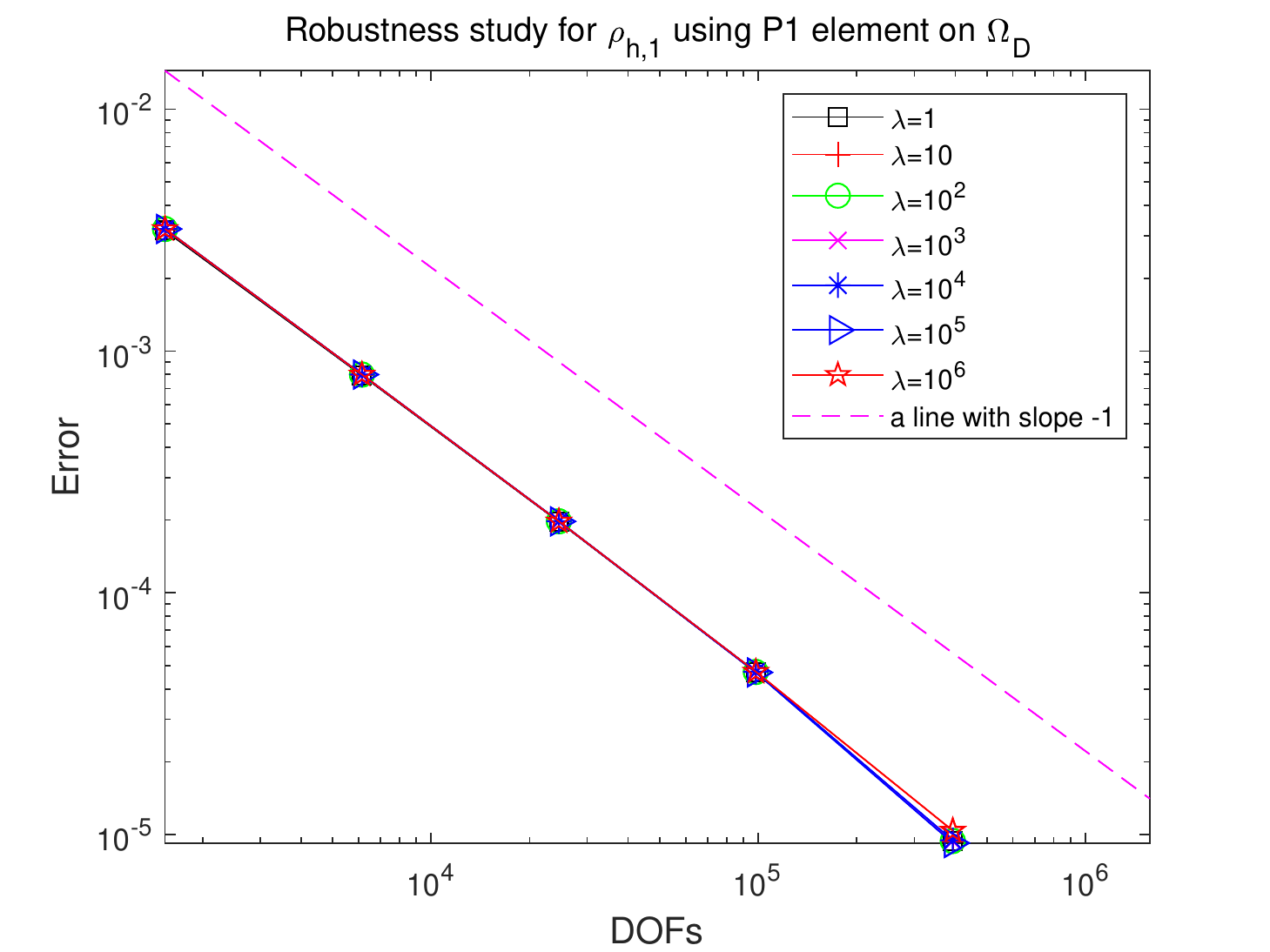}\\
\includegraphics[width=0.5\textwidth]{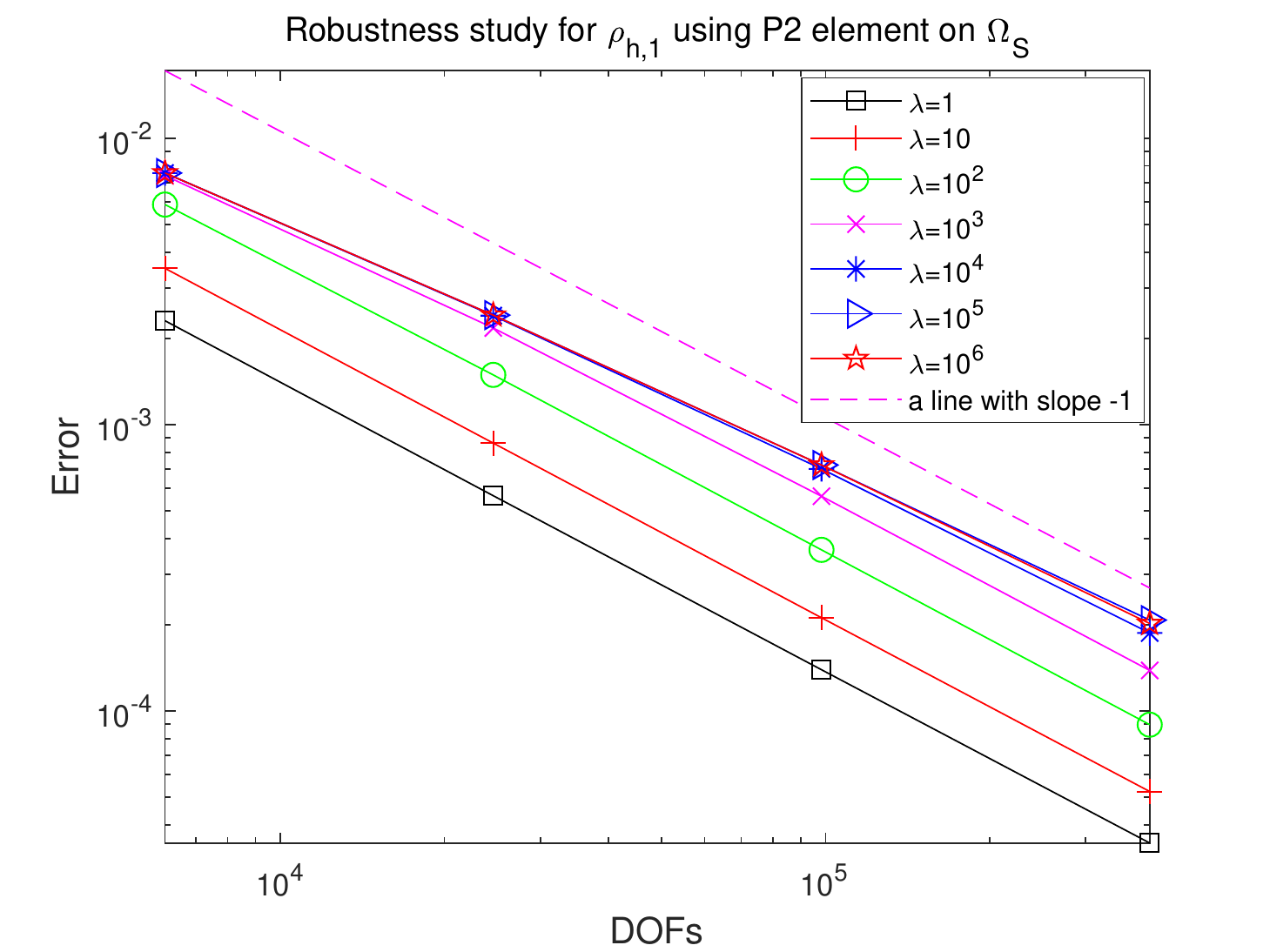}&\includegraphics[width=0.5\textwidth]{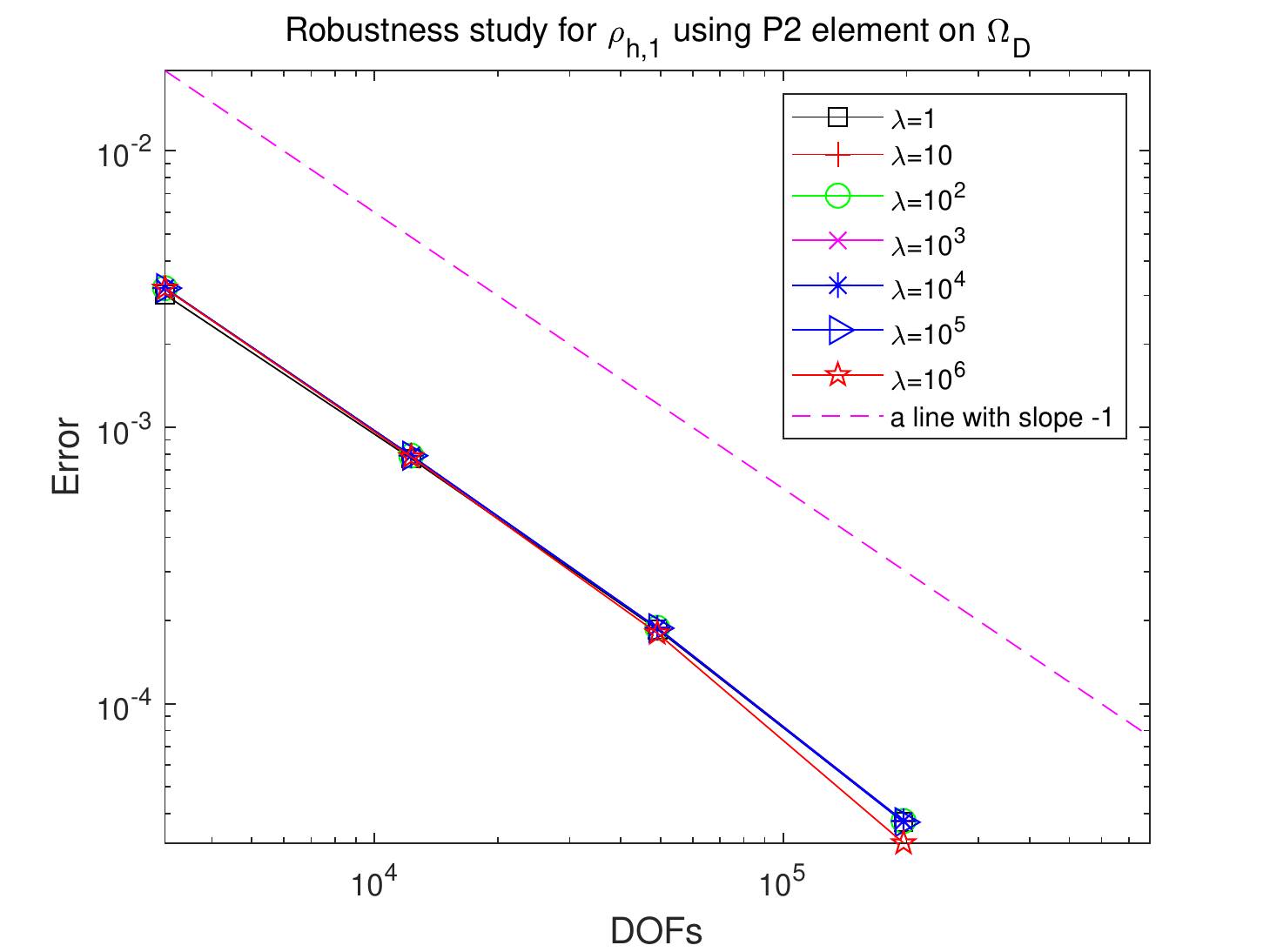}\\
\includegraphics[width=0.5\textwidth]{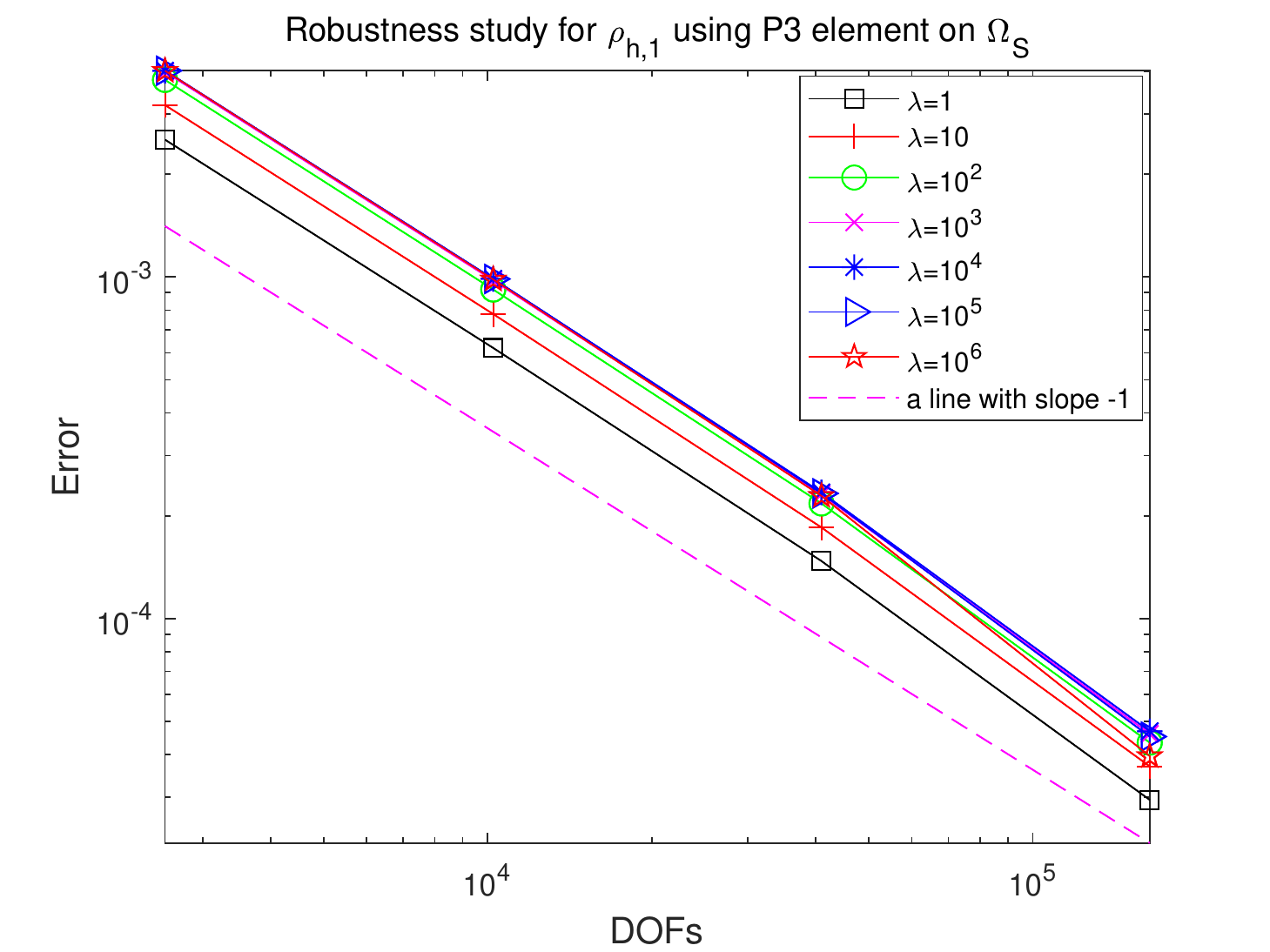}&\includegraphics[width=0.5\textwidth]{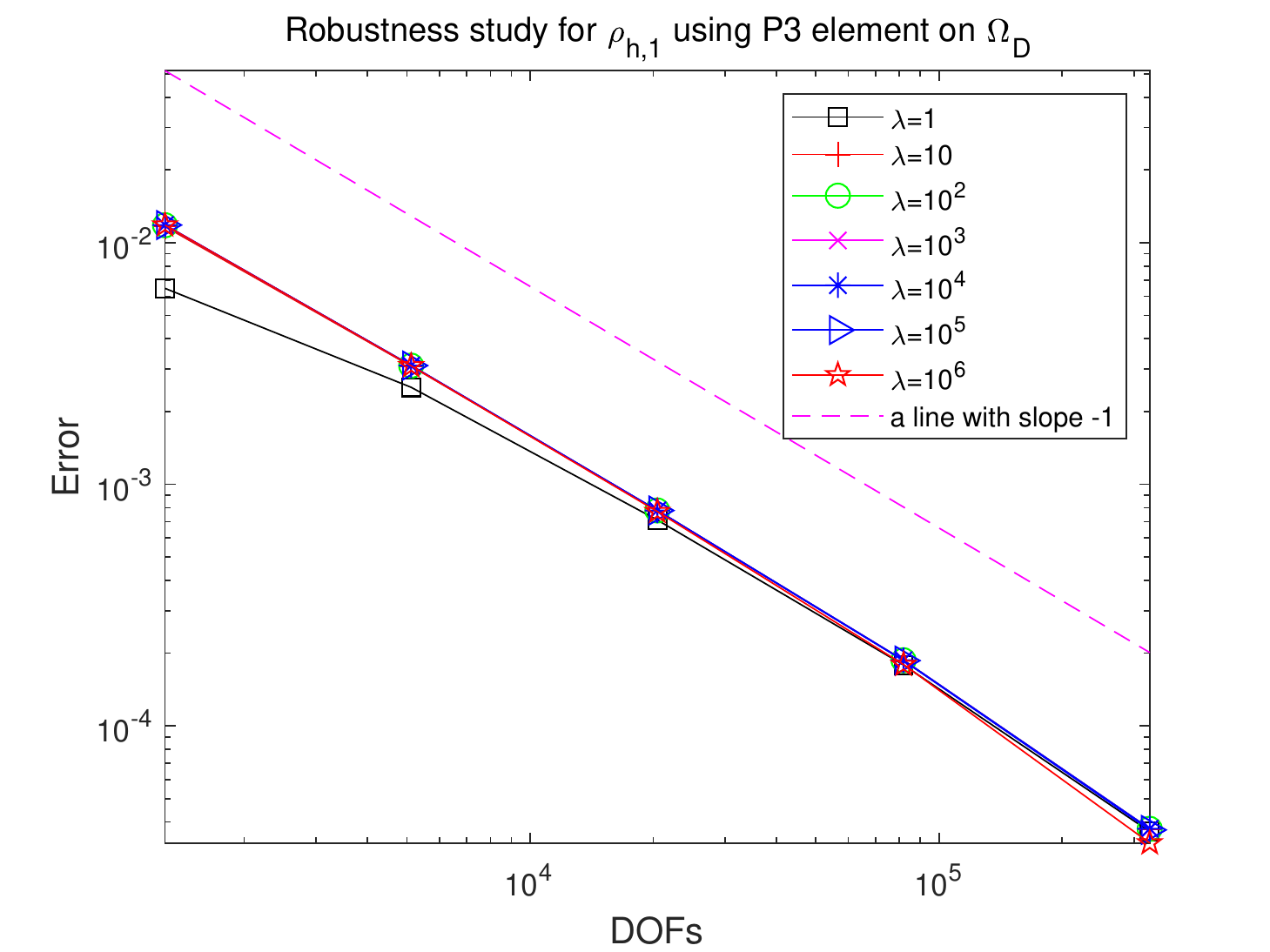}
\end{tabular}
\\{\textrm{\small\bf Figure 3. The robustness study for $\rho_{h,1}$ on $\Omega_{S}$ (left column) and $\Omega_{D}$ (right column)  using the P1, P2, P3 elements.}}
\end{figure}

\begin{figure}
\begin{tabular}{llll}
\includegraphics[width=0.5\textwidth]{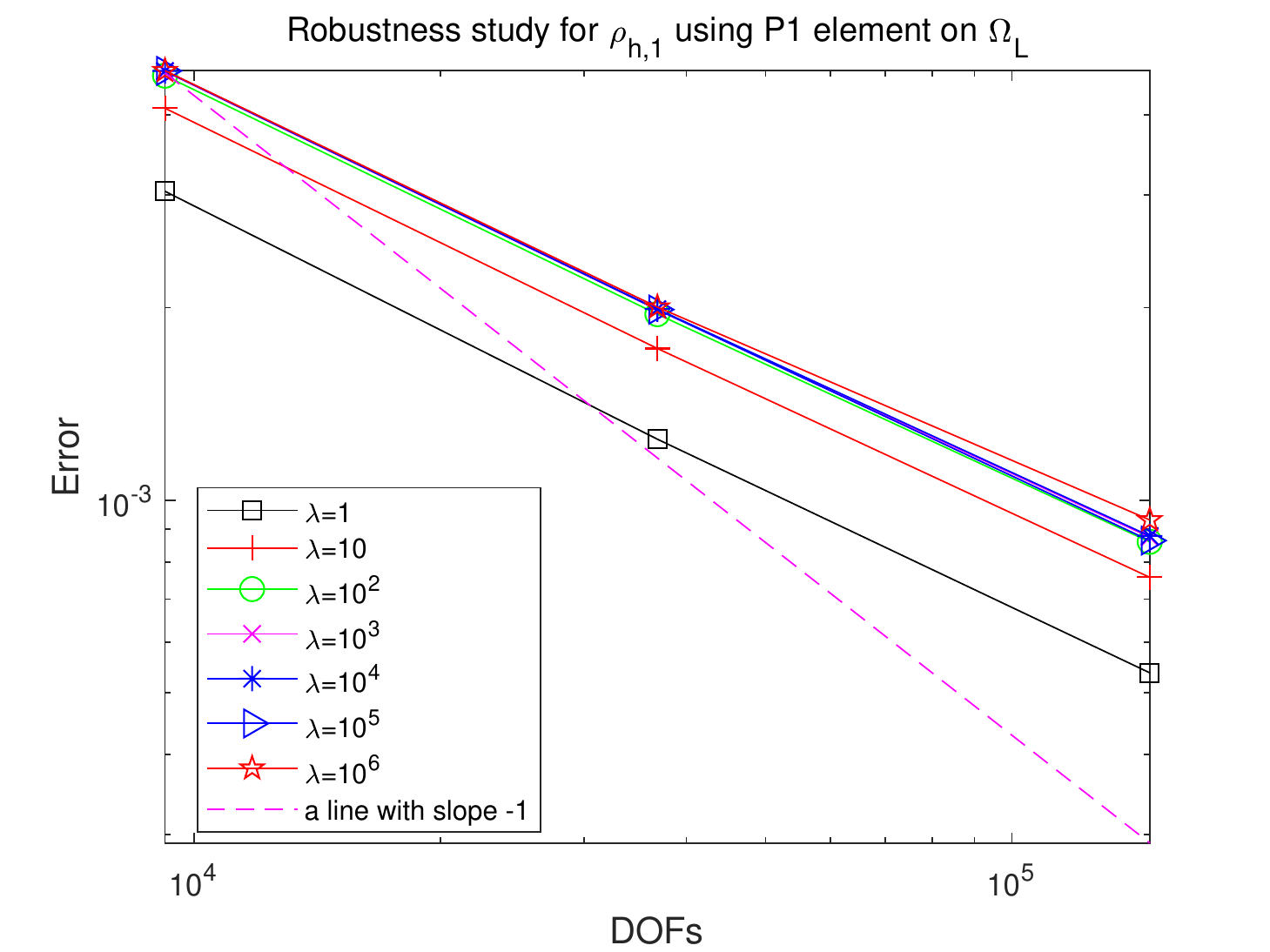}&\includegraphics[width=0.5\textwidth]{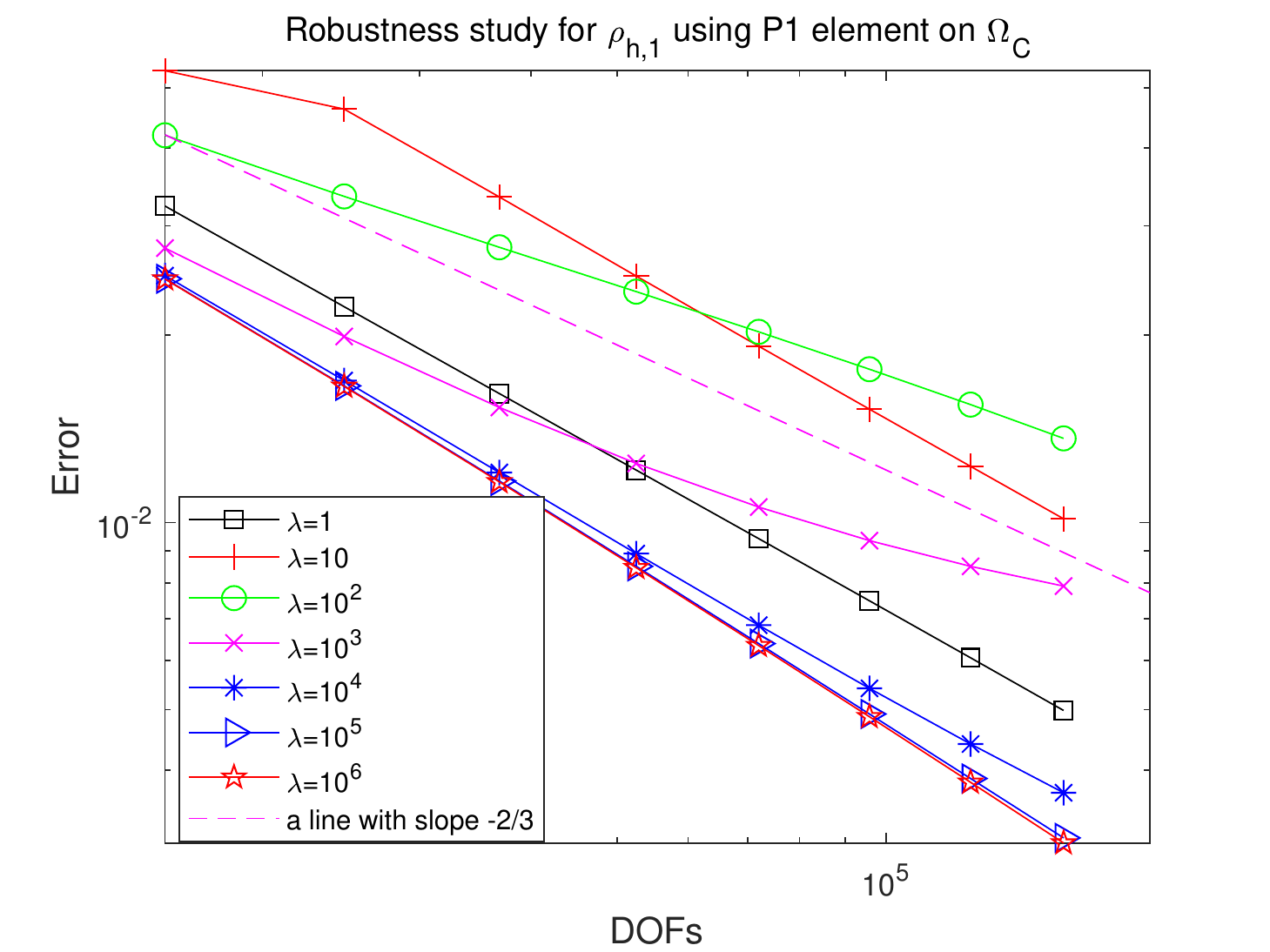}\\
\includegraphics[width=0.5\textwidth]{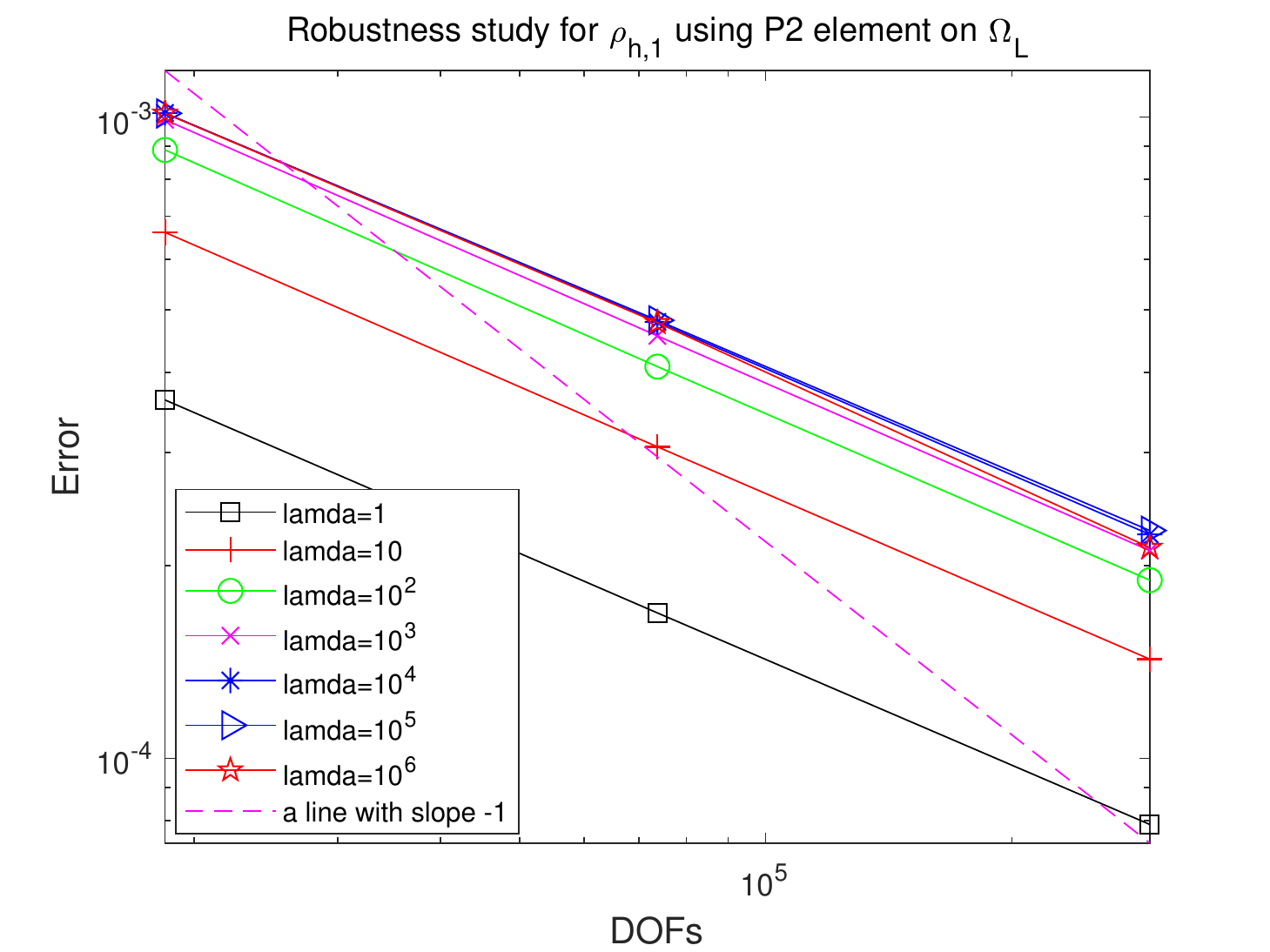}&\includegraphics[width=0.5\textwidth]{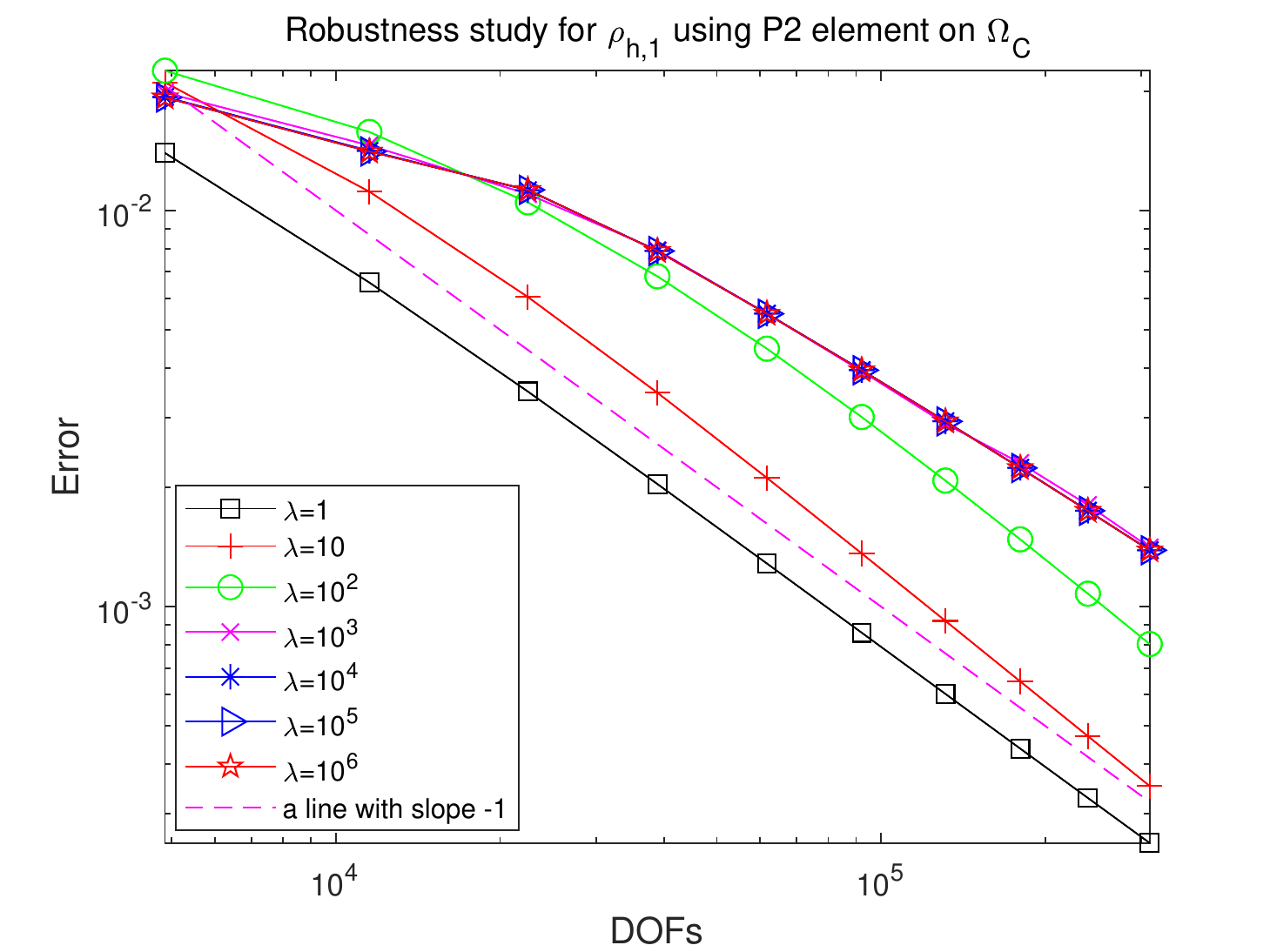}\\
\includegraphics[width=0.5\textwidth]{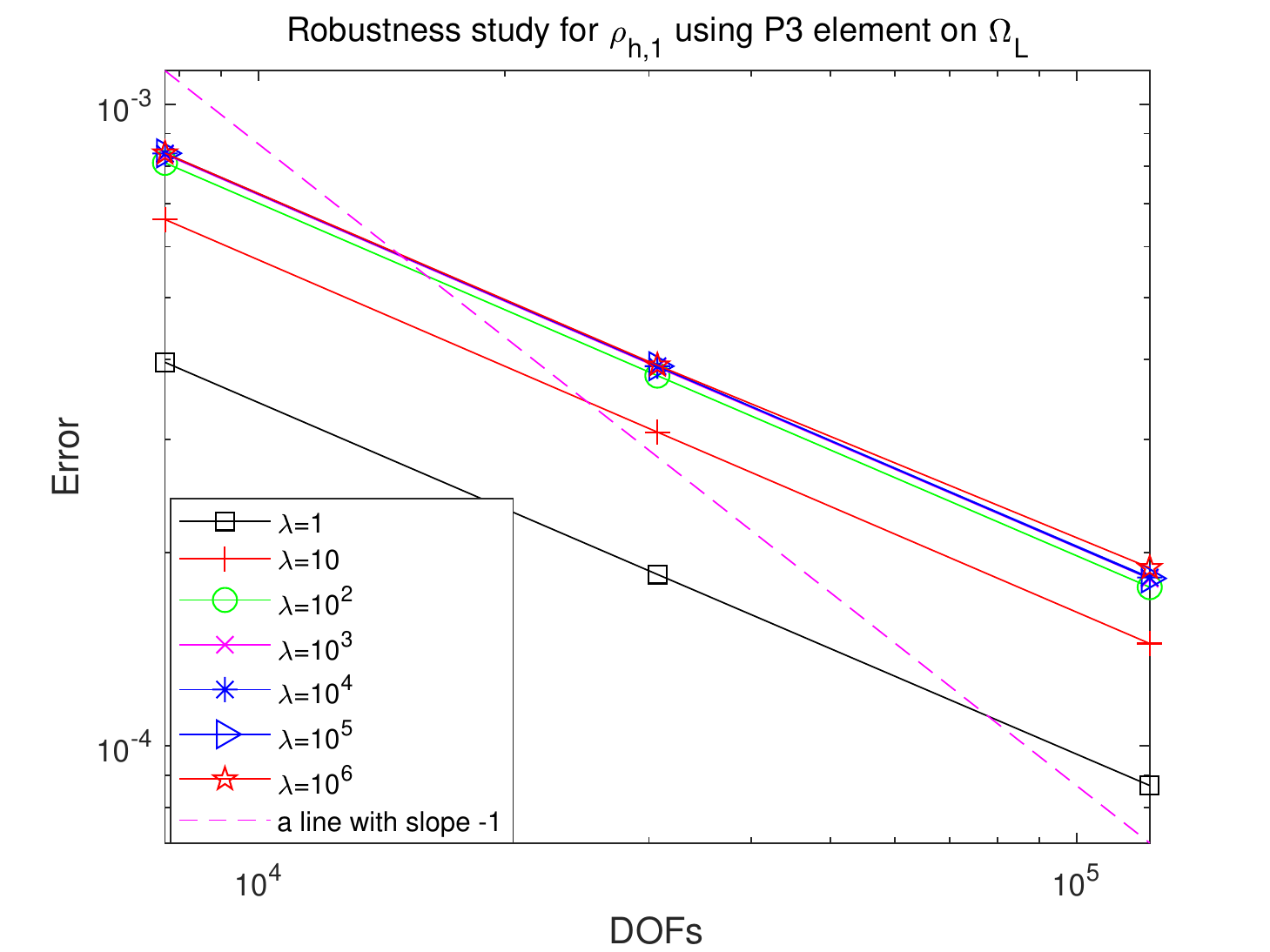}&\includegraphics[width=0.5\textwidth]{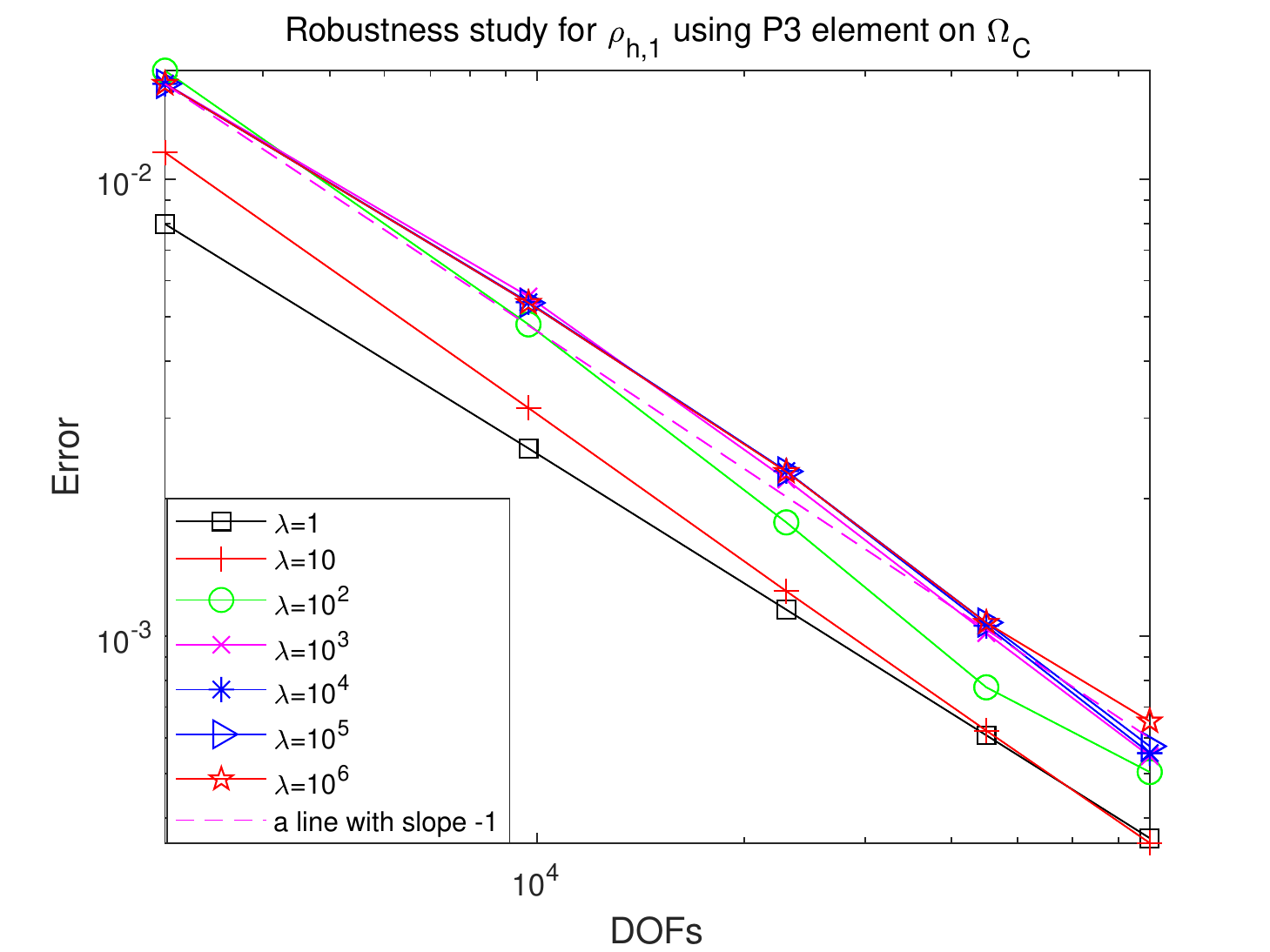}
\end{tabular}
\\{\textrm{\small\bf Figure 4. The robustness study for $\rho_{h,1}$  on $\Omega_{L}$ (left column) and $\Omega_{C}$ (right column)  using the P1, P2, P3 elements.}}
\end{figure}

\begin{figure}
\begin{tabular}{llll}
\includegraphics[width=0.5\textwidth]{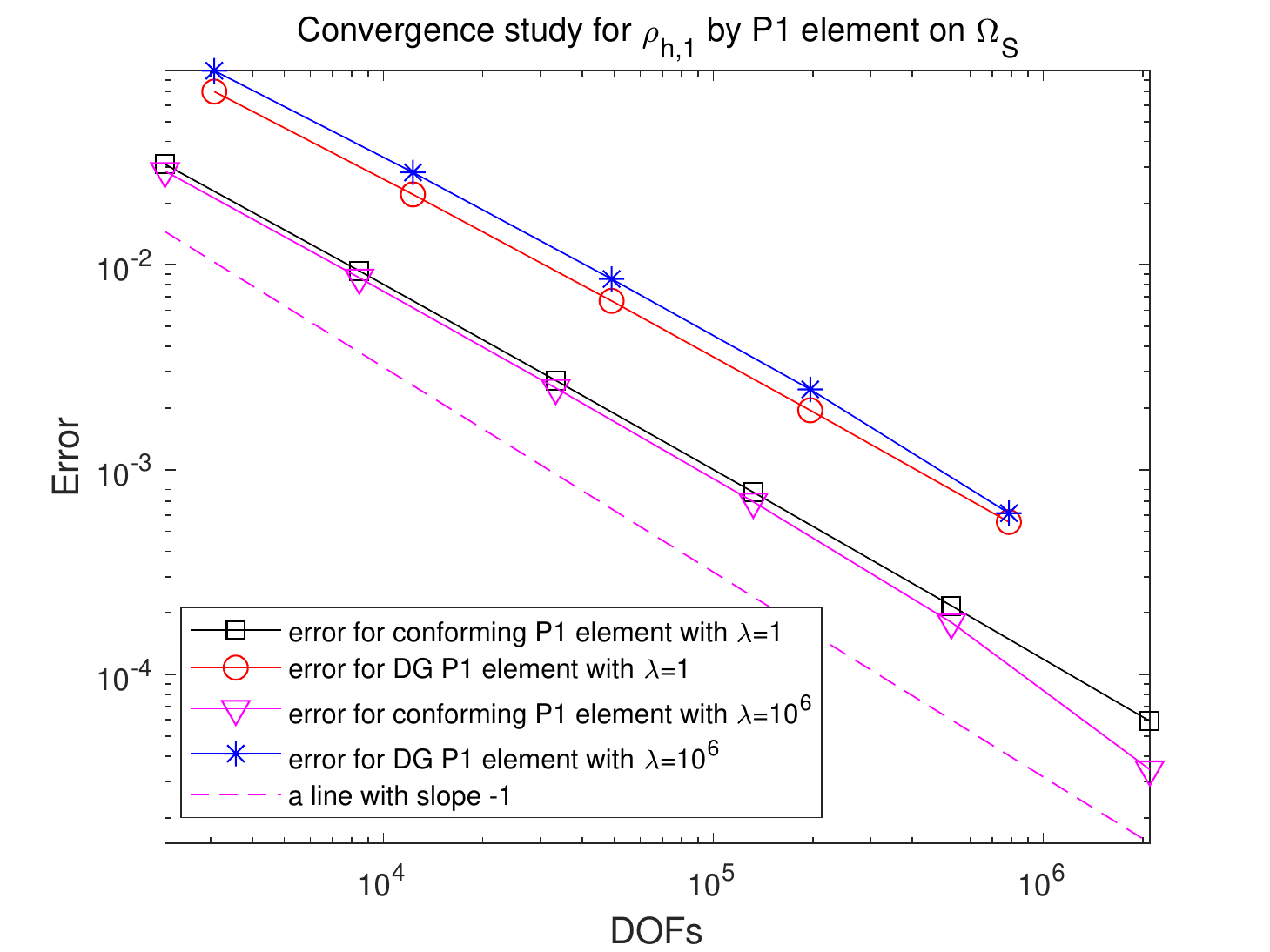}&\includegraphics[width=0.5\textwidth]{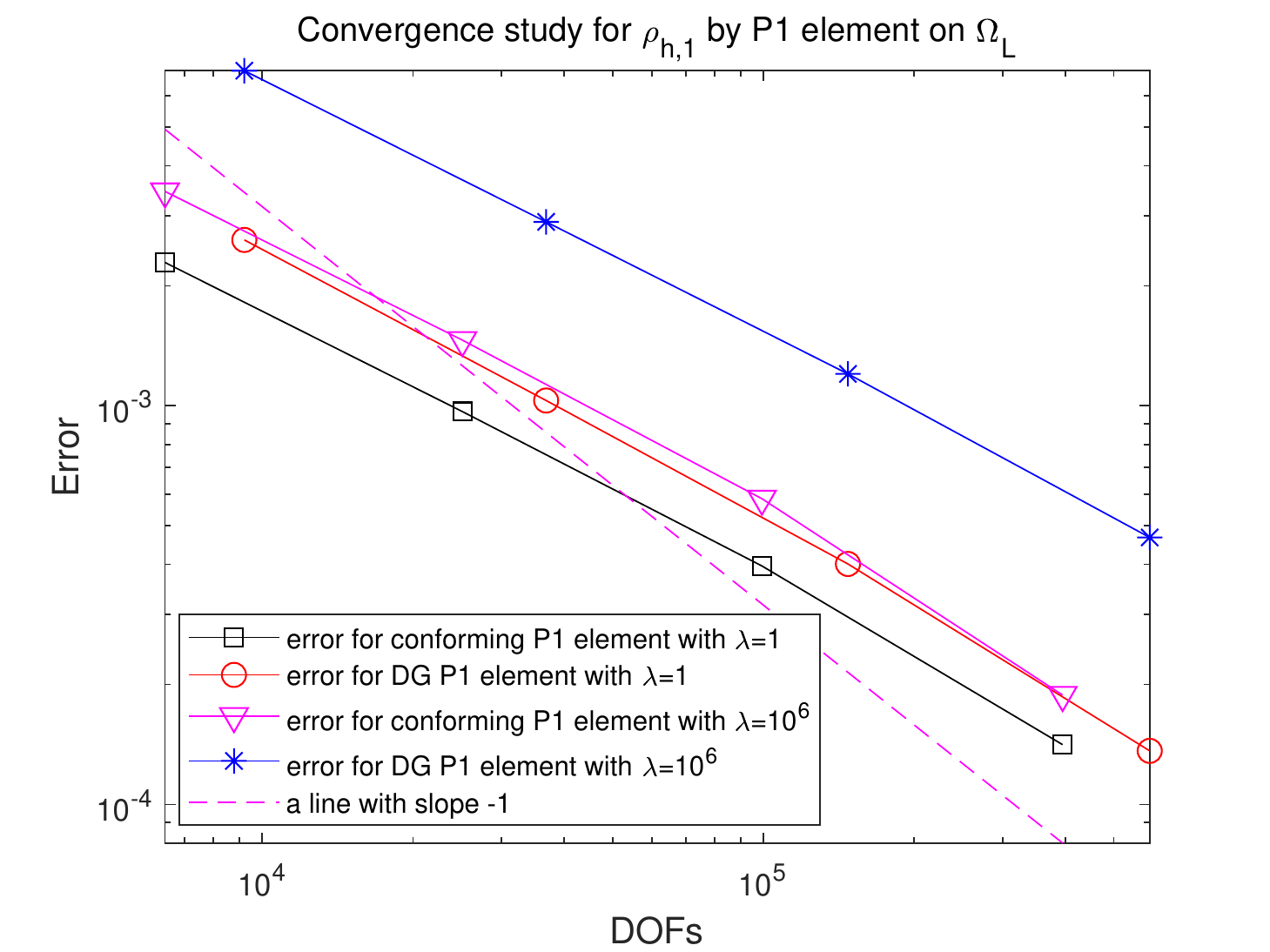}\\
\includegraphics[width=0.5\textwidth]{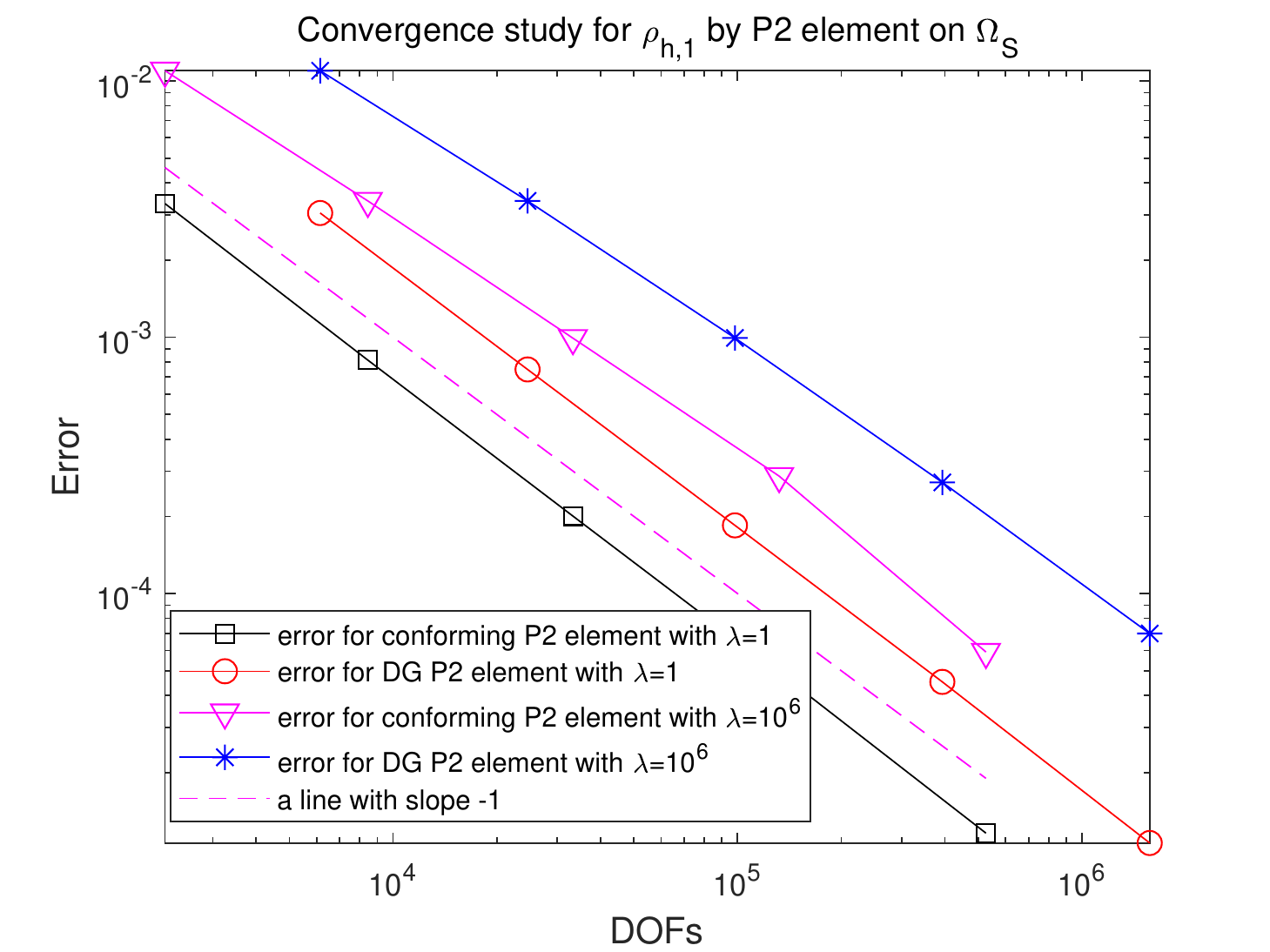}&\includegraphics[width=0.5\textwidth]{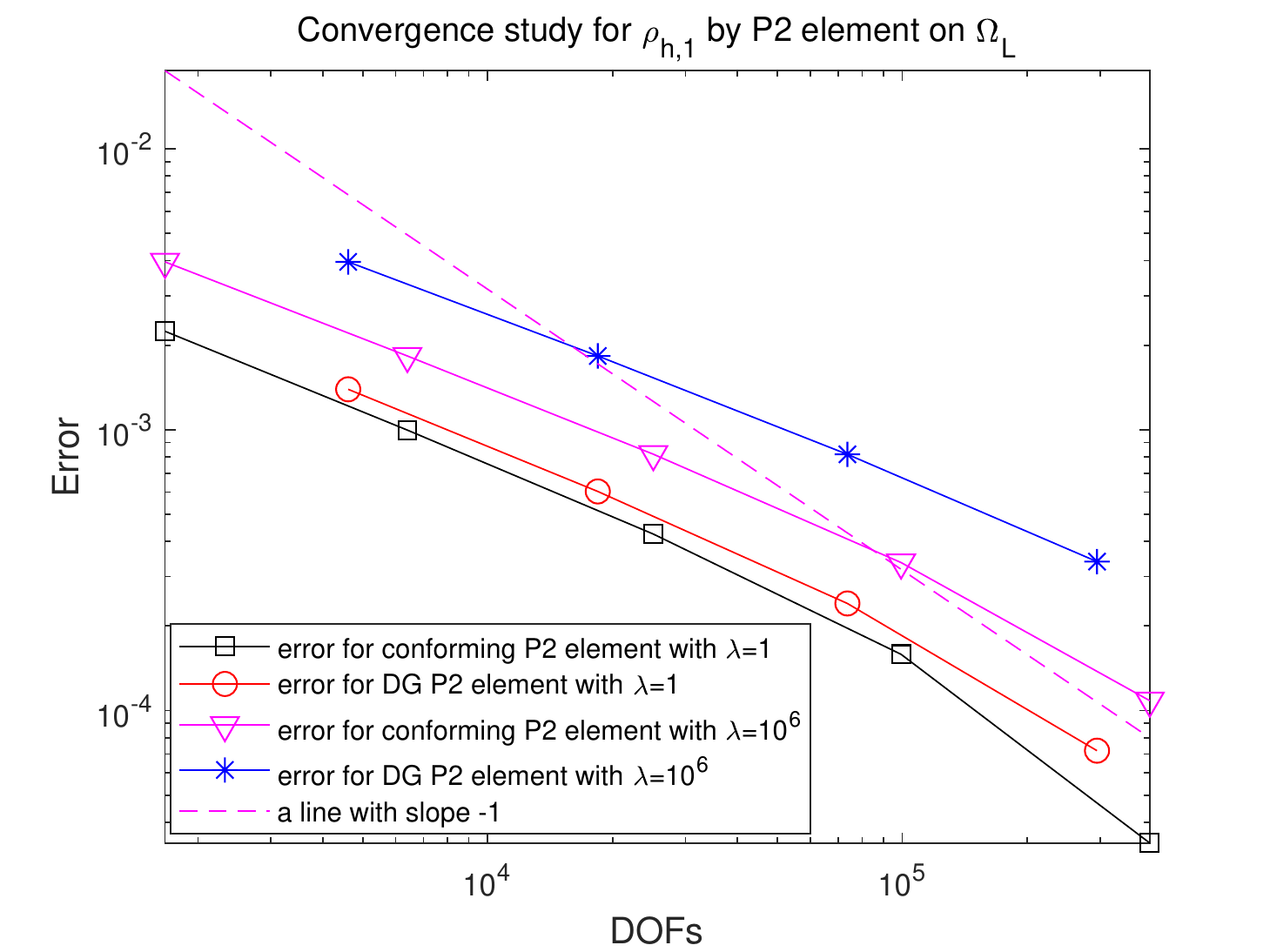}
\end{tabular}
\\{\textrm{\small\bf Figure 5. The convergence study for $\rho_{h,1}$ by conforming and
  DG P1 element with $\lambda=1,10^6$ on $\Omega_{S}$ and $\Omega_{L}$.}}
\end{figure}

\begin{figure}
\begin{tabular}{llll}
\includegraphics[width=0.5\textwidth]{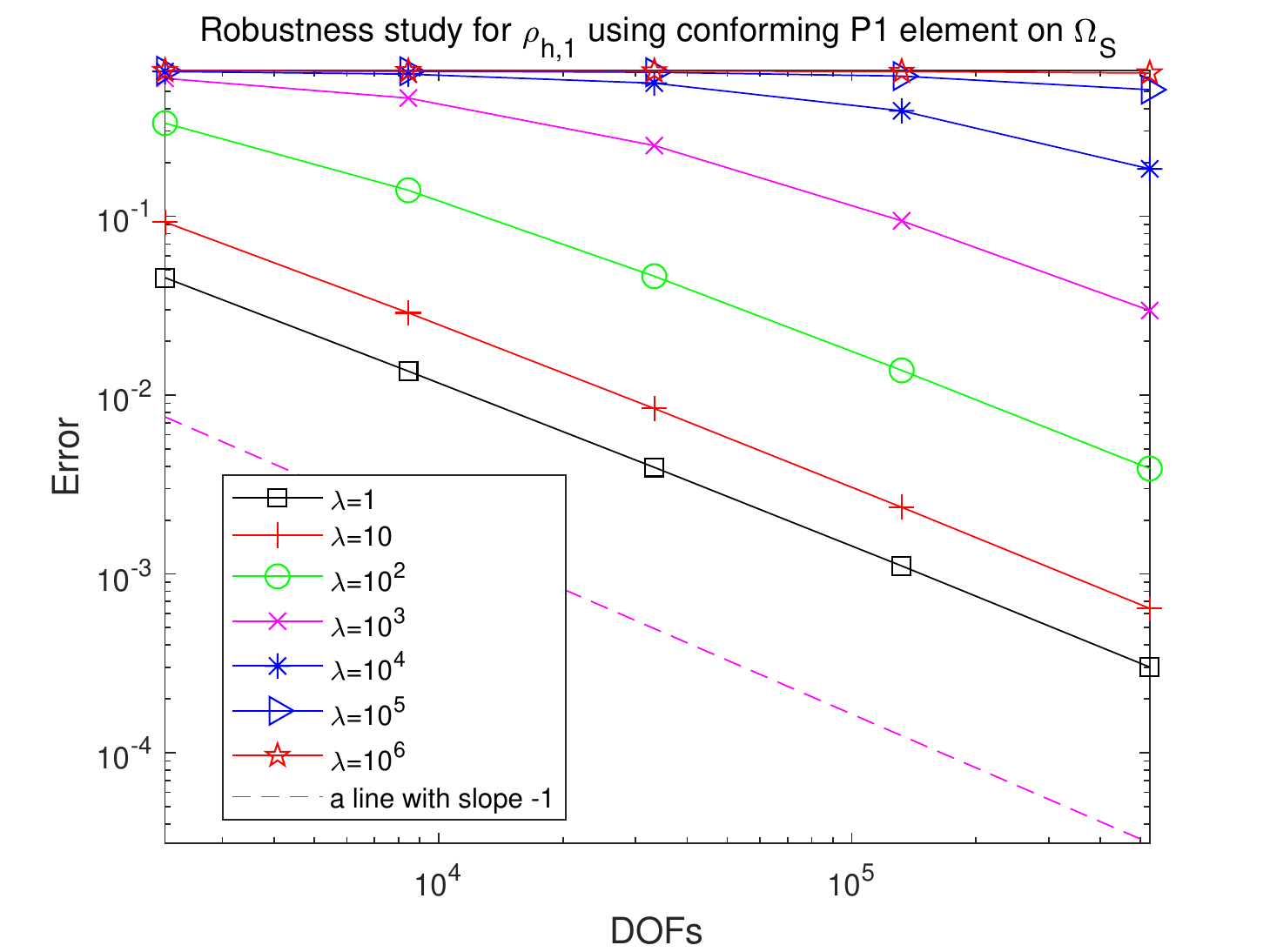}&\includegraphics[width=0.5\textwidth]{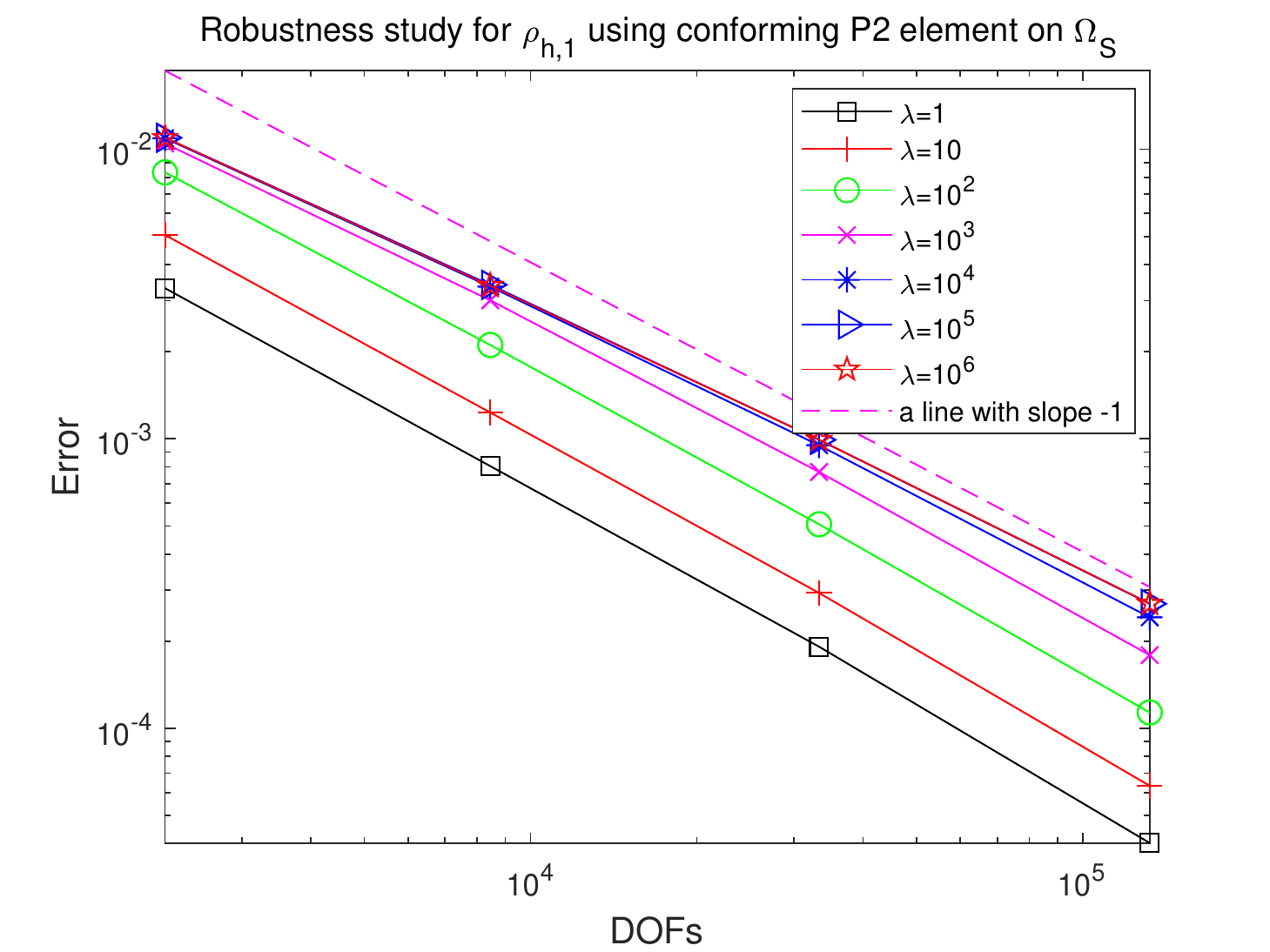}\\
\includegraphics[width=0.5\textwidth]{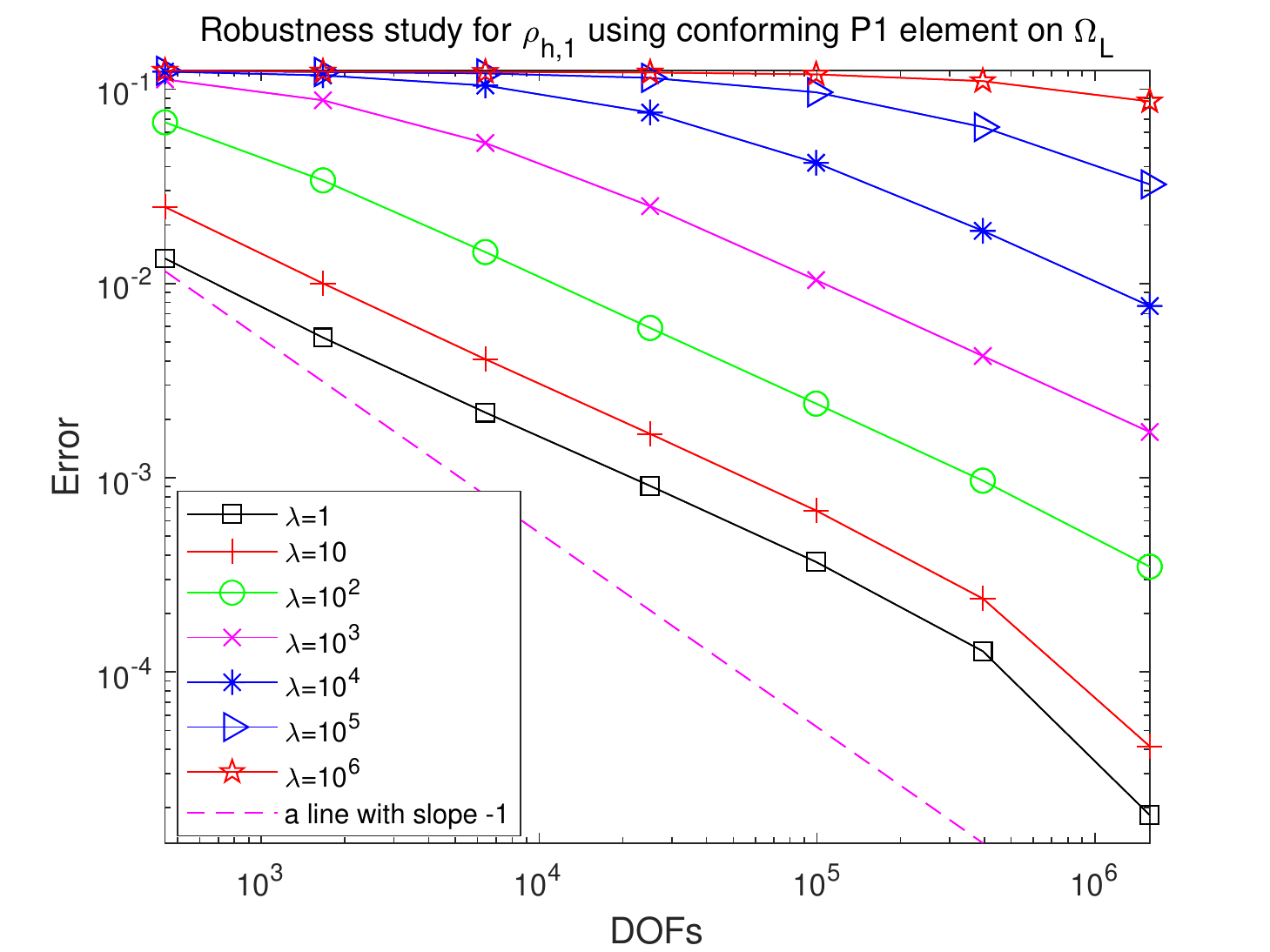}&\includegraphics[width=0.5\textwidth]{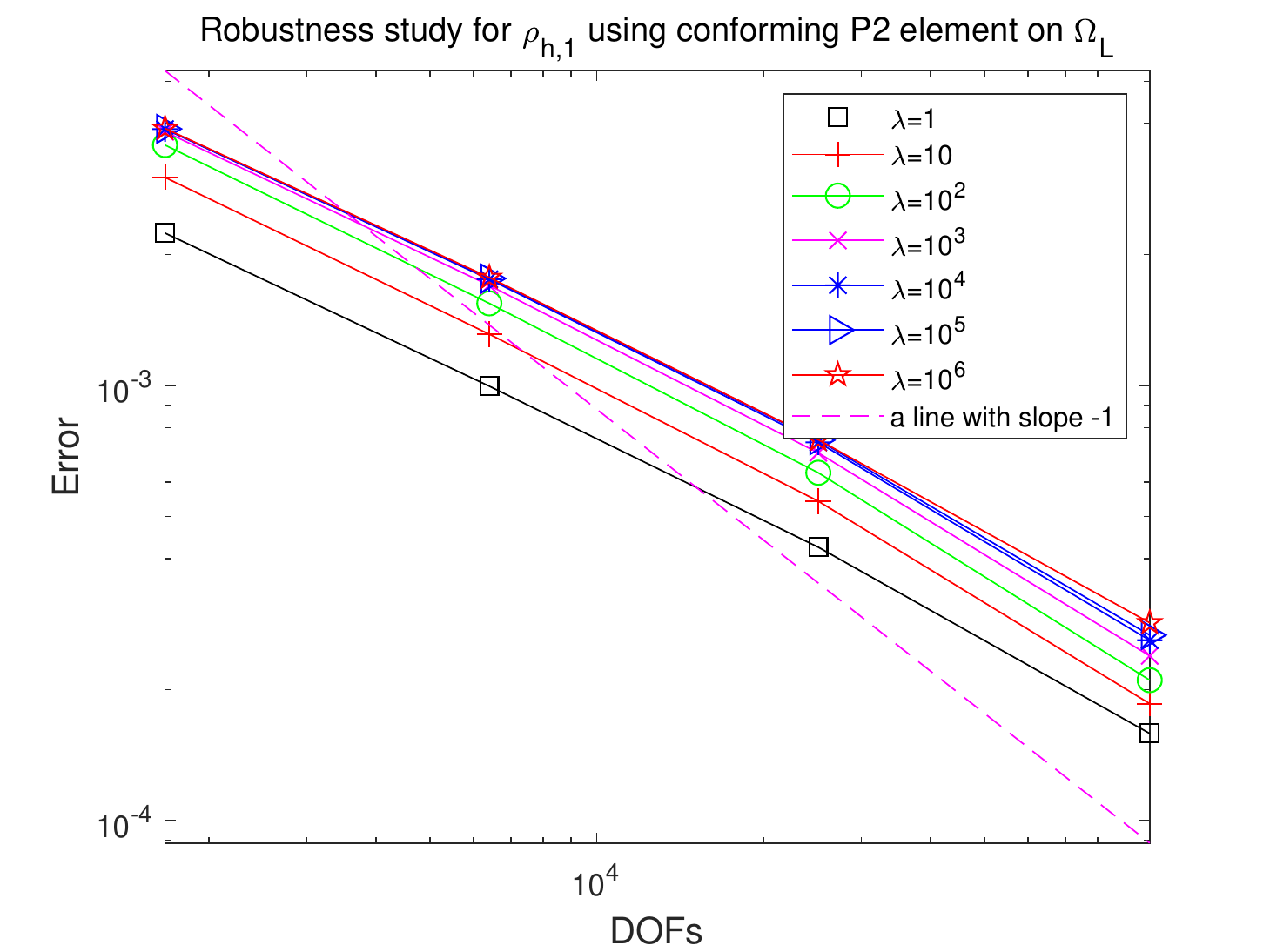}\\
\includegraphics[width=0.5\textwidth]{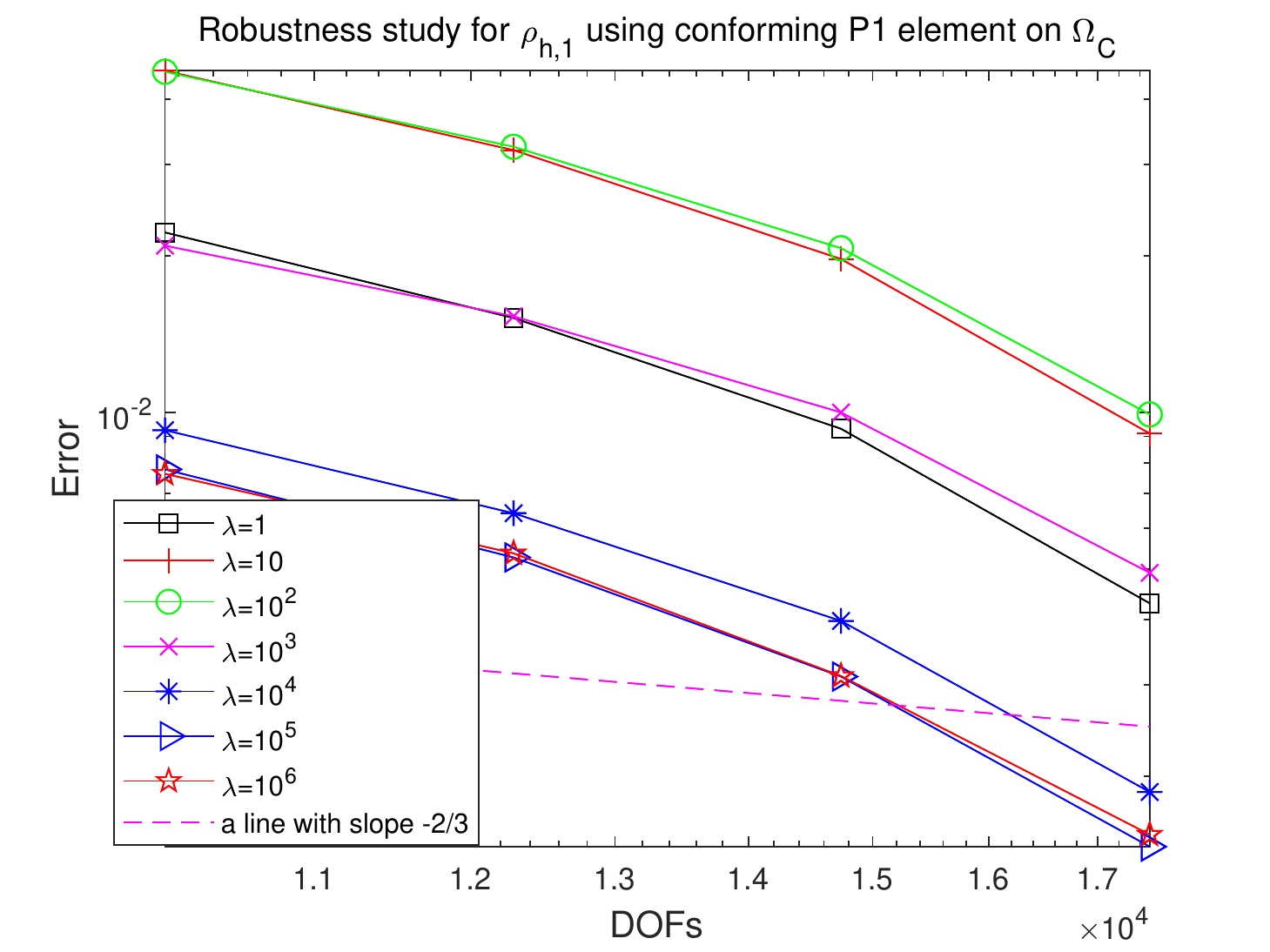}&\includegraphics[width=0.5\textwidth]{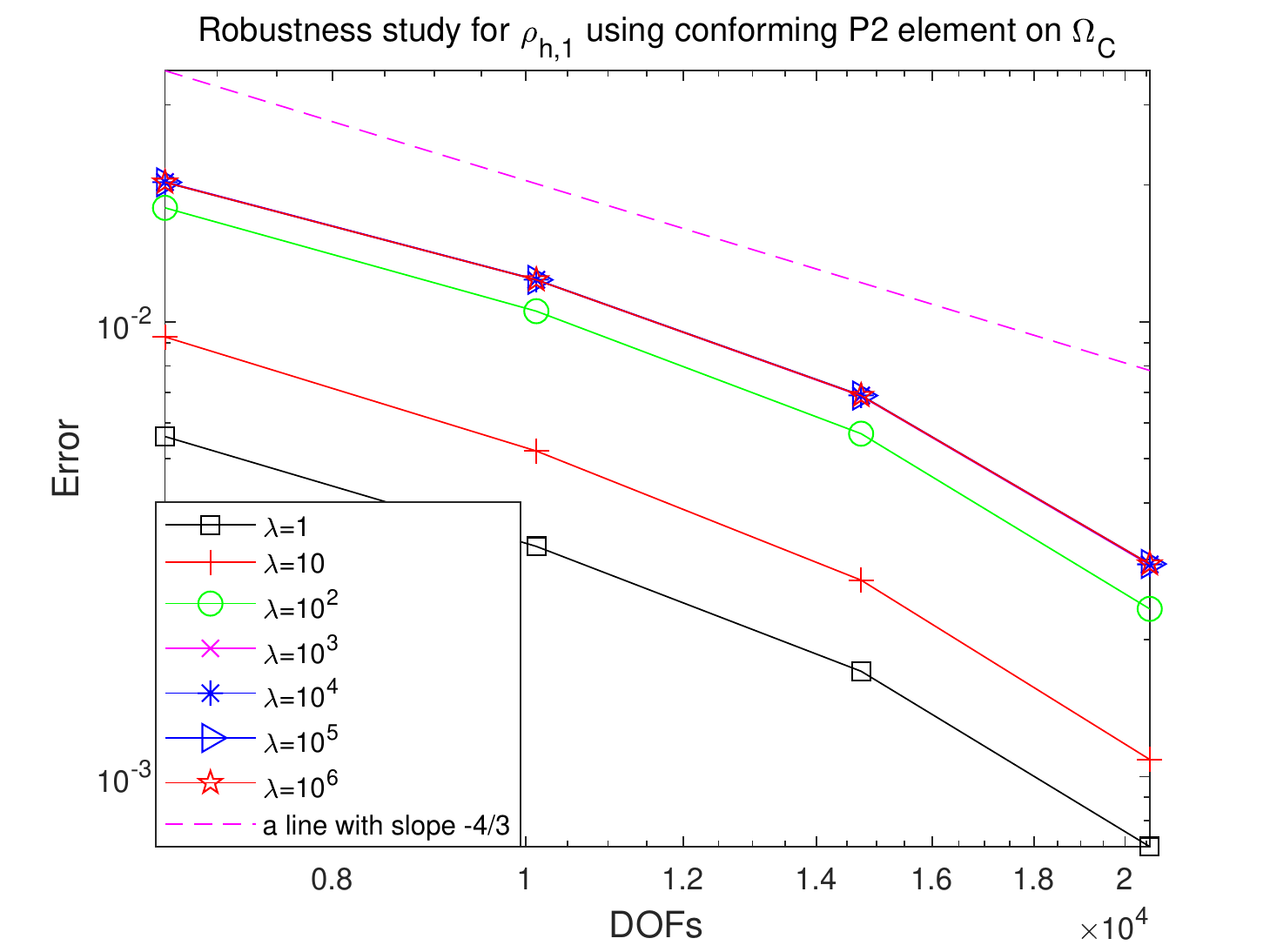}
\end{tabular}
\\{\textrm{\small\bf Figure 6. The robustness study for $\rho_{h,1}$  on domains using conforming P1 element (left column) and conforming P2 element (right column).}}
\end{figure}


%
%



\begin{thebibliography}{s29}

\bibitem{Brenner1992}S.C. Brenner, L.Y. Sung, Linear finite element methods for planar linear elasticity, Math. Comput. 59(1992) 321-338.

\bibitem{Hernandez2009}E. Hern$\acute{a}$ndez, Finite element approximation of the elasticity spectral problem on curved domains,
J. Comput. Appl. Math. 225(2009) 452-458.

\bibitem{Meddahi2013}S. Meddahi, D. Mora, R. Rodr\'{i}guez, Finite element spectral analysis for the mixed formulation of the elasticity equations, SIAM J. Numer. Anal. 51(2013) 1041-1063.

\bibitem{oden}J. Oden, S. Prudhomme, T. Westermann, J. Bass, M.E. Botkin, Error estimation of eigenfrequencies for elasticity and shell problems, Math. Models Methods Appl. Sci. 13(2003) 323-344.

\bibitem{Ovtchinnikov} E.E. Ovtchinnikov, L.S. Xanthis, Effective dimensional reduction algorithm for eigenvalue
problems for thin elastic structures: a paradigm in three dimensions, Proc. Natl. Acad. Sci. USA. 97(2000) 967-971.

\bibitem{Russo}A.D. Russo, Eigenvalue approximation by mixed non-conforming finite element methods: the
determination of the vibrational modes of a linear elastic solid, Calcolo. 51(2014) 563-597.

\bibitem{Walsh}T.F. Walsh, G.M. Reese, U.L. Hetmaniuk, Explicit a posteriori error estimates for eigenvalue
analysis of heterogeneous elastic structures, Comput. Methods Appl. Mech. Engrg. 196(2007) 3614-3623.

\bibitem{Dom¨ªnguez2021}S. Dom\'{\i}nguez, Steklov eigenvalues for the Lam\'{e} operator in linear elasticity, J. Comput. Appl. Math. 394(2021) 113558.


\bibitem{Kouhia}R. Kouhia, R. Stenberg, A linear nonconforming finite element method for nearly incompressible elasticity and Stokes flow, Comput. Methods Appl. Mech. Engrg. 124(3)(1995) 195-212.

\bibitem{Brezzi1991} F. Brezzi, M. Fortin, Mixed and Hybrid Finite Element Methods, Springer, Berlin, 1991.

\bibitem{Chapelle}D. Chapelle, R. Stenberg, Locking-free mixed stabilized finite element methods for bending-dominated shells, Centre de Recherche Mathematiques. CRM Prceedings  Lecture Notes. 21(1999) 81-94.

\bibitem{Vogelius} M. Vogelius, An analysis of the $p$-version of the finite element method for nearly incompressible materials, Numer. Math. 41(1983) 39-53.

\bibitem{Hansbo2002}P. Hansbo, M.G. Larson, Discontinuous Galerkin methods for incompressible and nearly incompressible elasticity by Nitsche's method, Comput. Methods Appl. Mech. Engrg. 191(2002) 1895-1908.

\bibitem{Wihler2004}T.P. Wihler, Locking-free DGFEM for elasticity problems in polygons, IMA J. Numer. Anal. 24(1)(2004) 45-75.
%
\bibitem{Wihler2006}T.P. Wihler, Locking-free adaptive discontinuous Galerkin FEM for linear elasticity problems, Math. Comput. 75(255)(2006) 1087-1102.

\bibitem{Arnold2002}DN. Arnold, F. Brezzi, B. Cockburn,  LD. Marini, Unified analysis of discontinuous Galerkin methods for elliptic problems, SIAM J. Numer. Anal. 39(5)(2002) 1749-1779.

\bibitem{Douglas1976}J. Douglas, T. Dupont, Interior penalty procedures for elliptic and parabolic Galerkin methods, pp. 207-216. Springer, Berlin, 1976.


\bibitem{Rusten1996}T. Rusten, P. Vassilevski, R. Winther, Interior penalty preconditioners for mixed finite element approximations of elliptic problems, Math. Comput. Am. Math. Soc. 65(214)(1996) 447-466.

\bibitem{Riviere2008}B. Rivi\`{e}re, Discontinuous Galerkin Methods for Solving Elliptic and Parabolic Equations, Theory and Implementation. Society for Industrial and Applied Mathematics, 2008.

\bibitem{Wihler2002}T.P. Wihler, Discontinuous Galerkin FEM for Elliptic Problems in Polygonal Domains, PhD thesis, Swiss Federal Institute of Technology Zurich, Diss. ETH No.14973(2002).

\bibitem{Antonietti2022}P.F. Antonietti, M. Botti, I. Mazzieri, S. Nati Poltri, A high-order discontinuous Galerkin method for the poro-elasto-acoustic problem on polygonal and polyhedral grids, SIAM J. Sci. Comput., 44(2022) B1-B28.

\bibitem{Antonietti2018}P.F. Antonietti, I. Mazzieri, High-order discontinuous Galerkin methods for the elastodynamics equation on polygonal and polyhedral meshes, Comput. Methods Appl. Mech. Engrg., 342(2018) 414-437.

\bibitem{Johnson1986}C. Johnson, J. Pitk\"{a}ranta, An analysis of the discontinuous Galerkin method for a scalar hyperbolic equation, Math. Comput. 46(173)(1986) 1-26.

\bibitem{Brezzi2004}F. Brezzi, LD. Marini, E. S\"{u}li, Discontinuous Galerkin methods for first-order hyperbolic problems, Math Models
Methods Appl Sci. 14(12)(2004) 1893-1903.

\bibitem{Johnson1993}C. Johnson, Discontinuous Galerkin finite element methods for second order hyperbolic problems, Comput. Methods Appl. Mech. Eng. 107(1993) 117-129.

\bibitem{Krivodonova2004}L. Krivodonova, J. Xin, JF. Remacle, N. Chevaugeon, JE. Flaherty, Workshop on innovative time integrators for PDEs detection and limiting with discontinuous Galerkin methods for hyperbolic conservation laws, Appl. Numer. Math. 48(3)(2004) 323-338.

\bibitem{Bassi1997}F. Bassi, S. Rebay, A high-order accurate discontinuous finite element method for the numerical solution of the
compressible Navier-Stokes equations, J. Comput. Phys. 131(1997) 267-279.

\bibitem{Pietro2009} D.Di Pietro, A. Ern, Discrete functional analysis tools for discontinuous Galerkin methods with application to the incompressible Navier-Stokes equations, Technical Report, 381(2009) CERMICS.

\bibitem{Cockburn1998b}B. Cockburn, C.W. Shu, The local discontinuous Galerkin method for time-dependent convection-diffusion systems,
SIAM J. Numer. Anal. 35(6)(1998) 2440-2463.

\bibitem{Ern2005}A. Ern, J. Proft, A posteriori discontinuous Galerkin error estimates for transient convection-diffusion equation, Appl. Math. Lett. 18(2005) 833-841.


\bibitem{Pietro2013}D.Di Pietro, S. Nicaise, A locking-free discontinuous Galerkin method for linear elasticity in locally nearly
incompressible heterogeneous media, Appl. Numer. Math. 63(2013) 105-116.

\bibitem{Babuska1991}I. Babu\v{s}ka, J.E. Osborn, Eigenvalue problems. In: Finite Element Methods (Part I). Handbook of Numerical Analysis, vol. 2, pp. 641-787. Elsevier Science Publishers North-Holand, 1991.

\bibitem{Ciarlet1991}P.G. Ciarlet, Basic error estimates for elliptic problems, Handbook of Numerical Analysis, 2(1991)17-351.

\bibitem{Ern2004} A. Ern, J.L. Guermond, Theory and Practice of Finite Elements, Volume 159, Applied Mathematical Sciences, Springer-Verlag, New York, 2004.

\bibitem{Nitche1971}J. Nitsche, $\ddot{U}$ber ein Variationsprinzip zur L$\ddot{o}$sung von Dirichlet-Problemen bei Verwendung von Teilr$\ddot{a}$umen, die keinen Randbedingungen unterworfen sind. Abh. Math. Sem. Univ. Hamburg, 36(1971)9-15.


\bibitem{Oden1998}J. Oden, I. Babu$\breve{s}$ka, C. Baumann, A discontinous hp finite element method for diffusion problems, J. Comp. Phys. 146(1998) 491-591.

\bibitem{Ortner2007}C. Ortner, E. S\"{u}li, Discontinuous Galerkin finite element approximation of nonlinear second-order
elliptic and hyperbolic systems, SIAM J. Numer. Anal. 45(2007) 1370-1397.



\bibitem{bernardi}C. Bernardi, F. Hecht, Error indicators for the mortar finite element discretization of Laplace equation,
Math. Comp. 71(240)(2001) 1371-1403.


\bibitem{Cai2011}Z. Cai, X. Ye, S. Zhang, Discontinuous Galerkin finite element methods for interface problems: a priori and a posteriori error estimations, SIAM J. Numer. Anal. 49(2011) 1761-1787.

\bibitem{Cai2017}Z. Cai, C. He, S. Zhang, Discontinuous finite element methods for interface problems: robust a priori and a posteriori error estimates, SIAM J. Numer. Anal. 55(2017) 400-418.

\bibitem{Bi2021}H. Bi, X. Zhang, Y. Yang, The nonconforming Crouzeix-Raviart element approximation and two-grid discretizations for the elastic eigenvalue problem, J. Comput. Math. http://www.global-sci.org/jcm, doi:10.4208/jcm.2201-m2020-0128.

\bibitem{Grisvard1985}P. Grisvard, Elliptic Problems in Nonsmooth Domains, Monogr. Stud. Math. 24, Pitman, Boston, 1985.


\bibitem{Brezzi1985}F. Brezzi, J. Douglas, D. Marini, Two families of mixed finite elements for second order elliptic problems, Numer. Math. 47(1985)217-235.

\bibitem{Ciarlet2013}P. Jr. Ciarlet, Analysis of the Scott-Zhang interpolation in the fractional order Sobolev spaces, J. Numer. Math.,
21(3) (2013), 173-180.

\bibitem{Brenner2007}S.C. Brenner, L.R. Scott, The Mathematical Theory of Finite Element Methods. 3rd ed. Spinger-Verlag, New York, 2007.


\bibitem{Gustafsson2020}T. Gustafsson, G.D. McBain, scikit-fem: A Python package for finite element assembly, Journal of Open Source Software, 5(2020) 2369.

\bibitem{Harris2020}C.R. Harris, K.J. Millman, S.J. Van Der Walt, R. Gommers, P. Virtanen, D. Cournapeau, E. Wieser, J. Taylor, S. Berg, N.J. Smith, et al., Array programming with NumPy, Nature, 585(2020) 357-362.

\bibitem{Virtanen2020}P. Virtanen, R. Gommers, T.E. Oliphant, M. Haberland, T. Reddy, D. Cournapeau, E. Burovski, P. Peterson, W. Weckesser, J. Bright, et al., SciPy 1.0: fundamental algorithms for scientific computing in Python, Nature methods, 17(2020) 261-272.


\bibitem{Ciarlet1978} P.G. Ciarlet, The Finite Element Method for Elliptic Problems, SIAM, 1978.
%


\end{thebibliography}
\end{document}